\documentclass[a4paper,12pt]{article}
\usepackage[cp1251]{inputenc}
\usepackage[russian]{babel}
\usepackage{amsfonts, amssymb, amsmath, amsthm, amscd}
\usepackage{cite}
\author{A.A. Vasil'eva\footnote{E-mail address: vasilyeva\_nastya@inbox.ru; telephone number: +7 495 939 56 32.}}
\title{Kolmogorov and linear widths of the weighted Besov classes
with singularity at the origin}
\date{}
\begin{document}

\maketitle

\newenvironment{Biblio}{%
                  \renewcommand{\refname}{\footnotesize REFERENCES}%
                  \begin{thebibliography}{99}%
                  \itemsep -0.2em}%
                  {\end{thebibliography}}

\renewcommand{\le}{\leqslant}
\renewcommand{\ge}{\geqslant}
\renewcommand{\tilde}{\widetilde}
\newcommand{\sgn}{\mathrm {sgn}\,}
\newcommand{\dist}{\mathrm {dist}}
\newcommand{\R}{\mathbb{R}}
\renewcommand{\C}{\mathbb{C}}
\newcommand{\Z}{\mathbb{Z}}
\newcommand{\N}{\mathbb{N}}
\newcommand{\Q}{\mathbb{Q}}
\theoremstyle{plain}
\newtheorem{Trm}{Theorem}
\newtheorem{trma}{Theorem}

\newtheorem{Def}{Definition}
\newtheorem{Cor}{Corollary}
\newtheorem{Lem}{Lemma}
\newtheorem{Rem}{Remark}
\newtheorem{Sta}{Proposition}
\renewcommand{\proofname}{\bf Proof}
\renewcommand{\thetrma}{\Alph{trma}}
\renewcommand{\abstractname}{Abstract}
\begin{abstract}
In this paper we obtain asymptotic estimates of Kolmogorov
and linear widths of the weighted Besov classes with singularity
at the origin. In addition, estimates of Kolmogorov and linear
widths of finite-dimensional balls in a mixed norm are obtained.
\end{abstract}

\section{Introduction}
The aim of this paper is to obtain asymptotic behavior of
the Kolmogorov and linear widths of embeddings of the weighted Besov spaces with
weight functions having a strong singularity at the origin.

Denote $\N$, $\Z$, $\Z_+$, $\R$, $\R_+$ the sets of natural,
integer, nonnegative integer, real and nonnegative real numbers,
respectively.

For $1<p<\infty$ put $p'=\frac{p}{p-1}$.
\begin{Def}
Let $w:\R^d\rightarrow (0, \, +\infty)$ be a locally integrable
function, and $1<p<\infty$. Then $w$ belongs to the Muckenhoupt
class ${\cal A}_p$, if there exists a constant $A>0$ such that for
all balls $B\subset \R^d$ the following inequality holds:
\begin{align}
\label{mukenh_cond} \left(\frac{1}{|B|}\int \limits _B w(x)\,
dx\right)^{1/p} \left(\frac{1}{|B|}\int \limits _B w^{-p'/p}(x)\,
dx\right)^{1/p'}\le A.
\end{align}
\end{Def}
Denote ${\cal A}_\infty =\cup _{p>1}{\cal A}_p$.

Let $w:\R^d\rightarrow (0, \, +\infty)$, $0<p<+\infty$, and let
$f:\R^d\rightarrow \R$ be a measurable function. Denote
$$
\| f\| _{L_p(w)}=\left(\int \limits _{\R^d}|f(x)|^pw(x)\,
dx\right)^{1/p},
$$
$\| f\| _{L_\infty(w)}=\| f\| _{L_\infty}$.

We denote by ${\cal S}(\R^d)$ and ${\cal S}'(\R^d)$ the Schwartz
space and its dual, respectively. For $f\in {\cal S}'(\R^d)$
denote ${\cal F}(f)$ the Fourier transform of $f$.

Let $\varphi\in {\cal S}(\R^d)$ be such that
$$
{\rm supp}\, \varphi \subset \{y\in \R^d:|y|<2\}, \;\;\;
\varphi(x)=1 \text{ for }|x|\le 1,
$$
$\varphi _0=\varphi$, $\varphi _j(x)=\varphi(2^{-j}x)
-\varphi(2^{-j+1}x)$, $j\in \N$.
\begin{Def}
Let $0<p\le \infty$, $0<q\le \infty$, $w\in {\cal A}_\infty$. The
weighted Besov space $B^s_{p,q}(\R^d, \, w)$ is the set of all
distributions $f\in {\cal S}'(\R^d)$ such that
$$
\| f\| _{B^s_{p,q}(\R^d, \, w)}:=\left(\sum \limits _{j=0}^\infty
2^{jsq}\|{\cal F}^{-1}(\varphi _j{\cal F}f)\|
^q_{L_p(w)}\right)^{1/q}
$$
is finite. In the case $q=\infty$ the usual modification is
required.
\end{Def}
The space $B^s_{p,q}(\R^d, \, w)$ is quasi-Banach, and if $p\ge 1$
and $q\ge 1$, then it is a Banach space.

Properties of the Besov spaces with the weight $w\equiv 1$ can be
found in monographs of Triebel \cite{tr1, tr2, tr3, tr4, tr5}. In
\cite{edm_tr, har_envel} the Besov spaces with admissible weights were
considered. Examples of such weights are
$w(x)=(1+\|x\|_2^2)^{\alpha/2}$, where $\|\cdot\|_2$ is the
standard Euclidean norm on $\R^d$. The case $w\in {\cal A}_\infty$
first was studied by Bui \cite{bui}, more general weights were
considered by Bownik, Frazier and Roudenko \cite{bow1, fr_roud}.
In recent papers of Haroske, Skrypczak, Piotrowska and Schneider
\cite{har_piot, har_schn, haroske} different characteristics of
the Besov spaces with Muckenhoupt weights were studied (for example,
atomic decompositions, wavelet characterizations, and growth
envelopes).

Let $X$ be a normed space, $n\in \Z_+$, let ${\cal L}_n(X)$ be a
set of subspaces in $X$ of dimension at most $n$. Denote $L(X, \,
Y)$ the space of linear continuous operators from $X$ into a
normed space $Y$, by ${\rm rk}\, A$ we denote the dimension of
image of an operator $A:X\rightarrow Y$, and by $\| A\|
_{X\rightarrow Y}$ its norm. By the Kolmogorov $n$-width of a set
$M\subset X$ in the space $X$ we mean the quantity
$$
d_n(M, \, X)=\inf _{L\in {\cal L}_n(X)} \sup_{x\in M}\inf_{y\in
L}\|x-y\|_X,
$$
and by linear $n$-width
we mean the quantity $$\lambda _n(M, \, X) =\inf _{A\in L(X, \, X), \,
{\rm rk} A\le n}\sup _{x\in M}\| x-Ax\| _X.$$

The approximation numbers of an operator $A\in L(X, \, Y)$ are defined
by
$$
{\cal A}_n(A:X\rightarrow Y)=\inf \{\| A-A_n\| _{X\rightarrow Y}:{\rm rk}\,
A_n\le n\}.
$$
If $M\subset X$ is the unit ball, then we denote $d_n(A(M), \,
Y)=: d_n(A:X\rightarrow Y)$ and $\lambda_n(A(M), \, Y)=:
\lambda_n(A:X\rightarrow Y)$. If ${\rm Id}$ is a compact operator
of embedding of $X$ into $Y$, then it follows from Heinrich's
results in \cite{heinr} (see Theorem 3.1) that
\begin{align}
\label{aneqln} {\cal A}_n({\rm Id}:X\rightarrow Y)=\lambda _n({\rm Id}:
X\rightarrow Y).
\end{align}

Let $1\le p\le \infty$. Denote $l_p^n$ the linear space $\R^n$ with the norm
$$\|(x_1, \, \dots , \, x_n)\| _p\equiv\|(x_1, \, \dots , \,
x_n)\| _{l_p^n}= \left\{
\begin{array}{l}(| x_1 | ^p+\dots+ | x_n | ^p)^{1/p},\text{ if
}p<\infty ,\\ \max \{| x_1 | , \, \dots, \, | x_n |\},\text{ if
}p=\infty ,\end{array}\right .$$ by $B_p^n$ denote the unit ball in
$l_p^n$.

In the sixties and seventies of the 20th century, problems on the values
of the widths were studied for function classes in
$L_q$ (see \cite{bibl6,
tikh_babaj, busl_tikh, bib_ismag, bib_kashin, bib_majorov,
bib_makovoz, bibl9, bibl10, bibl11, bibl12, bibl13, kashin1,
kulanin}, and also \cite{tikh_nvtp, itogi_nt, kniga_pinkusa})
and for the finite-dimensional balls $B_p^n$ in
$l_q^n$. For $p\ge q$, Pietsch \cite{pietsch1} and Stesin
\cite{stesin} found exact values of $d_n(B_p^\nu, \, l_q^\nu)$
and $\lambda_n(B_p^\nu, \, l_q^\nu)$. In the case of $p<q$, Kashin
\cite{bib_kashin}, Gluskin \cite{gluskin1,
bib_gluskin}, and Garnaev and Gluskin \cite{garn_glus} determined order values
of the widths of finite-dimensional balls up to quantities depending on
$p$ and $q$ only.

Order estimates for the widths of nonweighted Sobolev classes
were obtained by Tikhomirov, Ismagilov, Makovoz, Kashin, Temlyakov,
Galeev and Kulanin \cite{bibl6,
tikh_babaj, bib_ismag, bib_kashin, bib_makovoz, bibl9, bibl10,
bibl11, bibl12, bibl13, kulanin}.

Estimates for Kolmogorov and linear widths of weighted Sobolev classes
on smooth bounded domains with weights of the form $w(x)=\bigl[\operatorname{dist}(x, \,
\partial \Omega)\bigr]^\alpha$ were obtained by Triebel \cite{tr1}
(here $\operatorname{dist}(x, \,
\partial \Omega)$ is a distance from $x$ to $\partial \Omega$).
Estimates for linear widths of weighted Sobolev classes on~$\R^d$
with weights of the form $w_\alpha(x)=(1+\|x\|_2^2)^{\alpha/2}$
were established by Mynbaev and Otelbaev in~\cite{myn_otel}.

Denote by $|\cdot|$ an arbitrary norm on $\R^d$.

Problems on estimates of the entropy and approximation numbers of
embeddings of weighted Besov spaces were studied in papers of
Triebel, Haroske, Skrzypczak, K\"{u}hn, Leopold, Sickel, Caetano
\cite{haroske, har_tri, har_tr_wav, skr, kuhn1, kuhn2, kuhn3,
kuhn4, kuhn_leopold, kuhn5, haroske2, haroske3, caetano};
Kolmogorov and Gelfand widths were considered by G\c{a}siorowska,
Skrzypczak, Shun Zhang, Gensun Fang, Fanglun Huang \cite{gasior,
zhang1, zhang2}, and Weyl numbers were studied in \cite{gasior}.
In addition, asymptotic behavior of Kolmogorov, Gelfand and linear
widths of non-weighted Besov spaces on a bounded domain were found
in the paper of Vybiral \cite{vybiral}. In \cite{har_tr_wav, skr,
kuhn4, caetano, zhang1, zhang2} weights of form
$w_\alpha(x)=(1+\|x\|_2^2)^{\alpha/2}$ were considered. In
\cite{haroske} asymptotic estimates of the approximation and
entropy numbers of embedding operator ${\rm Id}:B^{s_1}_{p_1, \,
q_1}(\R^d, \, w)\rightarrow B^{s_2}_{p_2, \, q_2}(\R^d)$ were
obtained for
$$
w(x)=\left\{ \begin{array}{l} |x|^\alpha, \;\; |x|\le 1, \\
|x|^\beta, \;\; |x|>1.\end{array}\right.
$$
In \cite{haroske2} there was found a sufficient condition on a weight
$w$ (in terms of Muckenhoupt condition), under which the approximation numbers
of embedding operator ${\rm Id}:B^{s_1}_{p_1, \, q_1}(Q, \,
w)\rightarrow B^{s_2}_{p_2, \, q_2}(Q)$ have the same asymptotic behavior as
in the non-weighted case. Here $Q\subset \R^d$ is a cube,
$$
B^s_{p,q}(Q, \, w)=\{f|_Q: \, f\in B^s_{p,q}(\R^d, \, w)\}.
$$
In \cite{haroske3} for the weight function
$$
w(x)=\left\{ \begin{array}{l} |x|^{\alpha_1}(1-\log  |x|)^{\beta_1}, \;\; |x|\le 1, \\
|x|^{\alpha_2}(1+\log |x|)^{\beta_2}, \;\; |x|>1\end{array}\right .
$$
asymptotic behavior of the entropy numbers of the operator ${\rm
Id}:B^{s_1}_{p_1, \, q_1}(\R^d, \, w)\rightarrow B^{s_2}_{p_2, \,
q_2}(\R^d)$ were obtained under some conditions on parameters.
Here the parameters $\alpha _1$ and $\beta _1$ did not affect the
asymptotic. In \cite{kuhn5, gasior} weights of logarithmic type
were considered. Examples of such weights are continuous positive
functions such that $c_1(\log |x|)^\alpha\le w(x)\le c_2(\log
|x|)^\alpha$ for sufficiently large $|x|$ ($0<c_1\le
c_2<\infty$).

Similar problems for radial Besov spaces were studied in \cite{kuhn_leop2, gasior1}.

Let us formulate the main result of this paper. We first give
necessary notations and formulate the main theorem, and then compare
this result with those previously known.

We denote $\log _2 x=\frac{\log x}{\log 2}$.

Let $X$, $Y$ be sets, and $f_1$, $f_2:X\times Y\rightarrow \R_+$.
We denote $f_1(x, \, y)\underset{y}{\lesssim} f_2(x, \, y)$
($f_2(x, \, y)\underset{y}{\gtrsim} f_1(x, \, y)$), or $f_1(x, \,
y)\overset{x}{\lesssim} f_2(x, \, y)$ ($f_2(x, \,
y)\overset{x}{\gtrsim} f_1(x, \, y)$), if there exists a positive
function $c:Y\rightarrow \R$ such that $f_1(x, \, y)\le c(y)f_2(x,
\, y)$ for any $x\in X$. Denote $f_1(x, \, y)\underset{y}{\asymp}
f_2(x, \, y)$ (or $f_1(x, \, y)\overset{x}{\asymp} f_2(x, \, y)$),
if $f_1(x, \, y) \underset{y}{\lesssim} f_2(x, \,
y)\underset{y}{\lesssim} f_1(x, \, y)$.

For $s_1$, $s_2\in \R$, $1<p_1, \, q_1\le +\infty$, $1\le p_2, \,
q_2<+\infty$ we put
$\delta:=s_1-s_2+\frac{d}{p_2}-\frac{d}{p_1}$. We shall consider
the case $1<p_1\le p_2<\infty$, $1<q_1\le q_2<\infty$.

Following \cite{vas_alg_an}, define functions $g$ and $v$.

Let $\delta>0$, $\beta_g$, $\beta_v\in \R$, $\beta _g+\beta_v =\delta$,
$\beta_g >-\frac{d}{p_1}$, $\beta_v <\frac{d}{p_2}$, $\gamma
_g>-\frac{d}{p_1}$, $\gamma _v<\frac{d}{p_2}$,
$\gamma_g+\gamma_v>\delta$, $\alpha _g+\alpha _v>0$, $\rho _g$,
$\rho _v:[1, \, +\infty)\rightarrow (0, \, +\infty)$ be absolute
continuous functions such that
\begin{align}
\label{malo_izm} \lim \limits _{y\rightarrow +\infty}\frac{y\rho
'_g(y)}{\rho _g(y)}=0, \;\;\lim \limits _{y\rightarrow
+\infty}\frac{y\rho '_v(y)}{\rho _v(y)}=0,
\end{align}
$$
g(x)=\left\{ \begin{array}{l}
|x|^{-\beta_g}|\log _2|x||^{-\alpha _g}\rho _g(|\log _2|x||),
\text{ if } |x|\le \frac 12,\\
|x|^{-\gamma_g}, \text{ if }|x|>\frac 12,
\end{array}
\right.
$$
$$
v(x)=\left\{ \begin{array}{l} |x|^{-\beta_v}|\log _2|x||^{-\alpha
_v}\rho _v(|\log _2|x||), \text{ if } |x|\le \frac 12, \\
|x|^{-\gamma_v}, \text{ if }|x|>\frac 12.\end{array}\right .
$$

Let $w_1(x)=g^{-p_1}(x)$, $w_2(x)=v^{p_2}(x)$. Then $w_1$, $w_2\in
{\cal A}_\infty$. Indeed, there exist $\underline{c}>0$,
$\overline{c}>0$ such that for any $0<r\le t\le 2r$ we have
$\underline{c}\le\frac{w_i(r)}{w_i(t)}\le \overline{c}$, $i=1, \,
2$. Hence, it is sufficient to consider balls $B$ centered at the
origin. If $0<r<\frac 12$ is the radius of the ball $B$, then
\begin{align}
\label{ilbw1ovrsrrbp1d} \int \limits _B w_1(x)\, dx=\int \limits
_B |x|^{p_1\beta_g}|\log _2|x||^{p_1\alpha _g} \rho
_g^{-p_1}(|\log _2|x||)\, dx\overset{r}{\asymp}r^{\beta_g p_1+d}
|\log _2 r|^{\alpha _gp_1}\rho _g(|\log _2 r|)^{-p_1},
\end{align}
\begin{align}
\label{ilbw2ovrsrrbp2d} \int \limits _B w_2(x)\, dx=\int \limits
_B |x|^{-p_2\beta_v}|\log _2|x||^{-p_2\alpha _v} \rho
_v^{p_2}(|\log _2|x||)\, dx\overset{r}{\asymp}r^{-\beta_v p_2+d}
|\log _2 r|^{-\alpha _vp_2}\rho _v(|\log _2 r|)^{p_2},
\end{align}
and for sufficiently large $\theta >1$ we get
$$
\int \limits _B w_1^{-\theta '/\theta}(x)\, dx\overset{r}{\lesssim}
r^{-\frac{\beta_g p_1\theta '}{\theta}+d}|\log _2 r|
^{-\frac{\alpha _gp_1\theta '}{\theta}}\rho _g(|\log _2 r|)
^{\frac{p_1\theta '}{\theta}},
$$
$$
\int \limits _B w_2^{-\theta '/\theta}(x)\,
dx\overset{r}{\lesssim} r^{\frac{\beta_v p_2\theta
'}{\theta}+d}|\log _2 r| ^{\frac{\alpha _vp_2\theta '}{\theta}}\rho
_v(|\log _2 r|) ^{-\frac{p_2\theta '}{\theta}}.
$$
Therefore,
$$
\left(\frac{1}{|B|}\int \limits _B w_i(x)\, dx\right)^{1/\theta}
\left(\frac{1}{|B|}\int \limits _B w_i^{-\theta '/\theta}(x)\,
dx\right)^{1/\theta '} \overset{r}{\lesssim}1.
$$
Similar estimates can be obtained for $r\ge \frac 12$. Hence,
$w_i\in {\cal A}_\theta$, $i=1, \, 2$.

Denote $\alpha =\alpha _g+\alpha _v$, $\rho(y)=\rho _g(y)\rho
_v(y)$. Then $\alpha>0$,
\begin{align}
\label{malo_izm331}
\lim \limits _{y\rightarrow +\infty}\frac{y\rho '(y)}{\rho (y)}=0.
\end{align}
For $1\le p_1\le p_2\le \infty$ such that $p_2\ge 2$ we denote
${\bf p}=(p_1, \, p_2)$,
\begin{align}
\label{lamb_p} \lambda({\bf p})=\min
\left\{\frac{\frac{1}{p_1}-\frac{1}{p_2}}{\frac 12
-\frac{1}{p_2}}, \, 1\right\}
\end{align}
(if $p_2=2$, we put $\lambda({\bf p})=1$). Let us remark that if
$p_2>2$, then
\begin{align} \label{lamp1p2} \lambda({\bf
p})<1\;\;\Leftrightarrow p_1>2.
\end{align}

Denote
$$
J_*=J_*(p_1, \, q_1, \, p_2, \, q_2)=\left\{ \begin{array}{l} \{1,
\, 2, \, 3, \, 4\}, \text{ if }p_2>2, \; q_2>2,
\\ \{1, \, 2, \, 3\}, \text{ if }p_2>2, \; q_2=2, \\ \{2, \, 3, \, 4\}, \text{ if }
p_2=2, \; q_2>2,\end{array}\right.
$$
$$
J_{**}=J_{**}(p_1, \, q_1, \, p_2, \, q_2)=\left\{
\begin{array}{l} \{1, \, 2, \, 3, \, 4\}, \text{ if }\min\{p_2, \, p_1'\}>2, \;
\min \{q_2, \, q_1'\}>2,
\\ \{1, \, 2, \, 3\}, \text{ if }\min\{p_2, \, p_1'\}>2, \; \min \{q_2, \, q_1'\}=2, \\
\{2, \, 3, \, 4\}, \text{ if } \min \{p_2, \, p_1'\}=2, \; \min
\{q_2, \, q_1'\}>2.\end{array}\right.
$$
\begin{Trm}
\label{main}
\begin{enumerate}
\item Let $1<p_1\le p_2<\infty$, $1<q_1\le q_2<\infty$, $p_2\ge
2$, $q_2\ge 2$,
$$\theta _1=\frac{\delta}{d}+\frac{\lambda({\bf p})}{2}-
\frac{\lambda({\bf p})}{p_2}, \;\; \theta_2=\frac{p_2\delta}{2d},$$
$$\theta_3=\alpha+\frac{\lambda({\bf q})}{2}-\frac{\lambda({\bf
q})}{q_2}, \;\; \theta _4=\frac{q_2\alpha }{2},$$
$\sigma_1=\sigma_2=0$, $\sigma_3=1$, $\sigma_4=\frac{q_2}{2}$.
Suppose that there exists $j_*\in J_*$ such that $\theta
_{j_*}<\min _{j\in J_*\backslash\{j_*\}}\theta _j$. Then
$$
d_n({\rm Id}:B^{s_1}_{p_1,q_1}(\R^d, \, w_1)\rightarrow B^{s_2}_{p_2,q_2}(\R^d,
\, w_2))\overset{n}{\asymp} n^{-\theta _{j_*}}\rho(n^{\sigma
_{j_*}}).
$$
\item Let $1<p_1\le 2\le p_2<\infty$, $1<q_1\le 2\le q_2<\infty$,
$$
\tilde \theta _1=\frac{\delta}{d}+\min \left\{\frac{1}{2}-
\frac{1}{p_2}, \, \frac{1}{p_1}-\frac 12\right\}, \;\;
\tilde\theta_2=\frac{\min\{p_2, \, p_1'\}\delta}{2d},$$
$$\tilde\theta_3=\alpha+\min \left\{\frac{1}{2}-\frac{1}{q_2}, \,
\frac{1}{q_1}-\frac 12\right\}, \;\; \tilde\theta
_4=\frac{\min\{q_2, \, q_1'\}\alpha }{2},$$
$\tilde\sigma_1=\tilde\sigma_2=0$, $\tilde\sigma_3=1$,
$\tilde\sigma_4=\frac{\min\{q_2, \, q_1'\}}{2}$. Suppose that
there exists $j_*\in J_{**}$ such that $\tilde\theta _{j_*}<\min
_{j\in J_{**}\backslash\{j_*\}}\tilde\theta _j$. Then
$$
{\cal A}_n({\rm Id}:B^{s_1}_{p_1,q_1}(\R^d, \, w_1)\rightarrow
B^{s_2}_{p_2,q_2}(\R^d, \, w_2))=
$$
$$
=\lambda_n({\rm Id}:B^{s_1}_{p_1,q_1}(\R^d, \, w_1)\rightarrow
B^{s_2}_{p_2,q_2}(\R^d, \, w_2))\overset{n}{\asymp}
n^{-\tilde\theta _{j_*}}\rho(n^{\tilde\sigma _{j_*}}).
$$
\item Let $1<p_1\le p_2<\infty$, $1<q_1\le q_2<\infty$. If $p_2\le 2$,
$q_2\le 2$ and $\alpha\ne \frac{\delta}{d}$, then
$$
d_n({\rm Id}:B^{s_1}_{p_1,q_1}(\R^d, \, w_1)\rightarrow B^{s_2}_{p_2,q_2}(\R^d,
\, w_2))\overset{n}{\asymp} \lambda_n({\rm Id}:B^{s_1}_{p_1,q_1}(\R^d, \,
w_1)\rightarrow B^{s_2}_{p_2,q_2}(\R^d, \, w_2))=
$$
$$
={\cal A}_n({\rm Id}:B^{s_1}_{p_1,q_1}(\R^d, \, w_1)\rightarrow
B^{s_2}_{p_2,q_2}(\R^d, \,
w_2))\overset{n}{\asymp}\max\{n^{-\frac{\delta}{d}}, \,
n^{-\alpha}\rho(n)\}.
$$
If $p_1\ge 2$, $q_1\ge 2$ and $\alpha\ne \frac{\delta}{d}$, then
$$
\lambda_n({\rm Id}:B^{s_1}_{p_1,q_1}(\R^d, \,
w_1)\rightarrow B^{s_2}_{p_2,q_2}(\R^d, \, w_2))=
$$
$$
={\cal A}_n({\rm Id}:B^{s_1}_{p_1,q_1}(\R^d, \, w_1)\rightarrow
B^{s_2}_{p_2,q_2}(\R^d, \,
w_2))\overset{n}{\asymp}\max\{n^{-\frac{\delta}{d}}, \,
n^{-\alpha}\rho(n)\}.
$$
\end{enumerate}
\end{Trm}

Under the hypotheses of this theorem, $d_n\overset{n}{\asymp}\lambda_n$ for
$\frac{1}{p_1}+\frac{1}{p_2}\ge 1$ and
$\frac{1}{q_1}+\frac{1}{q_2}\ge 1$.
In other cases, the orders of Kolmogorov and
linear widths are different in general.

Let us compare Theorem \ref{main} with the available results.
Given $w_2(x)\equiv 1$, we have $\beta_g=\delta$, $\alpha_g=\alpha$ and
$\rho_g=\rho$. If $\alpha>0$ is sufficiently large and $\tilde
\theta_1\ne \tilde \theta_2$, then
\begin{align}
\label{llll} \lambda_n\bigl({\rm Id}:B^{s_1}_{p_1,q_1}(\mathbb R^d, \,
w_1)\rightarrow B^{s_2}_{p_2,q_2}(\mathbb R^d, \, w_2)\bigr)
\overset{\,n}{\asymp} n^{-\min\{\tilde \theta_1, \, \tilde
\theta_2\}}.
\end{align}
In \cite{haroske} the function $\overline{w}_1=\overline{g}^{\,
-p_1}(x)$ with $\beta_{\overline{g}}<\delta$,
$\alpha_{\overline{g}}=0$, $\rho_{\overline{g}}(x)\equiv 1$ was
considered. Let $\gamma_{\overline{g}}>\delta$. In this case, $B^{s_1}_{p_1,q_1}(\mathbb R^d, \,
\overline{w}_1)\subset B^{s_1}_{p_1,q_1}(\mathbb R^d, \, w_1)$, and
the result of~\cite{haroske} on approximation
numbers follows from~\eqref{llll}. It is also worth noting that if
$\alpha$~is sufficiently small,
then the asymptotics depends on the parameters $q_1$ and~$q_2$.
In the previous results (see, e.g., \cite{haroske, haroske2, haroske3,
zhang1}), the orders of widths and entropy numbers were independent of
$q$'s (except for some limiting cases, in which $q$ was contained in
the exponent of a~logarithmic factor).

The similar result for widths of weighted Sobolev classes on a
segment was obtained in \cite{vas_alg_an}.

The proof of Theorem \ref{main} depends on estimates for the Kolmogorov and
linear widths of finite-dimensional balls in the mixed norm, as
obtained in Theorems \ref{fin_dim_mix_norm} and~\ref{lin_fin}.
Some other relations between parameters follow from the estimates for
the Kolmogorov widths of such sets, as obtained by Izaak
and Galeev in \cite{gal_izv, isaak, gal_mix}, where the
cases $p_1=1$, $q_1=\infty$, $p_2=2$ or $1<p_2\le \min\{q_2, \, 2\}$, $1\le
q_2\le \infty$ were examined.

\smallskip

In the present paper, we obtain order estimates for the Kolmogorov
widths with $p_2\ge \max\{p_1, \, 2\}$, $q_2\ge \max\{q_1, \, 2\}$,
and derive upper estimates for the linear widths
with $1<p_1\le 2\le p_2<\infty$ and $1<q_1\le 2\le q_2<\infty$.
Our approach is an extension of the argument used by Gluskin
\cite{gluskin1, bib_gluskin}. The most nontrivial steps (general
results on the Gelfand and linear widths and the construction of a~polyhedron)
were already made in \cite{bib_gluskin}. To obtain the
upper estimates for the Kolmogorov widths, it remains to apply the
H\"older inequalities several times. The proof of the lower
estimates is an extension of the arguments of~\cite{gluskin1}.
The upper estimates of the linear widths depend on
properties of the Gluskin polyhedrons and uses a~construction of polyhedrons of
a~more general form.

\section{Preliminary results}
Let us formulate the result which was proved by Kashin
\cite{bib_kashin} and Gluskin \cite{bib_gluskin}.
\begin{trma}
\label{teor_glus} Let $1<p\le q<\infty$. Then
\begin{align}
\label{gluskin} d_n(B_p^\nu, \, l_q^\nu)\underset{q,p}{\asymp}
\Phi(n, \, \nu, \, p, \, q),
\end{align}
\begin{align}
\label{gluskin_lin} \lambda_n(B_p^\nu, \, l_q^\nu)\underset{q,p}{\asymp}
\Psi(n, \, \nu, \, p, \, q),
\end{align}
where
$$
\Phi(n, \, \nu, \, p, \, q)=\left\{
\begin{array}{l} \min\left\{ 1, \, \left(\nu^{1/q}n^{-1/2}\right)
^{\left(\frac1p-\frac1q\right)/\left(\frac12-\frac1q\right)}\right\},
\;
2\le p\le q< \infty, \\
\max\left\{\nu ^{\frac 1q-\frac 1p}, \, \min \left(1, \, \nu
^{\frac 1q}n^{-\frac 12}
\right)\left(1-\frac{n}{\nu}\right)^{1/2}\right\}, \; 1< p<2\le q< \infty , \\
\max\left\{\nu ^{\frac1q-\frac1p}, \,
\left(1-\frac{n}{\nu}\right)^{\left(\frac1q-
\frac1p\right)/\left(1-\frac{2}{p}\right)}\right\}, \; 1< p\le
q\le 2,\end{array}\right.
$$
$$
\Psi(n, \, \nu, \, p, \, q)=\left\{
\begin{array}{l} \Phi(n, \, \nu, \, p, \, q), \;\;
\frac 1p+\frac 1q\ge 1, \\
\Phi(n, \, \nu, \, q', \, p'), \;\;
\frac 1p+\frac 1q\le 1.
\end{array} \right.
$$
\end{trma}

It was proved in the paper of D.D. Haroske and L. Skrzypczak
\cite{haroske} that the space $B^s_{p,q}(\R^d, \, w)$ is
isomorphic to a certain space of sequences. Let us give the
precise formulation of their result.

Let $\nu \in \Z_+$, ${\bf m}\in \Z^d$. Denote $Q_{\nu ,{\bf m}}$
the $d$-dimensional cube with sides parallel to coordinate axes,
centered at $2^{-\nu}{\bf m}$ and with side length $2^{-\nu}$,
$$
\chi _{\nu,{\bf m}}^{(p)}(x)=\left \{ \begin{array}{l} 2^{\frac{\nu
d}{p}}, \;\; x\in Q_{\nu , {\bf m}}, \\ 0, \;\;x\notin Q_{\nu ,
{\bf m}}.\end{array}\right .
$$

Let $\sigma \in \R$, $w\in {\cal A}_\infty$, $\lambda_{\nu,{\bf
m}}\in \C$, $\nu \in \Z_+$, ${\bf m}\in \Z^d$. Denote
$$
w(Q_{\nu , \, {\bf m}})=\int \limits _{Q_{\nu ,{\bf m}}}w(x)\, dx,
$$
$$
\| (\lambda _{\nu , {\bf m}})_{\nu\in \Z_+, \, {\bf m}\in \Z^d}\|
_{b^\sigma _{p,q}(w)}= \left(\sum \limits _{\nu =0}^\infty
2^{\nu\sigma q}\left\| \sum \limits _{{\bf m}\in \Z^d}\lambda
_{\nu , {\bf m}}\chi ^{(p)}_{\nu , {\bf m}}\right\|
^q_{L_p(w)}\right)^{1/q}=
$$
$$
=\left(\sum \limits _{\nu =0}^\infty 2^{\nu q\left(\sigma +\frac
dp\right)}\left(\sum \limits _{{\bf m}\in \Z^d}|\lambda _{\nu, \,
{\bf m}}|^p w(Q_{\nu ,\, {\bf m}})\right)^{q/p}\right)^{1/q},
$$
$$
b^\sigma _{p,q}(w)=\left\{(\lambda _{\nu , {\bf m}})_{\nu \in
\Z_+, \, {\bf m}\in \Z^d}:\lambda _{\nu , {\bf m}}\in \C, \; \|
(\lambda _{\nu , {\bf m}})_{\nu \in \Z_+, \, {\bf m}\in \Z^d}\|
_{b^\sigma _{p,q}(w)}<\infty\right\}.
$$
The space $b^\sigma_{p,q}(w)$ is quasi-Banach (and Banach, if
$p\ge 1$, $q\ge 1$).

Let $p_i$, $q_i>0$, $s_i$, $\sigma _i\in \R$, $w_i\in {\cal
A}_\infty$, $i=1, \, 2$, $B^{s_1}_{p_1,q_1}(\R^d, \, w_1)\subset
B^{s_2}_{p_2, \, q_2}(\R^d, \, w_2)$, $b^{\sigma
_1}_{p_1,q_1}(w_1)\subset b^{\sigma _2}_{p_2, \, q_2}(w_2)$. We
shall denote by ${\rm Id}:B^{s_1}_{p_1,q_1}(\R^d, \,
w_1)\rightarrow B^{s_2}_{p_2, \, q_2}(\R^d, \, w_2)$ and ${\rm
id}:b^{\sigma _1}_{p_1,q_1}(w_1)\rightarrow b^{\sigma _2}_{p_2, \,
q_2}(w_2)$ the natural embedding operators.

Denote $\sigma(s, \, p)=s+\frac{d}{2}-\frac{d}{p}$.
\begin{trma}
\label{isom_norm_est} For any weight function $w\in {\cal
A}_\infty$ and any $s\in \R$, $p$, $q\in (0, \, +\infty)$ there
exists an isomorphism $T_{s,p,q,w}:B^s_{p,q}(\R^d, \,
w)\rightarrow b^{\sigma(s, \, p)} _{p,q}(w)$. Moreover, for any
$w_1$, $w_2\in {\cal A}_\infty$, $s_i\in \R$, $p_i$, $q_i\in (0,
\, +\infty)$, $i=1, \, 2$, such that the operator ${\rm
Id}:B^{s_1}_{p_1,q_1}(\R^d, \, w_1)\rightarrow
B^{s_2}_{p_2,q_2}(\R^d, \, w_2)$ or ${\rm id}:b^{\sigma (s_1, \,
p_1)}_{p_1,q_1}(w_1)\rightarrow b^{\sigma(s_2, \, p_2)}
_{p_2,q_2}(w_2)$ is continuous, the following diagram
\begin{equation*}\begin{CD} B^{s_1}_{p_1,q_1}(\R^d, \, w_1) @>T_{s_1,p_1,q_1,w_1}>> b^{\sigma
(s_1, \, p_1)}_{p_1,q_1}(w_1)\\
@V{{\rm Id}}VV @V{{\rm id}}VV\\
B^{s_2}_{p_2,q_2}(\R^d, \, w_2) @>T_{s_2,p_2,q_2,w_2}>>
b^{\sigma(s_2, \, p_2)} _{p_2,q_2}(w_2)\end{CD}\end{equation*} is
commutative.
\end{trma}
Also in the paper \cite{haroske} a criterion for the existence of
continuous and compact embedding of weighted Besov spaces was
obtained. We shall use the following form of this result. For
$\lambda=(\lambda_{\nu,{\bf m}})_{\nu\in \Z_+,{\bf m}\in \Z^d}$
denote
$$
(S\lambda)_{\nu,{\bf m}}=\tilde \lambda _{\nu ,{\bf
m}}=(w_2(Q_{\nu ,{\bf m}}))^{1/p_2}2^{\nu \left(\sigma(s_2, \,
p_2)+\frac{d}{p_2}\right)} \lambda _{\nu , \, {\bf m}}.
$$
Let $X_1$, $\tilde X_1$ and $X_2$ be spaces of sequences $\tilde
\lambda=(\tilde \lambda _{\nu ,{\bf m}})_{\nu\in \Z_+,{\bf m}\in
\Z^d}$ with
\begin{align}
\label{x1} \|\tilde \lambda\| _{X_1}=\left(\sum \limits _{\nu
=0}^{\infty}2^{\nu q_1(s_1-s_2)}\left(\sum \limits _{{\bf m}\in
\Z^d}w_1(Q_{\nu,{\bf m}})w_2(Q_{\nu ,{\bf m}})^{-p_1/p_2}|\tilde
\lambda _{\nu,{\bf m}}|^{p_1}\right)^{q_1/p_1}\right)^{1/q_1},
\end{align}
\begin{align}
\label{tilde_x1} \|\tilde \lambda\| _{\tilde X_1}=\left(\sum
\limits _{\nu =0}^{\infty}\left(\sum \limits _{{\bf m}\in
\Z^d}|\tilde \lambda _{\nu,{\bf
m}}|^{p_1}\right)^{q_1/p_1}\right)^{1/q_1}
\end{align}
and
\begin{align}
\label{x2} \|\tilde \lambda\| _{X_2}=\left(\sum \limits _{\nu
=0}^{\infty}\left(\sum \limits _{{\bf m}\in \Z^d}|\tilde \lambda
_{\nu,{\bf m}}|^{p_2}\right)^{q_2/p_2}\right)^{1/q_2},
\end{align}
respectively. Define operators $\tilde T_i:B^{s_i}_{p_i,q_i}(\R^d,
\, w_i)\rightarrow X_i$ by the formula $\tilde T_i=S\cdot
T_{s_i,p_i,q_i,w_i}$, $i=1, \, 2$. Then $\tilde T_i$ are
isomorphisms and the following diagram
\begin{align}
\label{com_dia}
\begin{CD} B^{s_1}_{p_1,q_1}(\R^d, \, w_1) @>\tilde T_1>> X_1 \\
@V{{\rm Id}}VV @V{{\rm id}}VV\\
B^{s_2}_{p_2,q_2}(\R^d, \, w_2) @>\tilde T_2>> X_2\end{CD}
\end{align}
is commutative.

Let ${\cal N}\subset \Z_+\times \Z^d$. For $\lambda =(\lambda
_{\nu , {\bf m}})_{\nu \in \Z_+, \, {\bf m}\in \Z^d}$ denote
\begin{align}
\label{def_of1_pmn} (P_{{\cal N}}\lambda)_{\nu , \, {\bf
m}}=\left\{ \begin{array}{l} \lambda _{\nu , \, {\bf m}}, \;
\text{ if }(\nu , \, {\bf m})\in {\cal N},
\\ 0,  \text{ otherwise.}\end{array}\right .
\end{align}

Denote
$$
{\cal M}_\nu =\{{\bf m}\in \Z^d:(\nu , \, {\bf m})\in {\cal N}\}.
$$

If ${\cal N}\ne \varnothing$, $p_1\le p_2$ and $q_1\le q_2$, then $\|P_{{\cal N}}\|_{\tilde
X_1\rightarrow X_2}=1$. In addition,
\begin{align}
\label{norm_bspq_proj} \| P_{{\cal N}}\| _{X_1\rightarrow \tilde
X_1}= \sup_{\nu\in \Z_+,\, {\bf m}\in {\cal M}_\nu}2^{-\nu
(s_1-s_2)}(w_1(Q_{\nu , {\bf m}}))^{-1/p_1} (w_2(Q_{\nu ,{\bf
m}}))^{1/p_2},
\end{align}
\begin{align}
\label{norm_bspq_proj_m1} \| P_{{\cal N}}\| _{\tilde
X_1\rightarrow X_1}= \sup_{\nu\in \Z_+,\, {\bf m}\in {\cal
M}_\nu}2^{\nu (s_1-s_2)}(w_1(Q_{\nu , {\bf m}}))^{1/p_1}
(w_2(Q_{\nu ,{\bf m}}))^{-1/p_2}
\end{align}
(it follows from results of the paper \cite{kuhn_leopold} or can
be proved directly).

{\bf Remark}. If $p_1>p_2$ or $q_1>q_2$ then the value of
$\|P_{{\cal N}}\| _{X_1\rightarrow X_2}$ was calculated in
\cite{kuhn_leopold}.

Let $X$, $Y$, $Z$ be normed spaces, $M\subset X$, let
$A:X\rightarrow Y$, $T:Z\rightarrow X$ be linear continuous
operators. It can be easily shown that
$$
d_n(AM, \, Y)\le \| A\| d_n(M, \, X), \;\; {\cal
A}_n(AT:Z\rightarrow Y)\le \| A\| {\cal A}_n(T:Z \rightarrow
X),\;\; n\in \Z_+.
$$
In particular, if ${\cal N}'\subset {\cal N}\subset \Z_+\times
\Z^d$, then
\begin{align}
\label{proj} d_n(P_{{\cal N}}:X_1\rightarrow X_2)\ge d_n(P_{{\cal
N}'}:X_1\rightarrow P_{{\cal N}'}X_2),
\end{align}
\begin{align}
\label{proj1} {\cal A}_n(P_{{\cal N}}:X_1\rightarrow X_2)\ge {\cal
A}_n(P_{{\cal N}'}:X_1\rightarrow P_{{\cal N}'}X_2).
\end{align}

\section{Estimates for Kolmogorov widths of finite-dimensional balls
in a mixed norm}
Let $m$, $k\in \N$, $1\le p, \, q<\infty$, let $l_{p,q}^{m,k}$
be the space $\R^{mk}$ endowed with the norm
$$
\left \| (x_{ij})_{1\le i\le m, \, 1\le j\le k}\right \|
_{l_{p,q}^{m,k}}=\left( \sum \limits _{j=1}^k\left(\sum \limits
_{i=1}^m |x_{ij}|^p\right)^{q/p}\right)^{1/q}
$$
and $B_{p,q}^{m,k}$ be the unit ball of this space. The dual space
of $l_{p,q}^{m,k}$ coincides with $l_{p',q'}^{m,k}$.

We denote by $B_X$ the unit ball of a normed space $X$; by $X^*$
denote the dual space of $X$.

The following lemma (see \cite{gluskin1}) is needed for the sequel.
\begin{Lem}
\label{theta} Let $X_0\subset X_{\theta}\subset
X_1$ be $N$-dimensional normed spaces and $0<\theta<1$.
Let for any $x^*\in X^*_1$ the following inequality holds:
$$
\| x^*\| _{X^*_\theta}\le \| x^*\| _{X_0^*}^{1-\theta}\| x^*\|
_{X_1^*}^\theta .
$$
Then for any $n\in \{0, \, \dots , \, N\}$ we have
\begin{align}
\label{interp_lem} d_n(B_{X_\theta}, \, X_1)\le (d_n(B_{X_0}, \,
X_1))^{1-\theta}.
\end{align}
\end{Lem}
\begin{Sta}
\label{1xicrho} Let $1<p, \, q<\infty$, $r\in \N$. Then there
exists $c_1(p, \, q)>0$ such that for any $x=(x_1, \, \dots , \,
x_r)\in \R ^r$ the following inequality holds:
\begin{align}
\label{1xicrho1} \left(\sum \limits _{i=1}^r
|1-x_i|^p\right)^{q/p}\ge \frac{r^{q/p}}{2}+c_1(p, \, q)\left(\sum
\limits _{i=1}^r |x_i|^p\right)^{q/p}- qr ^{\frac{q}{p}-1}\sum
\limits _{i=1}^r x_i.
\end{align}
Here we may assume that $c_1(p, \, q)$ continuously depends on $(p, \, q)$.
\end{Sta}
\begin{proof}
Note that the left-hand side of (\ref{1xicrho1}) is convex
with respect to $x$. Therefore,
$$
\left(\sum \limits _{i=1}^r |1-x_i|^p\right)^{q/p}\ge r ^{q/p}-qr
^{\frac{q}{p}-1}\sum \limits _{i=1}^r x_i.
$$
If $\sum \limits _{i=1}^r |x_i|^p\le t\cdot r$ then there exists a
number $c_2(p, \, q, \, t)$ such that
$$
r ^{q/p}-qr ^{\frac{q}{p}-1}\sum \limits _{i=1}^r x_i\ge \frac 12
r ^{q/p}-qr ^{\frac{q}{p}-1}\sum \limits _{i=1}^r x_i+c_2(p, \, q,
\, t)\left(\sum \limits _{i=1}^r |x _i|^p\right)^{q/p}.
$$
Let us prove that there exists a number $t_0(p, \, q)>0$ such
that, if $\sum \limits _{i=1}^r |x_i|^p\ge t_0(p, \, q)r$, then
the following inequality holds:
$$
\left(\sum \limits _{i=1}^r |1-x_i|^p\right)^{q/p}\ge r ^{q/p}-qr
^{\frac{q}{p}-1}\sum \limits _{i=1}^r x_i+2^{-q}\left(\sum \limits
_{i=1}^r |x_i|^p\right)^{q/p}.
$$
Indeed, by the inequality $(a+b)^q\le 2^{q-1}(a^q+b^q)$ and the
Minkowski inequality we have
$$
\left(\sum \limits _{i=1}^r |1-x_i|^p\right)^{q/p}\ge
2^{1-q}\left(\sum \limits _{i=1}^r |x_i|^p\right)^{q/p}-r ^{q/p}.
$$
We claim that for sufficiently large $t_0(p, \, q)$ the right-hand
side is minorized by
$$
2^{-q}\left(\sum \limits _{i=1}^r |x_i|^p\right)^{q/p}+r ^{q/p}-qr
^{\frac{q}{p}-1}\sum \limits _{i=1}^r x_i,
$$
that is
$$
2r ^{q/p}-qr ^{\frac{q}{p}-1}\sum \limits _{i=1}^r x_i\le
2^{-q}\left(\sum \limits _{i=1}^r |x_i|^p\right)^{q/p}.
$$
Actually, there exists $t\ge t_0(p, \, q)$ such that $\sum \limits
_{i=1}^r |x_i|^p=t\cdot r$. By H\"{o}lder inequality, for enough
large $t_0(p, \, q)$ we have
$$
2r ^{q/p}-qr ^{\frac{q}{p}-1}\sum \limits _{i=1}^r x_i\le 2r
^{q/p}+qr ^{\frac{q}{p}-1}r ^{1-\frac 1p}t^{1/p}r ^{1/p}=
$$
$$
= r ^{q/p}(2+qt^{1/p})\le 2^{-q}t^{q/p}r ^{q/p}.
$$
\end{proof}
\begin{Lem}
\label{lem_interp} Let $n\in \N$, $(a_j)_{j=1}^n\subset \R_+$,
$(b_j)_{j=1}^n\subset \R_+$, ${\bf v}=(v_1, \, v_2)\in [2, \,
+\infty)^2$, $v_1\le v_2$, let $\lambda({\bf v})$ be defined by
formula similar to (\ref{lamb_p}) and $0<\lambda\le \lambda({\bf
v})$. Then
\begin{align}
\label{ajbjlambdap_lem} \left(\sum \limits _{j=1}^n
a_j^{v_1'\lambda}b_j^{v_1'(1-\lambda)}\right)^{1/v_1'}\le
\left(\sum \limits _{j=1}^n a_j^2\right)^{\lambda /2}\left(\sum
\limits _{j=1}^n b_j^{v_2'}\right)^{(1-\lambda)/v_2'}.
\end{align}
\end{Lem}
\begin{proof}
Let $s=\lambda /2$, $t=(1-\lambda)/v_2'$,
$y_j=a_j^{\frac{2s}{s+t}}$, $z_j=b_j^{\frac{v_2't}{s+t}}$. Then
(\ref{ajbjlambdap_lem}) is equivalent to the following inequality:
$$
\left(\sum \limits _{j=1}^n y_j^{v_1'(s+t)}z_j^{v_1'(s+t)}\right)
^{\frac{1}{v_1'(s+t)}}\le \left(\sum \limits _{j=1}^n
y_j^{\frac{s+t}{s}}\right)^{\frac{s}{s+t}}\left(\sum \limits
_{j=1}^n z_j^{\frac{s+t}{t}}\right)^{\frac{t}{s+t}}.
$$
By the H\"{o}lder inequality, the right-hand side is minorized by
$\sum \limits _{j=1}^n y_jz_j$. Therefore it is enough to show
that $v_1'(s+t)\ge 1$. Indeed,
$$
v_1'\left(\frac{\lambda}{2}+\frac{1-\lambda}{v_2'}\right)\ge 1
$$
is equivalent to the inequality $\lambda\left(\frac
12-\frac{1}{v_2}\right)\le \frac{1}{v_1}-\frac{1}{v_2}$, which
follows from conditions $v_2\ge v_1\ge 2$ and $\lambda\le
\lambda({\bf v})$.
\end{proof}
\begin{Trm}
\label{fin_dim_mix_norm} Let $1\le p_1\le p_2<\infty$, $1\le
q_1\le q_2<\infty$, $p_2\ge 2$, $q_2\ge 2$. Then there exists
$a=a(p_2, \, q_2)>0$ such that for any $m$, $k$,
$n\in \N$ satisfying the condition $n\le amk$ the following
estimate holds:
$$
d_n(B_{p_1,q_1}^{m,k}, \,
l_{p_2,q_2}^{m,k})\underset{p_1,p_2,q_1,q_2}{\asymp}\Phi_0=\Phi
_0(m, \, k, \, n, \, p_1, \, p_2, \, q_1, \, q_2),
$$
where
$$
\Phi _0=\min \left\{ 1, \, n^{-1/2}m^{1/p_2}k^{1/q_2}\right\},
\text{ if }p_1\le 2, \; q_1\le 2,
$$
$$
\Phi _0= \min \left\{1, \,
(n^{-1/2}m^{1/p_2}k^{1/q_2})^{\lambda({\bf p})}, \,
m^{\frac{1}{p_2}-\frac{1}{p_1}}\left(n^{-1/2}m^{1/2}k^{1/q_2}\right)^{\lambda({\bf
q})} \right\}
$$
if $\lambda({\bf p})\le \lambda({\bf q})$ and $\lambda({\bf p})<1$,
$$
\Phi _0= \min \left\{1, \, (n^{-1/2}m^{1/p_2}k^{1/q_2})^{\lambda({\bf q})},
\,k^{\frac{1}{q_2}-\frac{1}{q_1}}
\left(n^{-1/2}m^{1/{p_2}}k^{1/2}\right)^{\lambda({\bf p})} \right\}
$$
if $\lambda({\bf p})\ge \lambda({\bf q})$ and
$\lambda({\bf q})<1$.
\end{Trm}
For convenience of the reader we give a list of special cases
where the formula for $\Phi_0$ can be simplified. If $n\le
m^{2/p_2}k^{2/q_2}$ then $\Phi _0=1$; if $\lambda({\bf p})\le
\lambda({\bf q})$, $p_1>2$ and $m^{2/p_2}k^{2/q_2}\le n\le
mk^{2/q_2}$ then $\Phi _0=(n^{-1/2}m^{1/p_2}k^{1/q_2})
^{\lambda({\bf p})}$; if $\lambda({\bf p})\le \lambda({\bf q})$,
$p_1>2$ and $mk^{2/q_2}\le n\le mk$ then $\Phi _0=
m^{\frac{1}{p_2}-\frac{1}{p_1}}\left(n^{-1/2}m^{1/2}k^{1/q_2}\right)^{\lambda({\bf
q})}$; if $\lambda({\bf p})\ge \lambda({\bf q})$, $q_1>2$ and
$m^{2/p_2}k^{2/q_2}\le n\le km^{2/p_2}$ then $\Phi
_0=(n^{-1/2}m^{1/p_2}k^{1/q_2})^{\lambda({\bf q})}$; if
$\lambda({\bf p})\ge \lambda({\bf q})$, $q_1>2$ and $km^{2/p_2}\le
n\le mk$ then $\Phi _0=k^{\frac{1}{q_2}-\frac{1}{q_1}}
\left(n^{-1/2}m^{1/{p_2}}k^{1/2}\right)^{\lambda({\bf p})}$.
\begin{proof}
{\it Proof of the upper estimate.} Since $p_1\le p_2$,
$q_1\le q_2$, we have
\begin{align}
\label{dnbp1q1le1} d_n(B_{p_1,q_1}^{m,k}, \, l_{p_2,q_2}^{m,k})\le
1.
\end{align}
Let us prove other inequalities of the upper estimate.

Note that
$$
d_n(B_{p_1,q_1}^{m,k}, \, l_{p_2,q_2}^{m,k})\le d_n(B_{\max(p_1,
2),\max(q_1,2)}^{m,k}, \, l_{p_2,q_2}^{m,k}).
$$
Thus we may assume that $p_1\ge 2$, $q_1\ge 2$ and
$\lambda({\bf p})=\frac{\frac{1}{p_1}-\frac{1}{p_2}}{\frac
12-\frac{1}{p_2}}$, $\lambda({\bf
q})=\frac{\frac{1}{q_1}-\frac{1}{q_2}}{\frac 12-\frac{1}{q_2}}$.

First consider the case $p_1=q_1=2$. Let $s=\max \{p_2, \,
q_2\}$. By Theorem \ref{teor_glus}, we have
$$
d_n(B^{mk}_2, \,
l^{mk}_s)\underset{s}{\lesssim}n^{-1/2}(mk)^{1/s}.
$$
This, together with the inequality
$$
\| x\| _{l^{m,k}_{p_2, \, q_2}}\le
m^{\frac{1}{p_2}-\frac{1}{s}}k^{\frac{1}{q_2}-\frac 1s}\| x\|
_{l^{mk}_{s}},
$$
implies the estimate
\begin{align}
\label{dnb2pl2q} d_n(B^{m,k}_{2,2}, \,
l^{m,k}_{p_2,q_2})\underset{p_2,q_2}{\lesssim}n^{-1/2}m^{1/p_2}k^{1/q_2}.
\end{align}

By (\ref{lamp1p2}), if $\min(\lambda({\bf p}), \,
\lambda({\bf q}))=1$ then $p_1=2$, $q_1=2$ and the assertion follows from
(\ref{dnbp1q1le1}) and (\ref{dnb2pl2q}).

Let us consider the case $\min(\lambda({\bf p}), \, \lambda({\bf q}))<1$.
First we prove the inequalities
\begin{align}
\label{dn_est1} d_n(B_{p_1,q_1}^{m,k}, \,
l_{p_2,q_2}^{m,k})\underset{p_1,q_1,p_2,q_2}{\lesssim}\left(n^{-1/2}m^{1/p_2}
k^{1/q_2}\right)^{\lambda({\bf p})}, \;\; \lambda({\bf p})\le
\lambda({\bf q}),\;\;\lambda({\bf p})<1,
\end{align}
\begin{align}
\label{dn_est2} d_n(B_{p_1,q_1}^{m,k}, \,
l_{p_2,q_2}^{m,k})\underset{p_1,q_1,p_2,q_2}{\lesssim}\left(n^{-1/2}m^{1/p_2}
k^{1/q_2}\right)^{\lambda({\bf q})}, \;\; \lambda({\bf p})\ge
\lambda({\bf q}), \;\;\lambda({\bf q})<1.
\end{align}

Let $\lambda({\bf p})\le \lambda({\bf q})$, $\lambda({\bf
p})<1$. We shall apply Lemma \ref{theta} for $\theta =1-\lambda({\bf
p})$, $X_0=l^{m,k}_{2,2}$, $X_1=l^{m,k}_{p_2,q_2}$, $X_\theta
=l^{m,k}_{p_1,q_1}$. If
\begin{align}
\label{xp1q1x22xp2q2} \| x\| _{l^{m,k}_{p_1',q_1'}}\le \| x\|
_{l^{m,k}_{2,2}}^{\lambda({\bf p})}\| x\|
_{l^{m,k}_{p_2',q_2'}}^{1-\lambda({\bf p})},
\end{align}
then (\ref{dn_est1}) follows from (\ref{interp_lem}) and
(\ref{dnb2pl2q}).

We claim that (\ref{xp1q1x22xp2q2}) holds. Using Lemma
\ref{lem_interp} for $n:=m$, $v_1:=p_1$, $v_2:=p_2$,
$\lambda:=\lambda({\bf p})$, $a_i:=b_i:=x_{ij}$, $1\le i\le m$, we
get
\begin{align}
\label{int2} \left(\sum \limits _{j=1}^k\left(\sum \limits
_{i=1}^m |x_{ij}|^{p_1'}\right)^{q_1'/p_1'}\right)^{1/q_1'} \le
\left(\sum \limits _{j=1}^k\left(\sum \limits _{i=1}^m
|x_{ij}|^{2}\right)^{q_1'\lambda({\bf p})/2}\left(\sum \limits
_{i=1}^m |x_{ij}|^{p_2'}\right)^{q_1'(1-\lambda({\bf
p}))/p_2'}\right)^{1/q_1'}.
\end{align}
Now we denote
\begin{align}
\label{def_of_ajbj} a_j:=\left(\sum \limits _{i=1}^m
|x_{ij}|^2\right)^{1/2}, \;\; b_j:=\left(\sum \limits _{i=1}^m
|x_{ij}|^{p_2'}\right)^{1/p_2'}, \;\; 1\le j\le k.
\end{align}
Taking into account Lemma \ref{lem_interp} for $n:=k$, $v_1:=q_1$,
$v_2:=q_2$, $\lambda :=\lambda({\bf p})$, we obtain
$$
\left(\sum \limits _{j=1}^k a_j^{q_1'\lambda({\bf
p})}b_j^{q_1'(1-\lambda({\bf p}))}\right)^{1/q_1'}\le \left(\sum
\limits _{j=1}^k a_j^2\right)^{\lambda ({\bf p})/2}\left(\sum
\limits _{j=1}^k b_j^{q_2'}\right)^{(1-\lambda({\bf p}))/q_2'}.
$$
This, together with (\ref{int2}), implies (\ref{xp1q1x22xp2q2}).

Let us prove (\ref{dn_est2}). By Lemma \ref{theta},
it is enough to check the inequality
\begin{align}
\label{nxp1sq1smklq}
\| x\| _{l^{m,k}_{p_1',q_1'}}\le \| x\|
_{l^{m,k}_{2,2}}^{\lambda({\bf q})}\|x\|
_{l^{m,k}_{p_2',q_2'}}^{1-\lambda({\bf q})}.
\end{align}
Using Lemma \ref{lem_interp} with $n:=m$, $v_1:=p_1$, $v_2:=p_2$,
$\lambda:=\lambda({\bf q})$, $a_i:=b_i:=x_{ij}$, $1\le i\le m$, we
obtain that
\begin{align}
\label{twsev} \left(\sum \limits _{j=1}^k\left(\sum \limits
_{i=1}^m |x_{ij}|^{p_1'}\right)^{q_1'/p_1'}\right)^{1/q_1'} \le
\left(\sum \limits _{j=1}^k\left(\sum \limits _{i=1}^m
|x_{ij}|^{2}\right)^{q_1'\lambda({\bf q})/2}\left(\sum \limits
_{i=1}^m |x_{ij}|^{p_2'}\right)^{q_1'(1-\lambda({\bf
q}))/p_2'}\right)^{1/q_1'}.
\end{align}
Now we define $(a_j)_{j=1}^k$, $(b_j)_{j=1}^k$ by the formula
(\ref{def_of_ajbj}). Taking into account Lemma  \ref{lem_interp}
for $n:=k$, $v_1:=q_1$, $v_2:=q_2$, $\lambda :=\lambda({\bf q})$,
we get
$$
\left(\sum \limits _{j=1}^k a_j^{q_1'\lambda({\bf
q})}b_j^{q_1'(1-\lambda({\bf q}))}\right)^{1/q_1'}\le \left(\sum
\limits _{j=1}^k a_j^2\right)^{\lambda({\bf q})/2}\left(\sum
\limits _{j=1}^k b_j^{q_2'}\right)^{(1-\lambda({\bf q}))/q_2'}.
$$
This, together with (\ref{def_of_ajbj}) and (\ref{twsev}), implies
(\ref{nxp1sq1smklq}).

Let us prove inequalities
$$
d_n(B_{p_1,q_1}^{m,k}, \,
l_{p_2,q_2}^{m,k})\underset{p_1,q_1,p_2,q_2}{\lesssim}n^{-\frac{\lambda({\bf
q})}{2}}k^{\frac{\lambda({\bf
q})}{q_2}}m^{\frac{1}{p_2}-\frac{1}{p_1}+\frac{\lambda({\bf
q})}{2}}, \text{ if } \lambda({\bf p})\le \lambda({\bf q}), \;\;
p_1>2,
$$
\begin{align}
\label{dn_est4} d_n(B_{p_1,q_1}^{m,k}, \,
l_{p_2,q_2}^{m,k})\underset{p_1,q_1,p_2,q_2}{\lesssim}
n^{-\frac{\lambda({\bf p})}{2}}m^{\frac{\lambda({\bf
p})}{p_2}}k^{\frac{1}{q_2}-\frac{1}{q_1}+\frac{\lambda({\bf
p})}{2}}, \text{ if } \lambda({\bf p})\ge \lambda({\bf q}), \;\;
q_1>2.
\end{align}
Let $\lambda({\bf p})\le \lambda({\bf q})$. We choose $\sigma \in
[2, \, p_2]$ such that
$$
\frac{\frac{1}{\sigma}-\frac{1}{p_2}}{\frac{1}{2}-\frac{1}{p_2}}=\lambda({\bf
q}).
$$
Since $p_1>2$, it follows that $\sigma \le p_1$ and
$$
d_n(B_{p_1,q_1}^{m,k}, \, l_{p_2,q_2}^{m,k})\le
m^{\frac{1}{\sigma}-\frac{1}{p_1}}d_n(B_{\sigma,q_1}^{m,k}, \,
l_{p_2,q_2}^{m,k})\stackrel{(\ref{dnb2pl2q}),(\ref{dn_est2})}{\underset{p_1,q_1,p_2,q_2}{\lesssim}}
$$
$$
\lesssim m^{\frac{1}{p_2}+\lambda({\bf
q})\left(\frac{1}{2}-\frac{1}{p_2}\right)-\frac{1}{p_1}}\left(n^{-1/2}m^{1/p_2}
k^{1/q_2}\right)^{\lambda({\bf q})}=n^{-\frac{\lambda({\bf
q})}{2}}k^{\frac{\lambda({\bf
q})}{q_2}}m^{\frac{1}{p_2}-\frac{1}{p_1}+\frac{\lambda({\bf
q})}{2}}.
$$
The proof of (\ref{dn_est4}) is similar.

{\it Proof of the lower estimate.} We use the method from the
article \cite{gluskin1}. Let $r \in \{1, \, \dots, \, m\}$, $l\in
\{1, \, \dots, \, k\}$, let $S_m$, $S_k$ be groups of permutations
of $m$ and $k$ elements respectively,
$$
E_m=\{(\xi_1, \, \dots, \, \xi_m)\in \R^m:\xi_i=\pm 1, \; 1\le
i\le m\},
$$
$$
E_k=\{(\xi_1, \, \dots, \, \xi_k)\in \R^k:\xi_i=\pm 1, \; 1\le
i\le k\}.
$$
By $G$ denote the set
$$
\{(\sigma _1, \, \sigma _2, \, \varepsilon _1, \, \varepsilon
_2):\sigma _1\in S_m, \, \sigma _2\in S_k, \, \varepsilon _1\in
E_m, \, \varepsilon _2\in E_k\}.
$$
For $x=(x_{i,j})_{1\le i\le m, \, 1\le j\le k}$, $\gamma=(\sigma
_1, \, \sigma _2, \, \varepsilon _1, \, \varepsilon _2)\in G$,
$\varepsilon _1=(\varepsilon _{1,1}, \, \dots, \, \varepsilon
_{1,m})$, $\varepsilon _2=(\varepsilon _{2,1}, \, \dots, \,
\varepsilon _{2,k})$ we define $\gamma(x)=(\varepsilon
_{1,i}\varepsilon _{2,j}x_{\sigma _1(i),\sigma _2(j)})_{1\le i\le
m,1\le j\le k}$.

Let $1\le i\le m$, $1\le j\le k$,
$$
e^{m,k,r,l}_{i,j}=\left\{ \begin{array}{l} 1, \text{ if
 }1\le i\le r \text{ and } 1\le j\le l, \\ 0,\text{ in other cases}.\end{array}\right.
$$
Denote $e=(e^{m,k,r,l}_{i,j})_{1\le i\le m,1\le j\le k}$,
$$
V^{m,k}_{r ,l}={\rm conv}\, \{\gamma(e):\gamma\in G\}.
$$
We claim that there exist $a=a(p_2, \, q_2)>0$ and $b=b(p_2, \,
q_2)>0$ such that for $n\le a m^{2/p_2}k^{2/q_2}r
^{1-2/p_2}l^{1-2/q_2}$ the following inequality holds:
\begin{align}
\label{low_est_polytope} d_n(V^{m,k}_{r , l}, \,
l_{p_2,q_2}^{m,k})\ge br ^{1/p_2}l^{1/q_2}.
\end{align}
Here we may assume that $a(p_2, \, q_2)$ decreases in $p_2$ and $q_2$ and
$b(p_2, \, q_2)$ is continuous.

Let $Y\subset l_{p_2,q_2}^{m,k}$ be a subspace of dimension not
exceeding $n$. For any $\gamma\in G$ we denote by
$y^\gamma=(y^\gamma_{i,j})_{1\le i\le m,1\le j\le k}$ the element
of $Y$ which is nearest to $\gamma(e)$. Let us get the lower
estimate for $\max _{\gamma\in G}\|
\gamma(e)-y^\gamma\|_{l^{m,k}_{p_2,q_2}}$.

We claim that there exists $\tilde c_1(p_2,\, q_2)>0$
(continuously depending on $p_2$ and $q_2$) such that
\begin{align}
\label{sj1ksi1mge}
\begin{array}{c}
\displaystyle \sum \limits _{j=1}^k \left(\sum \limits _{i=1}^m
|\gamma(e)_{i,j}-y^\gamma _{i,j}|^{p_2}\right)^{q_2/p_2}\ge
\frac{lr ^{q_2/p_2}}{4}+
\\
+\tilde c_1(p_2, \, q_2)\sum \limits _{j=1}^k\left(\sum \limits
_{i=1}^m |y^\gamma_{i,j}|^{p_2}\right)^{q_2/p_2}-\frac{q_2}{2}r
^{\frac{q_2}{p_2}-1}\sum \limits _{j=1}^k \sum \limits _{i=1}^m
\gamma(e)_{i,j}y^\gamma_{i,j}.
\end{array}
\end{align}

Indeed, let
$$
J_1^\gamma=\{j\in \overline{1, \, k}: \forall i\in\overline{1, \,
m}\;\;\; \gamma(e)_{i,j}=0\}, \;\;\; J_2^\gamma=\{1, \, \dots , \,
k\}\backslash J_1^\gamma;
$$
for $j\in J_2^\gamma$ we put
$$
I_1^\gamma=\{i\in\overline{1, \, m}:\gamma(e)_{i,j}=0\}, \;\;\;
I_2^\gamma=\{1, \, \dots , \, m\}\backslash I_1^\gamma.
$$
Then Proposition \ref{1xicrho} together with the equality ${\rm
card}\, I^\gamma_2=r$ and the inequality
\begin{align}
\label{th} (\xi+\eta)^\theta\ge\frac{\xi^\theta+\eta^\theta}{2},
\;\; \xi, \; \eta\in\R_+, \;\; \theta>0,
\end{align}
imply
$$
\sum \limits _{j=1}^k \left(\sum \limits _{i=1}^m
|\gamma(e)_{i,j}-y^\gamma _{i,j}|^{p_2}\right)^{q_2/p_2}=\sum
\limits _{j\in J_1^\gamma} \left(\sum \limits
_{i=1}^m|y^\gamma_{i,j}|^{p_2}\right)^{q_2/p_2}+
$$
$$
+\sum \limits _{j\in J_2^\gamma}\left(\sum \limits _{i\in
I_1^\gamma}|y^\gamma_{i,j}|^{p_2}+ \sum \limits _{i\in
I_2^\gamma}|1-\gamma(e)_{i,j}y^\gamma_{i,j}|^{p_2}\right)^{q_2/p_2}\stackrel{(\ref{1xicrho1}),
(\ref{th})}{\ge}
$$
$$
\ge \sum \limits _{j\in J_1^\gamma}\left(\sum \limits _{i=1}^m
|y^\gamma_{i,j}|^{p_2}\right)^{q_2/p_2}+ \frac 12\sum \limits
_{j\in J_2^\gamma}\left( \sum \limits _{i\in
I_1^\gamma}|y^\gamma_{i,j}|^{p_2}\right)^{q_2/p_2}+ \frac 14\sum
\limits _{j\in J_2^\gamma}r ^{q_2/p_2}+
$$
$$
+\frac 12 \sum \limits _{j\in J_2^\gamma}c_1(p_2,\, q_2)
\left(\sum \limits _{i\in
I_2^\gamma}|y^\gamma_{i,j}|^{p_2}\right)^{q_2/p_2}- \frac{q_2}{2}r
^{\frac{q_2}{p_2}-1}\sum \limits _{j\in J_2^\gamma}\sum \limits
_{i\in I_2^\gamma} \gamma(e)_{i,j}y^\gamma_{i,j}.
$$
It remains to apply the equality ${\rm card}\, J_2^\gamma=l$.

Averaging both sides of the inequality (\ref{sj1ksi1mge}) over
$\gamma\in G$, we get
\begin{align}
\label{low1}
\begin{array}{c}
\max _{\gamma\in G}\| \gamma(e)-y^\gamma\|
_{l^{m,k}_{p_2,q_2}}^{q_2}\ge |G|^{-1}\sum \limits _{\gamma\in
G}\| \gamma(e)-y^\gamma\|
_{l^{m,k}_{p_2,q_2}}^{q_2} \ge\frac{lr ^{q_2/p_2}}{4}+ \\
+\tilde c_1(p_2, \, q_2)|G|^{-1}\sum \limits _{\gamma\in G}\sum
\limits _{j=1}^k\left(\sum \limits _{i=1}^m
|y^\gamma_{i,j}|^{p_2}\right)^{q_2/p_2}-\frac{q_2}{2}|G|^{-1}r
^{\frac{q_2}{p_2}-1}\sum \limits _{\gamma\in G}\sum \limits
_{j=1}^k \sum \limits _{i=1}^m \gamma(e)_{i,j}y^\gamma_{i,j}.
\end{array}
\end{align}
Let us obtain the upper estimate of module of the last summand.
Consider the space $l_2(G)=\{\varphi :G\rightarrow \R\}$ with
inner product
$$
\langle \varphi , \, \psi\rangle =|G|^{-1}\sum \limits _{\gamma\in
G}\varphi(\gamma)\psi(\gamma).
$$
For $1\le i\le m$, $1\le j\le k$, $\gamma\in G$ we denote $\varphi
_{ij}(\gamma)=\gamma(e)_{i,j}$, $z_{ij}(\gamma)=y^\gamma_{i,j}$,
$L={\rm span}\, \{z_{ij}\}_{1\le i\le m, \, 1\le j\le k}\subset
l_2(G)$. Then $\dim L\le n$. Indeed, consider the matrix
$(z_s(\gamma))_{s\in S,\, \gamma\in G}$, where $S=\{s=(i,j):1\le
i\le m, \, 1\le j\le k\}$. Since $y^\gamma\in Y$, $\dim Y\le n$,
then the rank of this matrix does not exceed $n$; therefore, $\dim
L\le n$. Let $P$ be the orthogonal projector onto $L$. Then its
Hilbert -- Schmidt norm does not exceed $n^{1/2}$.

We claim that the vectors $\varphi _{ij}$ form an orthogonal
system and $\|\varphi _{ij}\| _{l_2(G)}^2=\frac{r l}{mk}$. Indeed,
let $(i, \, j)\ne (i', \, j')$. Check that $\sum \limits
_{\gamma\in G}\varphi _{ij}(\gamma)\varphi_{i'j'}(\gamma)=0$.
Suppose that $i\ne i'$ (the case $j\ne j'$ is considered
similarly). Let $\gamma=(\sigma _1, \, \sigma _2, \, \varepsilon
_1, \, \varepsilon _2)$ and $\varepsilon_{1,i'}=1$. We have
$\varphi _{ij}(\gamma)\varphi_{i'j'}(\gamma)=\varepsilon
_{1,i}\varepsilon_{2,j}\varepsilon _{1,i'}\varepsilon_{2,j'}
e^{m,k,r,l}_{\sigma _1(i),\sigma _2(j)}e^{m,k,r,l}_{\sigma
_1(i'),\sigma _2(j')}$. Put $\tilde \gamma=(\sigma _1, \, \sigma
_2, \, \tilde\varepsilon _1, \, \varepsilon _2)$; here $\tilde
\varepsilon _{1,s}=\varepsilon _{1,s}$ if $s\ne i'$, $\tilde
\varepsilon _{1,i'}=-1$. Since $i\ne i'$, then $\tilde
\varepsilon_{1,i}\tilde \varepsilon _{1,i'}=-\varepsilon_{1,i}
\varepsilon _{1,i'}$. Thus we have the orthogonality of the system
$(\varphi _{ij})$. Let us show that $|G|^{-1}\sum \limits
_{\gamma\in G}\varphi _{ij}^2(\gamma)=\frac{r l}{mk}$. Indeed,
$|G|=2^m\cdot 2^km!k!$,
$$
{\rm card}\, \{\gamma\in G:\varphi _{ij}(\gamma)=\pm 1\}=
$$
$$
=2^m\cdot 2^k{\rm card}\{\sigma _1\in S_m:\,\sigma _1(i)\le r\}
\cdot {\rm card}\, \{\sigma _2\in S_k:\,\sigma _2(j)\le
l\}=2^m\cdot 2^k r(m-1)!l(k-1)!.
$$

We have
$$
\left ||G|^{-1}\sum \limits _{\gamma\in G} \sum\limits _{j=1}^k
\sum \limits _{i=1}^m y^\gamma_{i,j}\gamma(e)_{i,j}\right |=\left
| \sum \limits _{j=1}^k\sum \limits _{i=1}^m \langle \varphi
_{ij}, \, z_{ij}\rangle\right |=
$$
$$
=\left |\sum \limits _{j=1}^k\sum \limits _{i=1}^m \langle
P\varphi _{ij}, \, z_{ij}\rangle\right |\le \left(
\sum \limits _{j=1}^k\sum \limits _{i=1}^m \|P\varphi
_{ij}\|^2_{l_2(G)}\right)^{1/2}\left(
\sum\limits _{j=1}^k \sum \limits _{i=1}^m\| z_{ij} \|
^2_{l_2(G)}\right)^{1/2}\le
$$
$$
\le n^{1/2}\left(\frac{r l}{mk}\right)^{1/2}\left( \sum \limits
_{j=1}^k\sum \limits _{i=1}^m \|z_{ij}\|^2_{l_2(G)}\right)^{1/2}=
n^{1/2}\left(\frac{r l}{mk}\right)^{1/2}\left(|G|^{-1}\sum \limits
_{\gamma\in G}\sum \limits _{j=1}^k\sum \limits _{i=1}^m
|y_{i,j}^\gamma|^2\right)^{1/2}\le
$$
$$
\le (nr l)^{1/2}m^{-1/p_2}k^{-1/q_2}\left(|G|^{-1}\sum \limits
_{\gamma\in G}\sum \limits _{j=1}^k \left(\sum \limits
_{i=1}^m|y_{i,j}^\gamma|^{p_2}\right)^{q_2/p_2}\right)^{1/q_2}=:M.
$$
Let $n=\tilde am^{2/p_2}k^{2/q_2}r^{1-2/p_2}l^{1-2/q_2}$. Applying the
inequality $\xi \eta \le |\xi| ^{q_2'}+|\eta |^{q_2}$ for $\xi =r
^{\frac{q_2-1}{p_2}}l^{\frac{q_2-1}{q_2}}$ and
$$
\eta =\left(|G|^{-1}\sum\limits _{\gamma\in G}\sum \limits
_{j=1}^k \left(\sum \limits
_{i=1}^m|y_{i,j}^\gamma|^{p_2}\right)^{q_2/p_2}\right)^{1/q_2},
$$
we get that
$$
M=\tilde a^{\frac{1}{2}}r
^{1-\frac{1}{p_2}}l^{1-\frac{1}{q_2}}\left(|G|^{-1}\sum \limits
_{\gamma\in G}\sum \limits _{j=1}^k \left(\sum \limits
_{i=1}^m|y_{i,j}^\gamma|^{p_2}\right)^{q_2/p_2}\right)^{1/q_2}\le
$$
$$
\le \tilde a^{\frac{1}{2}}r
^{1-\frac{q_2}{p_2}}\left(r^{\frac{q_2}{p_2}}l+|G|^{-1}\sum
\limits _{\gamma\in G}\sum \limits _{j=1}^k \left(\sum \limits
_{i=1}^m|y_{i,j}^\gamma|^{p_2}\right)^{q_2/p_2}\right).
$$
Combining this estimate and (\ref{low1}), we have
$$
\max _{\gamma\in G}\| \gamma(e)-y^\gamma\|
_{l^{m,k}_{p_2,q_2}}^{q_2}\ge \frac{lr
^{\frac{q_2}{p_2}}}{4}+\tilde c_1(p_2, \, q_2)|G|^{-1}\sum \limits
_{\gamma\in G}\sum \limits _{j=1}^k\left(\sum \limits _{i=1}^m
|y^\gamma_{i,j}|^{p_2}\right)^{q_2/p_2}-
$$
$$
-\frac{q_2}{2}\tilde a^{\frac{1}{2}}\left(r^{\frac{q_2}{p_2}}l+|G|^{-1}\sum
\limits _{\gamma\in G}\sum \limits _{j=1}^k \left(\sum \limits
_{i=1}^m|y_{i,j}^\gamma|^{p_2}\right)^{q_2/p_2}\right).
$$
Taking $\tilde a$ enough small, we get (\ref{low_est_polytope}). It remains
to put $a(p_2, \, q_2)=\min _{2\le p\le p_2, \, 2\le q\le q_2}\tilde a(p, \, q)$.

Let us estimate Kolmogorov widths of balls $B^{m,k}_{p_1,q_1}$.
Since $B^{m,k}_{p_1,q_1}\supset r
^{-1/p_1}l^{-1/q_1}V^{m,k}_{r,l}$, it follows that
\begin{align}
\label{dnbmkpq} d_n(B^{m,k}_{p_1,q_1}, \,
l^{m,k}_{p_2,q_2})\underset{p_2,q_2}{\gtrsim} r
^{\frac{1}{p_2}-\frac{1}{p_1}}l^{\frac{1}{q_2}-\frac{1}{q_1}}
\end{align}
for $n\le a m^{2/p_2}k^{2/q_2}r ^{1-2/p_2}l^{1-2/q_2}$, $a=a(p_2, \, q_2)$.
\begin{enumerate}
\item $n\le am^{2/p_2}k^{2/q_2}$. Taking $r =l=1$, we
have $$d_n(B^{m,k}_{p_1,q_1}, \,
l^{m,k}_{p_2,q_2})\underset{p_2,q_2}{\gtrsim}1.$$
\item $am^{2/p_2}k^{2/q_2}\le n\le amk$.
First choose $\tilde p\in [2, \, p_2]$ and $\tilde q\in [2, \,
q_2]$ such that $n=a(p_2, \, q_2)m^{2/\tilde p}k^{2/\tilde q}\le
a(\tilde p, \, \tilde q)m^{2/\tilde p}k^{2/\tilde q}$. Then
$$
d_n(B^{m,k}_{p_1,q_1}, \, l^{m,k}_{p_2,q_2})\ge d_n(V_{1,1}^{m,k},
\, l^{m,k}_{p_2,q_2})\ge m^{\frac{1}{p_2}-\frac{1}{\tilde
p}}k^{\frac{1}{q_2}-\frac{1}{\tilde q}}d_n(V_{1,1}^{m,k}, \,
l_{\tilde p,\tilde
q}^{m,k})\underset{p_2,q_2}{\gtrsim}n^{-1/2}m^{1/p_2}k^{1/q_2}
$$
(in the last order inequality we used that $\tilde p\in [2, \,
p_2]$, $\tilde q\in [2, \, q_2]$ and $b(\tilde p, \, \tilde q)$ is
continuous).

This estimate can be improved for $\lambda({\bf p})<1$ or
$\lambda({\bf q})<1$. Consider the following cases.
\begin{enumerate}
\item $n\le akm^{2/p_2}$, $q_1>2$. Let $r =1$,
$$
l=\left\lceil\left(a^{-1}nm^{-2/p_2}k^{-2/q_2}\right)^{\frac{1}{1-2/q_2}}\right\rceil.
$$
Then $1\le
\left(a^{-1}nm^{-2/p_2}k^{-2/q_2}\right)^{\frac{1}{1-2/q_2}}\le
k$, $1\le l\le k$, and
\begin{align}
\label{dnlq} d_n(B^{m,k}_{p_1,q_1}, \,
l^{m,k}_{p_2,q_2})\stackrel{(\ref{dnbmkpq})}{\underset{p_2,q_2}{\gtrsim}}\left(n^{-1/2}m^{1/p_2}k^{1/q_2}
\right)^{\lambda(\bf q)}.
\end{align}
\item $n\le amk^{2/q_2}$, $p_1>2$. Let $l=1$,
$$
r=\left\lceil\left(a^{-1}nm^{-2/p_2}k^{-2/q_2}\right)^{\frac{1}{1-2/p_2}}\right\rceil.
$$
Then $1\le
\left(a^{-1}nm^{-2/p_2}k^{-2/q_2}\right)^{\frac{1}{1-2/p_2}}\le
m$, $1\le r\le m$, and
\begin{align}
\label{dnlp} d_n(B^{m,k}_{p_1,q_1}, \,
l^{m,k}_{p_2,q_2})\stackrel{(\ref{dnbmkpq})}{\underset{p_2,q_2}{\gtrsim}}\left(n^{-1/2}m^{1/p_2}k^{1/q_2}
\right)^{\lambda(\bf p)}.
\end{align}
\item $amk\ge n>akm^{2/p_2}$. Let $p_1>2$. Take $l=k$,
$$
r =\left\lceil \left(a^{-1}nk^{-1}m^{-2/p_2}\right)
^{\frac{1}{1-2/p_2}}\right\rceil.
$$
Then $1\le \left(a^{-1}nk^{-1}m^{-2/p_2}\right)
^{\frac{1}{1-2/p_2}} \le m$, $1\le r\le m$, and we get
$$
d_n(B^{m,k}_{p_1,q_1}, \,
l^{m,k}_{p_2,q_2})\stackrel{(\ref{dnbmkpq})}{\underset{p_2,q_2}{\gtrsim}}
k^{\frac{1}{q_2}-\frac{1}{q_1}}\left(n^{-1/2}k^{1/2}
m^{1/p_2}\right)^{\lambda({\bf p})}.
$$
Let $p_1\le 2$ and $q_1>2$. Choose $\tilde p\in [2, \, p_2]$ such
that $n=a(p_2, \, q_2)km^{\frac{2}{\tilde p}}\le a(\tilde p, \,
q_2)km^{\frac{2}{\tilde p}}$. Then
$$
d_n(B^{m,k}_{p_1,q_1}, \, l^{m,k}_{p_2,q_2})\ge
m^{\frac{1}{p_2}-\frac{1}{\tilde p}}d_n(B^{m,k}_{p_1,q_1}, \,
l^{m,k}_{\tilde p,q_2})
\stackrel{(\ref{dnlq})}{\underset{p_2,q_2}{\gtrsim}}
m^{\frac{1}{p_2}-\frac{1}{\tilde p}}\left(n^{-1/2}m^{1/\tilde
p}k^{1/q_2} \right)^{\lambda(\bf
q)}\underset{p_2,q_2}{\gtrsim}$$$$\gtrsim
m^{\frac{1}{p_2}}\left(n^{\frac 12}k^{-\frac
12}\right)^{\lambda({\bf q})-1}n^{-\frac{\lambda({\bf
q})}{2}}k^{\frac{\lambda({\bf q})}{q_2}}=n^{-\frac
12}m^{\frac{1}{p_2}}k^{\frac 12+\frac{1}{q_2}-\frac{1}{q_1}}.
$$
\item $amk\ge n>amk^{2/q_2}$. Let $q_1>2$. Take $r=m$,
$$
l=\left\lceil \left(a^{-1}nm^{-1}k^{-2/q_2}\right)
^{\frac{1}{1-2/q_2}}\right\rceil.
$$
Then $1\le \left(a^{-1}nm^{-1}k^{-2/q_2}\right)
^{\frac{1}{1-2/q_2}} \le k$, $1\le l\le k$, and we have
$$
d_n(B^{m,k}_{p_1,q_1}, \,
l^{m,k}_{p_2,q_2})\stackrel{(\ref{dnbmkpq})}{\underset{p_2,q_2}{\gtrsim}}
m^{\frac{1}{p_2}-\frac{1}{p_1}}\left(n^{-1/2}m^{1/2}
k^{1/q_2}\right)^{\lambda({\bf q})}.
$$
Let $q_1\le 2$ and $p_1>2$. Choose $\tilde q\in [2, \, q_2]$ such
that $n=a(p_2, \, q_2)mk^{\frac{2}{\tilde q}}\le a(p_2, \, \tilde
q)mk^{\frac{2}{\tilde q}}$. Then
$$
d_n(B^{m,k}_{p_1,q_1}, \, l^{m,k}_{p_2,q_2})\ge
k^{\frac{1}{q_2}-\frac{1}{\tilde q}}d_n(B^{m,k}_{p_1,q_1}, \,
l^{m,k}_{p_2,\tilde q})
\stackrel{(\ref{dnlp})}{\underset{p_2,q_2}{\gtrsim}}
k^{\frac{1}{q_2}-\frac{1}{\tilde q}}\left(n^{-1/2}m^{1/
p_2}k^{1/\tilde q} \right)^{\lambda(\bf
p)}\underset{p_2,q_2}{\gtrsim}$$$$\gtrsim
k^{\frac{1}{q_2}}\left(n^{\frac 12}m^{-\frac
12}\right)^{\lambda({\bf p})-1}n^{-\frac{\lambda({\bf
p})}{2}}m^{\frac{\lambda({\bf p})}{p_2}}=n^{-\frac
12}k^{\frac{1}{q_2}}m^{\frac 12+\frac{1}{p_2}-\frac{1}{p_1}}.
$$
\end{enumerate}
\end{enumerate}
This ends the proof.
\end{proof}
\section{Estimates for linear widths of finite-dimensional balls
in a mixed norm}

In this paragraph upper estimates for linear widths $\lambda
_n(B^{m,k}_{p_1,q_1}, \, l^{m,k}_{p_2,q_2})$ are established. In
order to do this we generalize arguments of the paper
\cite{bib_gluskin}. Let us give some notations and formulate
auxiliary assertions from this article.

Let $\nu\in \N$, $r>2$, $\lambda\ge 1$,
$\mu_r(\lambda)=\left[\lambda^{\frac{1}{1/2-1/r}}\right]$. For any
vector $x=(x_1, \, \dots, \, x_\nu)\in \R^\nu$ we denote by ${\rm
supp}\, x=\{j\in\overline{1, \, \nu}:\, x_j\ne 0\}$, $(x^*_1, \,
\dots, \, x^*_\nu)$ the non-increasing rearrangement of $|x_j|$,
$1\le j\le \nu$. The constructions in section 4 and the proof of
Lemma 4 of the paper \cite{bib_gluskin} yield the following
assertions.
\begin{Lem}
\label{lemma1} There exist a set ${\cal E}={\cal
E}_{r,\lambda,\nu}\subset B_2^\nu$ and a number $c(r)>0$ such that
${\cal E}=-{\cal E}$ and
\begin{enumerate}
\item for any $y\in {\cal E}$ we have $|{\rm supp}\, y|\le [\lambda^2
\nu ^{2/r}]+1$;
\item $|{\cal E}|\le
\exp\left(c(r)\lambda^2\nu^{2/r}\right)$;
\item for any $x\in \R^\nu$ there exists a vector $y\in {\cal E}$
such that $$\sup _{\mu_r(\lambda)<j\le \nu}\lambda j^{1/r}x^*_j\le
\sum \limits _{i=1}^\nu x_iy_i.$$
\end{enumerate}
\end{Lem}
\begin{Lem}
\label{lemma2} There exist a set ${\cal E}'={\cal
E}'_{r,\lambda,\nu}\subset B_2^\nu$ and a number $c(r)>0$ such
that ${\cal E}'=-{\cal E}'$ and
\begin{enumerate}
\item for any $y\in {\cal E}'$ we have $|{\rm supp}\, y|\le
\mu_r(\lambda)$;
\item $|{\cal E}'|\le\exp\left(c(r)\lambda^2\nu^{2/r}\right)$;
\item for any $x\in \R^\nu$ there exists a vector $y\in {\cal E}'$
such that $$\left(\sum \limits _{1\le j\le
\mu_r(\lambda)}|x^*_j|^2\right)^{1/2}\le 2\sum \limits _{j=1}^\nu
x_jy_j.$$
\end{enumerate}
\end{Lem}
Let $N\in \N$. Denote ${\cal F}_0(N)$ the family of
absolute convex polyhedrons which are contained in
$B_2^\nu$ and have not more than $2N$ vertices. We say that an absolute
convex body belongs to ${\cal F}(N)$, if
there exists a set $K\in {\cal F}_0(N)$ such that $V\subset K$.
Let $W$ be a bounded closed absolute convex body.
Denote $\R^\nu_W$ the linear space $\R^\nu$ with the norm $\|\cdot\|_W$
such that $W=\{x\in \R^\nu:\|x\|_W\le 1\}$. By $(\R^\nu_W)^*$
denote the dual space. The following lemma is proved in \cite{bib_gluskin}.
\begin{Lem}
\label{lemma3} There exists an absolute constant $c_0>0$ such that
for any $\theta>0$, $N\in \N$ such that
$N<\frac{1}{16}\exp\left(\frac{\theta^2\nu}{4}\right)$, for any
sets $V$, $W\in {\cal F}(N)$ and for any $n\in \{0, \, \dots, \,
\nu\}$ we have
$$
\lambda_n(V, \, (\R^\nu_W)^*)\le c_0\left(\theta\sqrt{\frac{\nu
-n}{n}}+\theta ^2\frac{\nu}{n}\right).
$$
\end{Lem}

The following duality relation for linear widths was proved in \cite{bib_ismag}.
Let $B$, $C\subset \R^\nu$ be absolute convex closed bounded
bodies, let $B^\circ$, $C^\circ$ be their polars. Then
\begin{align}
\label{dual}
\lambda_n(B, \, \R^\nu_C)=\lambda_n(C^\circ, \, \R^\nu_{B^\circ}).
\end{align}

Now we formulate the main result of this section.
\begin{Trm}
\label{lin_fin} Let $m$, $k\in \N$, $1<p_1\le 2\le p_2<\infty$,
$1<q_1\le 2\le q_2<\infty$, $n\in \N$. Then
$$
\lambda_n(B^{m,k}_{p_1,q_1}, \,
l^{m,k}_{p_2,q_2})\underset{p_1,p_2,q_1,q_2}{\lesssim}\min
\left\{n^{-\frac 12}m^{\max\left\{\frac{1}{p_2}, \,
\frac{1}{p_1'}\right\}}k^{\max \left\{\frac{1}{q_2}, \,
\frac{1}{q_1'}\right\}}, \, 1\right\}.
$$
Let $a$ be defined in Theorem \ref{fin_dim_mix_norm}, $n\le amk$.
If $\frac{1}{p_1}+\frac{1}{p_2}\ge 1$ and
$\frac{1}{q_1}+\frac{1}{q_2}\ge 1$, then
\begin{align}
\label{ln1}
\lambda _n(B^{m,k}_{p_1,q_1}, \,
l^{m,k}_{p_2,q_2})\underset{p_1,p_2,q_1,q_2}{\gtrsim}n^{-\frac 12}m^{\frac{1}{p_2}}
k^{\frac{1}{q_2}};
\end{align}
if $\frac{1}{p_1}+\frac{1}{p_2}\le 1$ and $\frac{1}{q_1}+\frac{1}{q_2}\le 1$,
then
\begin{align}
\label{ln2}
\lambda _n(B^{m,k}_{p_1,q_1}, \,
l^{m,k}_{p_2,q_2})\underset{p_1,p_2,q_1,q_2}{\gtrsim}n^{-\frac 12}m^{\frac{1}{p_1'}}
k^{\frac{1}{q_1'}}.
\end{align}
\end{Trm}
\begin{proof}
The estimate (\ref{ln1}) follows from Theorem
\ref{fin_dim_mix_norm} and the inequality $\lambda_n(M, \, X)\ge
d_n(M, \, X)$. The relation (\ref{ln2}) follows from (\ref{ln1})
and (\ref{dual}).

Since $p_1\le p_2$ and $q_1\le q_2$, we have $\lambda
_n(B^{m,k}_{p_1,q_1}, \, l^{m,k}_{p_2,q_2})\le 1$. Let us prove that
\begin{align}
\label{eq_lin} \lambda _n(B^{m,k}_{p_1,q_1}, \,
l^{m,k}_{p_2,q_2})\underset{p_1,p_2,q_1,q_2}{\lesssim}n^{-\frac
12}m^{\max\left\{\frac{1}{p_2}, \, \frac{1}{p_1'}\right\}}k^{\max
\left\{\frac{1}{q_2}, \, \frac{1}{q_1'}\right\}}.
\end{align}

Let $2\le p<\infty$, $2\le q<\infty$ and $n^{-\frac
12}m^{\frac{1}{p}} k^{\frac{1}{q}}<1$. We show that the following estimate
holds:
\begin{align}
\label{pprim} \lambda _n(B^{m,k}_{p',q'}, \,
l^{m,k}_{p,q})\underset{p,q}{\lesssim}n^{-\frac 12}m^{\frac
1p}k^{\frac 1q}
\end{align}
(then (\ref{eq_lin}) easily follows from(\ref{pprim})). In order
to do this we find constants $c_1(p, \, q)>0$ and $c_2(p, \, q)>0$
such that
\begin{align}
\label{c1pq}
c_1(p, \, q)B^{m,k}_{p',q'}\in {\cal F}(N),
\end{align}
where $N<\frac{1}{16}\exp \left(c_2(p, \,
q)m^{\frac{2}{p}}k^{\frac{2}{q}}\right)$, and apply Lemma
\ref{lemma3} for $\theta=2\sqrt{c_2(p, \, q)}m^{\frac 1p-\frac 12}
k^{\frac 1q-\frac12}$ and $V=W=c_1(p, \, q)B^{m,k}_{p',q'}$. Then
$\theta ^2\frac{mk}{n}=4c_2(p, \, q) m^{\frac 2p}k^{\frac
2q}n^{-1}\underset{p,q}{\lesssim}1$ and
$$
\lambda _n(B^{m,k}_{p',q'}, \,
l^{m,k}_{p,q})\underset{p,q}{\lesssim} \theta\sqrt{\frac{mk}{n}}+\theta ^2\frac{mk}{n}
\underset{p,q}{\lesssim}n^{-\frac 12}m^{\frac 1p}k^{\frac 1q}.
$$

Let us prove (\ref{c1pq}). For any $x=(x_1, \, \dots , \, x_m)\in \R^m$ and
$1\le l\le k$ we denote $I_l(x)=(x_{i,j})_{1\le i\le m, \, 1\le
j\le k}$, where $x_{i,j}=0$ if $j\ne l$, $x_{i,l}=x_i$, $1\le i\le
m$. Let
$$
{\cal E}_{1,1}=\left\{\sum \limits_{j\in {\rm supp}\, y}
y_jI_j(x_j), \; y=(y_1, \, \dots, \, y_k)\in {\cal
E}'_{2q,k^{\frac{1}{2q}},k}, \;\; x_j\in {\cal
E}'_{2p,m^{\frac{1}{2p}},m}, \;j\in {\rm supp}\, y\right\},
$$
$$
{\cal E}_{1,2}=\left\{\sum \limits_{j\in {\rm supp}\, y}
y_jI_j(x_j), \; y=(y_1, \, \dots, \, y_k)\in {\cal
E}'_{2q,k^{\frac{1}{2q}},k}, \;\; x_j\in {\cal
E}_{2p,m^{\frac{1}{2p}},m}, \;j\in {\rm supp}\, y\right\},
$$
$$
{\cal E}_{2,1}=\left\{\sum \limits_{j\in {\rm supp}\, y}
y_jI_j(x_j), \; y=(y_1, \, \dots, \, y_k)\in {\cal
E}_{2q,k^{\frac{1}{2q}},k}, \;\; x_j\in {\cal
E}'_{2p,m^{\frac{1}{2p}},m}, \;j\in {\rm supp}\, y\right\},
$$
$$
{\cal E}_{2,2}=\left\{\sum \limits_{j\in {\rm supp}\, y}
y_jI_j(x_j), \; y=(y_1, \, \dots, \, y_k)\in {\cal
E}_{2q,k^{\frac{1}{2q}},k}, \;\; x_j\in {\cal
E}_{2p,m^{\frac{1}{2p}},m}, \;j\in {\rm supp}\, y\right\},
$$
$$
K={\rm conv}\, \left({\cal E}_{1,1}\cup {\cal E}_{1,2}\cup {\cal E}_{2,1}\cup
{\cal E}_{2,2}\right).
$$
Then $K$ is an absolute convex polyhedron. Lemmas \ref{lemma1} and
\ref{lemma2} yield that
$$
\max\left\{{\rm card}\, {\cal E}_{1,1}, \, {\rm card}\, {\cal
E}_{1,2}\right\}\le \exp\left(c(2q)k^{2/q}\right)
\left(\exp\left(c(2p)m^{2/p}\right)\right)^{\mu_{2q}\left(k^{1/2q}\right)}=$$
$$
=\exp\left(c(2q)k^{2/q}+c(2p)m^{2/p}\mu_{2q}(k^{1/2q})\right),
$$
$$
\max\left\{{\rm card}\, {\cal E}_{2,1}, \, {\rm card}\, {\cal
E}_{2,2}\right\} \le \exp\left(c(2q)k^{2/q}\right)
\left(\exp\left(c(2p)m^{2/p}\right)\right)^{[k^{2/q}]+1}=
$$
$$
=\exp\left(c(2q)k^{2/q}+c(2p)m^{2/p}(1+[k^{2/q}])\right);
$$
since $\mu_{2q}(k^{1/2q})\le k^{\frac{1}{q-1}}\le k^{2/q}$
(the last inequality follows from the relation $q\ge 2$), then
there exists a constant $c_3(p, \, q)$ such that
$$
\max\left\{{\rm card}\, {\cal E}_{1,1}, \, {\rm card}\, {\cal
E}_{1,2}, \, {\rm card}\, {\cal E}_{2,1}, \, {\rm card}\, {\cal
E}_{2,2}\right\}\le \exp\left(c_3(p, \, q)m^{2/p}k^{2/q}\right).
$$
Therefore, there exists $c_2(p, \, q)>0$ such that $K\in {\cal
F}_0(N)$, where $N<\frac{1}{16}\exp\left(c_2(p, \,
q)m^{\frac{2}{p}}k^{\frac{2}{q}}\right)$.

Let us show that there exists $c_1(p, \, q)>0$ such that
$c_1(p, \, q)B^{m,k}_{p',q'}\subset K$. It is enough to check the inequality
$$
\max \{\langle x, \, y\rangle :\; y\in
B^{m,k}_{p',q'}\}\underset{p,q}{\lesssim} \max \{\langle x, \,
y\rangle :\; y\in K\},\;\; x\in \R^{mk},
$$
that is
\begin{align}
\label{xmlk} \|x\|_{l^{m,k}_{p,q}}\underset{p,q}{\lesssim} \max
\{\langle x, \, y\rangle :\; y\in K\}.
\end{align}

Let $x=(x_{i,j})_{1\le i\le m, \, 1\le j\le k}\in \R^{mk}$. Denote
$a_j=\left(\sum \limits _{i=1}^m |x_{i,j}|^p\right)^{1/p}$,
$(x^*_{1,j}, \, \dots, \, x^*_{m,j})$ the non-increasing
rearrangement of $|x_{i,j}|$, $1\le i\le m$, and $(a^*_1, \,
\dots, \, a^*_k)$ the non-increasing rearrangement of $|a_j|$,
$1\le j\le k$.

We claim that
\begin{align}
\label{nxmkpq}
\|x\|_{l^{m,k}_{p,q}}\underset{q}{\lesssim}\left(\sum \limits
_{j\le \mu_{2q}(k^{1/2q})}|a^*_j|^2\right)^{1/2}+ \max
_{j>\mu_{2q}(k^{1/2q})}k^{1/2q} j^{1/2q}a^*_j.
\end{align}

Indeed, there exists an absolute constant $c_4>0$ such that
$$
\sum \limits _{\mu_{2q}(k^{1/2q})<j\le k} k^{-\frac 12}j^{-\frac
12}\le c_4.
$$
Therefore,
$$
\|x\|_{l_{p,q}^{m,k}}=\left(\sum \limits _{j=1}^k
|a^*_j|^q\right)^{1/q}\le \left(\sum \limits
_{j=1}^{\mu_{2q}(k^{1/2q})}|a^*_j|^2\right)^{1/2}+
$$
$$
+\left(\sum \limits _{j=\mu_{2q}(k^{1/2q})+1}^k
\left(k^{1/2q}j^{1/2q}a^*_j\right)^q k^{-\frac 12}j^{-\frac
12}\right)^{1/q}\underset{q}{\lesssim}
$$
$$
\lesssim \left(\sum \limits
_{j=1}^{\mu_{2q}(k^{1/2q})}|a^*_j|^2\right)^{1/2}+\max
_{j>\mu_{2q}(k^{1/2q})}k^{1/2q}j^{1/2q}a^*_j.
$$
By Lemmas \ref{lemma1} and \ref{lemma2}, there exist
$y^{(1)}=(y^{(1)}_j)_{1\le j\le k}\in {\cal E}'_{2q,k^{1/2q},k}$
and $y^{(2)}=(y^{(2)}_j)_{1\le j\le k}\in {\cal
E}_{2q,k^{1/2q},k}$ such that
\begin{align}
\label{11} \left(\sum \limits
_{j=1}^{\mu_{2q}(k^{1/2q})}|a^*_j|^2\right)^{1/2}\le 2\sum \limits
_{j=1}^k y_j^{(1)}a_j=2\sum \limits _{j=1}^k y_j^{(1)}\left(\sum
\limits _{i=1}^m |x_{i,j}^*|^p\right)^{1/p},
\end{align}
\begin{align}
\label{22} \max _{j>\mu_{2q}(k^{1/2q})}k^{1/2q}j^{1/2q}a_j^*\le
\sum \limits _{j=1}^k y_j^{(2)}a_j= \sum \limits _{j=1}^k
y_j^{(2)}\left(\sum \limits _{i=1}^m |x_{i,j}^*|^p\right)^{1/p}.
\end{align}
In the same way as (\ref{nxmkpq}) we can prove that for any $j=1, \, \dots, \, k$
$$
\left(\sum \limits _{i=1}^m
|x_{i,j}^*|^p\right)^{1/p}\underset{p}{\lesssim} \left(\sum
\limits _{i=1}^{\mu_{2p}(m^{1/2p})}
|x^*_{i,j}|^2\right)^{1/2}+\max
_{i>\mu_{2p}(m^{1/2p})}m^{1/2p}i^{1/2p}|x^*_{i,j}|.
$$
Therefore, there exist $z_j^{(1)}=(z^{(1)}_{i,j})_{1\le i\le m}\in
{\cal E}'_{2p,m^{1/2p},m}$ and $z_j^{(2)}=(z^{(2)}_{i,j})_{1\le
i\le m}\in {\cal E}_{2p,m^{1/2p},m}$ such that
\begin{align}
\label{33}
\left(\sum \limits _{i=1}^{\mu_{2p}(m^{1/2p})}
|x^*_{i,j}|^2\right)^{1/2}\le 2\sum \limits _{i=1}^m z^{(1)}_{i,j}
x_{i,j},
\end{align}
\begin{align}
\label{44}
\max _{i>\mu_{2p}(m^{1/2p})}m^{1/2p}i^{1/2p}|x^*_{i,j}|\le \sum
\limits _{i=1}^m z^{(2)}_{i,j} x_{i,j}.
\end{align}

Combining (\ref{nxmkpq}), (\ref{11}), (\ref{22}), (\ref{33}),
(\ref{44}) and the inequality $|\xi_1|+|\xi_2|\le 2\max \{|\xi_1|,
\, |\xi_2|\}$, we obtain that for any $x$ there exist $s\in \{1,
\, 2\}$ and $t\in \{1, \, 2\}$ such that
$$
\|x\|_{l^{m,k}_{p,q}}\underset{p,q}{\lesssim}\sum \limits _{j=1}^k
y_j^{(s)}\sum \limits _{i=1}^m (I_j(z_j^{(t)}))_i x_{i,j}=\langle
b, \, x\rangle,
$$
where $b\in {\cal E}_{s,t}$. This completes the proof of
(\ref{xmlk}).
\end{proof}
\section{Bounds for widths of weighted Besov classes}
Let $m\in \N$, $n\in \N\cap[2, \, +\infty)$, $N_n\in \N$, let
$\varphi _1^n, \, \dots , \, \varphi _m^n:\R_+ \rightarrow \R$ be
affine functions,
$$
\varphi _s^n(\xi)=\alpha _s\xi +\beta _{s,n},\;\;s=1, \, \dots, \,
m,\;\; \varphi _{m+1}^n(\xi)=+\infty, \;\; \xi \in \R_+,
$$
$$
\varphi _1^n(\xi)\le \dots \le\varphi _m^n(\xi), \;\; \;\; 0\le
\xi \le N_n, \; n\in \N\cap[2, \, +\infty).
$$
Denote
$$
E_s^n=\{(\xi , \, \eta)\in \R^2:\varphi _s^n(\xi)\le \eta <
\varphi _{s+1}^n(\xi), \; 0\le \xi \le N_n\}, \;\; 1\le s\le m,
$$
$$
E^n=\bigcup _{s=1}^m E_s^n=\{(\xi , \, \eta)\in \R^2:\varphi
_1^n(\xi)\le \eta < +\infty, \; 0\le \xi \le N_n\}.
$$

Let $\psi _n:E^n\rightarrow \R$, $\tilde \psi _n: E^n\rightarrow
\R$,
$$
\psi _n(\xi , \, \eta)=\lambda _s\xi +\mu _s\eta +\nu _{s,n}, \;\;
(\xi , \, \eta)\in E_s^n,
$$
$$
\tilde \psi _n(\xi , \, \eta)=\tilde \lambda _s\xi +\tilde \mu
_s\eta +\tilde\nu _{s,n}, \;\; (\xi , \, \eta)\in E_s^n.
$$
In addition, we suppose that functions $\psi_n$ are continuous.

Denote
$$
\Lambda =\left\{(\alpha _s)_{1\le s\le m}, \;
(\lambda _s)_{1\le s\le m}, \;
(\mu _s)_{1\le s\le m}\right\},
$$
$$
\tilde\Lambda =\left\{(\alpha _s)_{1\le s\le m}, \; (\tilde\lambda
_s)_{1\le s\le m}, \; (\tilde\mu _s)_{1\le s\le m}\right\},
$$
$\Psi =\{\psi _n, \; n\in \N\cap[2, \, +\infty)\}$, $\tilde \Psi
=\{\tilde \psi _n, \; n\in \N\cap[2, \, +\infty)\}$,
$$
V_n=\{(0, \, \varphi ^n_s(0)), \;(N_n, \, \varphi _s^n(N_n)): 1\le
s\le m\}.
$$

Let either $j_*^n=0$ for any $n\in \N\cap[2, \, +\infty)$ or
$j_*^n=N_n$ for any $n\in \N\cap[2, \, +\infty)$, and let
$S=\{s_-, \, s_-+1,\, \dots ,\, s_+-1, \, s_+\}\subset \{1, \,
\dots , \, m\}$ (this set can be empty in general), $l_*^n\in \R$,
\begin{align}
\label{sphinsln}
\{s\in \overline{1, \, m}:
\varphi ^n_s(j^n_*)=l_*^n\}=S,
\end{align}
$J=\{(j_*^n, \, l_*^n):n\in \N\cap[2, \, +\infty)\}$,
$\varepsilon
>0$. We say that $\tilde \Psi\in {\cal O}_{\varepsilon ,J}(\Psi)$
if
\begin{align}
\label{lstlse} |\lambda _s-\tilde\lambda _s|<\varepsilon , \;\;
|\mu _s -\tilde\mu _s|<\varepsilon , \;\; 1\le s\le m,
\end{align}
\begin{align}
\label{1622} |\psi _n(j, \, l)-\tilde \psi _n(j, \,
l)|<\varepsilon \log _2n, \;\;0\le j\le N_n, \;\;\varphi_1^n(j)\le
l\le \varphi ^n_m(j),
\end{align}
\begin{align}
\label{psinj} \psi _n(j_*^n, \, l_*^n)=\tilde \psi _n(j_*^n, \,
l_*^n)=\tilde \psi _n(j_*^n, \, l_*^n-0)\text{ for any }n\in
\N\cap [2, \, +\infty).
\end{align}

\begin{Lem}
\label{osn_lemma} Let $m\in \N$, $c>0$, $\gamma >0$ be fixed, let
$\mu _m<0$,
\begin{align}
\label{psinjl0}
\psi _n(j_*^n, \, l_*^n)=0, \; \; n\in \N\cap [2, \, +\infty),
\end{align}
and let $\tilde N_n\in \Z_+$,
\begin{align}
\label{nnleclog2nphin} N_n\le c\log _2 n, \;\; -c\log _2 n\le
\varphi ^n_1(\xi)\le\dots \le \varphi _m^n(\xi)\le c\log _2 n,
\;\; \xi \in [0, \, N_n],
\end{align}
\begin{align}
\label{psijllemglog2nvbjl} \psi _n(j, \, l)\le -\gamma \log _2 n,
\;\; (j, \, l)\in V_n\backslash (j_*^n, \, l_*^n)
\end{align}
for any $n\in \N\cap [2, \, +\infty)$. Let $\rho :[1, \,
+\infty)\rightarrow (0, \, +\infty)$ be an absolute continuous
function that satisfies (\ref{malo_izm331}).

Then there exists $\varepsilon _0=\varepsilon _0(\Lambda , \, c,
\, \gamma)>0$ such that for any $\tilde \Psi\in {\cal
O}_{\varepsilon _0,J}(\Psi)$ the following estimate holds:
$$
\sum \limits _{(j, \, l)\in \Z^2\cap E^n}2^{\tilde\psi _n(j, \,
l)}\rho(2^{j+\tilde N_n})
\overset{n}{\lesssim}\rho(2^{j_*^n+\tilde N_n}).
$$
\end{Lem}
\begin{proof}
First we prove that for any $\sigma >0$ there exists $M(\sigma)>0$
such that for any $n\in \N$, $j=0, \, \dots , \, N_n$ the
following inequalities hold:
\begin{align}
\label{rho2jsigma} \frac{\rho(2^{j+\tilde N_n})}{\rho(2^{\tilde
N_n})}\le M(\sigma)2^{\sigma j}, \;\;\; \frac{\rho(2^{j+\tilde
N_n})}{\rho(2^{N_n+\tilde N_n})}\le M(\sigma)2^{\sigma(N_n-j)}.
\end{align}
Indeed, there exists $\xi _\sigma \ge 1$ such that the function
$y^{\sigma}\rho(y)$ increases on $[\xi _\sigma , \, +\infty)$ and
the function $y^{-\sigma}\rho(y)$ decreases on $[\xi _\sigma , \,
+\infty)$. Since the function $\rho$ is continuous and positive,
then there exists $M_1(\sigma)>0$ such that
$\frac{\rho(y_1)}{\rho(y_2)}\le M_1(\sigma)$ for any $y_1$,
$y_2\in [1, \, \xi _\sigma]$.

Let $2^{\tilde N_n}\ge \xi _\sigma$. Then
$$
\frac{\rho(2^{j+\tilde N_n})}{\rho(2^{\tilde
N_n})}=\frac{\rho(2^{j+\tilde N_n})2^{-\sigma(j+\tilde
N_n)}}{\rho(2^{\tilde N_n})2^{-\sigma \tilde N_n}}\cdot 2^{\sigma
j}\le 2^{\sigma j}.
$$
Let $2^{\tilde N_n}<\xi_\sigma$. If $2^{j+\tilde N_n}\le \xi
_\sigma$, then $\frac{\rho(2^{j+\tilde N_n})}{\rho(2^{\tilde
N_n})}\le M_1(\sigma)$. If $2^{j+\tilde N_n}>\xi_\sigma$, then
$$
\frac{\rho(2^{j+\tilde N_n})}{\rho(2^{\tilde N_n})}\le M_1(\sigma)
\frac{\rho(2^{j+\tilde N_n})}{\rho(\xi_\sigma)}=$$$$=M_1(\sigma)
\frac{\rho(2^{j+\tilde N_n})2^{-\sigma(j+\tilde
N_n)}}{\rho(\xi_\sigma)\xi_\sigma^{-\sigma}}\cdot 2^{\sigma
j}\xi_\sigma ^{-\sigma}\cdot 2^{\sigma\tilde N_n}\le
M_1(\sigma)2^{\sigma j}.
$$

If $2^{\tilde N_n+j}\ge \xi_\sigma$, then
$$
\frac{\rho(2^{j+\tilde N_n})}{\rho(2^{N_n+\tilde
N_n})}=\frac{2^{\sigma (j+\tilde N_n)}\rho(2^{j+\tilde N_n})}{
2^{\sigma (N_n+\tilde N_n)}\rho(2^{N_n+\tilde N_n})}\cdot
2^{\sigma (N_n-j)}\le 2^{\sigma (N_n-j)}.
$$
Let $2^{\tilde N_n+j}<\xi_\sigma$. If in addition $2^{N_n+\tilde
N_n}\le\xi _\sigma$, then $\frac{\rho(2^{j+\tilde
N_n})}{\rho(2^{N_n+\tilde N_n})} \le M_1(\sigma)$. If
$2^{N_n+\tilde N_n}>\xi _\sigma$, then
$$
\frac{\rho(2^{j+\tilde N_n})}{\rho(2^{N_n+\tilde N_n})}\le
M_1(\sigma) \frac{\rho(\xi _\sigma)}{\rho(2^{N_n+\tilde
N_n})}=M_1(\sigma) \frac{\xi _\sigma ^\sigma\rho(\xi
_\sigma)}{2^{\sigma (N_n+\tilde N_n)}\rho(2^{N_n+\tilde N_n})}
\cdot 2^{\sigma (N_n+\tilde N_n)}\xi _\sigma ^{-\sigma}\le
$$
$$
\le M_1(\sigma)\cdot 2^{\sigma (N_n+\tilde N_n)}\cdot
2^{-\sigma(\tilde N_n+j)}=M_1(\sigma)\cdot 2^{\sigma(N_n-j)}.
$$
This completes the proof of (\ref{rho2jsigma}).

Let $s_--1\le s\le s_+$, $0\le j\le N_n$, $\varphi ^n_s(j)\le
l<\varphi ^n_{s+1}(j)$. Denote
$$
t(s)=\left\{ \begin{array}{l} s,\text{ if }\tilde \mu_s\le 0 \text{ and } s_-\le s<s_+,\\
s+1,\text{ if }\tilde \mu_s>0\text{ and } s_-\le s<s_+,\\ s_+,
\text{ if }s=s_+,
\\ s_-, \text{ if }s=s_--1,\end{array}\right.
$$
$$
A=\max\{|\alpha_s|:s_-\le s\le s_+\},
$$
\begin{align}
\label{est_eps} \varepsilon _0=\min \left\{\frac{|\mu _m|}{2}, \,
\frac{\gamma}{4(c+1)}, \; \frac{|\mu _s|}{2}, \; \frac{|\lambda
_s+\mu _s\alpha _{t(s)}|}{8(A+1)}, \; s_- -1\le s\le s_+\right\}.
\end{align}

Let us obtain an upper estimate for the sum
$$
\Sigma _s^n=\sum \limits _{(j, \, l)\in \Z^2\cap
E^n_s}2^{\tilde\psi _n(j, \, l)}\rho(2^{j+\tilde N_n}), \;\; 1\le
s\le m.
$$
Let $S\ne \varnothing$. Since (\ref{psinjl0}) holds and functions
$\psi_n$ are continuous, we have
$$
\psi_n(j, \, \varphi ^n_{t(s)}(j))=
(\lambda_s+\mu_s\alpha_{t(s)})(j-j_*^n), \;\;\;s\in
S\cup(\{s_--1\}\cap \N).
$$
This together with (\ref{psijllemglog2nvbjl}) yield that
\begin{align}
\label{lamsl0} \lambda _s+ \mu _s\alpha _{t(s)}<0, \text{ if }
j^n_*=0, \;\;\;\;\lambda _s+ \mu _s\alpha _{t(s)}>0, \text{ if }
j^n_*=N_n.
\end{align}
Therefore, $\varepsilon_0>0$.

First we consider cases $s=s_+$ and $s=s_--1$. It follows from
(\ref{psinjl0}), from the condition $\mu_m<0$ and from
(\ref{psijllemglog2nvbjl}) that $\mu _{s_- -1}>0$ (if $s_-\ge 2$),
$\mu _{s_+}<0$. Then for any $(j, \, l)\in E^n_{s_+}$ we have
$$
\tilde \psi _n(j, \, l)=\tilde \lambda _{s_+} j +\tilde \mu _{s_+}
l +\tilde\nu _{s_+,n}=\tilde \lambda _{s_+}j+\tilde \mu _{s_+}\varphi ^n_{s_+}(j)+\tilde\nu _{s_+,n}+
\tilde \mu _{s_+}(l-\varphi ^n_{s_+}(j))\stackrel{(\ref{est_eps})}{\le}
$$
$$
\le \tilde \lambda _{s_+}j+\tilde \mu _{s_+}\varphi ^n_{s_+}(j)+\tilde\nu _{s_+,n}+
\mu _{s_+}(l-\varphi ^n_{s_+}(j))+\frac{|\mu _{s_+}|}{2}(l-\varphi ^n_{s_+}(j))=
$$
$$
=\tilde \lambda _{s_+} j +\tilde \mu
_{s_+} \varphi _{s_+}^n(j)+\tilde\nu _{s_+,n}-\frac{|\mu _{s_+}|}{2}(l-\varphi
_{s_+}^n(j)).
$$
Hence,
$$
\sum \limits _{(j, \, l)\in \Z^2\cap E^n_{s_+}}2^{\tilde\psi _n(j,
\, l)}\rho(2^{j+\tilde N_n}) \underset{\mu _{s_+}}{\lesssim} \sum
\limits _{j=0}^{N_n} 2^{\tilde \lambda _{s_+} j+ \tilde \mu
_{s_+}\varphi _{s_+}^n(j)+\tilde\nu _{s_+,n}}\rho(2^{j+\tilde
N_n}).
$$
Similarly, we obtain that for $s_-\ge 2$
$$
\sum \limits _{(j, \, l)\in \Z^2\cap E^n_{s_--1}}2^{\tilde\psi
_n(j, \, l)}\rho(2^{j+\tilde N_n}) \underset{\mu
_{s_--1}}{\lesssim} \sum \limits _{j=0}^{N_n} 2^{\tilde \lambda
_{s_--1} j+ \tilde \mu _{s_--1}\varphi _{s_-}^n(j)+\tilde\nu
_{s_--1,n}}\rho(2^{j+\tilde N_n}).
$$

Since the functions $\tilde \psi _n|_{E^n_{s_+}}$, $\tilde \psi
_n|_{E^n_{s_--1}}$ and $\varphi ^n_{s_\pm}$ are affine, then
conditions of Lemma yield
$$
\tilde \lambda _{s_+} j+\tilde \mu _{s_+}\varphi _{s_+}^n(j)+
\tilde\nu _{s_+,n}=\tilde \psi _n(j, \, \varphi _{s_+}^n(j))
\stackrel{(\ref{sphinsln}), (\ref{psinj})}{=}
$$
$$
=\tilde \psi _n(j, \, \varphi _{s_+}^n(j))- \tilde \psi _n(j_*^n,
\, \varphi _{s_+}^n(j_*^n))+ \psi _n(j_*^n, \,
l^n_*)\stackrel{(\ref{psinjl0})}{=}
$$
$$
=\tilde \psi _n(j, \, \varphi _{s_+}^n(j))- \tilde \psi _n(j_*^n,
\, \varphi _{s_+}^n(j_*^n))= (\tilde \lambda _{s_+}+\tilde \mu
_{s_+}\alpha _{s_+})(j-j_*^n)\stackrel{(\ref{lstlse})}{\le}
$$
$$
\le(\lambda _{s_+}+\mu _{s_+}\alpha
_{s_+})(j-j_*^n)+\varepsilon_0(1+A)|j-j_*^n|,
$$
$$
\tilde \lambda _{s_--1} j+\tilde \mu _{s_--1}\varphi _{s_-}^n(j)+
\tilde\nu _{s_--1,n}=\tilde \psi _n(j, \, \varphi _{s_-}^n(j)-0)
\stackrel{(\ref{sphinsln}), (\ref{psinj})}{=}
$$
$$
=\tilde \psi _n(j, \, \varphi _{s_-}^n(j)-0)- \tilde \psi
_n(j_*^n, \, \varphi _{s_-}^n(j_*^n)-0)+ \psi _n(j_*^n, \,
l^n_*)\stackrel{(\ref{psinjl0})}{=}
$$
$$
=\tilde \psi _n(j, \, \varphi _{s_-}^n(j)-0)- \tilde \psi
_n(j_*^n, \, \varphi _{s_-}^n(j_*^n)-0)= (\tilde \lambda
_{s_--1}+\tilde \mu _{s_--1}\alpha
_{s_-})(j-j_*^n)\stackrel{(\ref{lstlse})}{\le}
$$
$$
\le(\lambda _{s_--1}+\mu _{s_--1}\alpha
_{s_-})(j-j_*^n)+\varepsilon_0(1+A)|j-j_*^n|.
$$

By (\ref{est_eps}) and (\ref{lamsl0}),
\begin{align}
\label{kjhbcrjhb53627} 2^{\tilde \lambda _{s_+} j+\tilde \mu
_{s_+}\varphi _{s_+}^n(j) +\tilde\nu _{s_+,n}}\le
2^{-\frac{1}{2}|(\lambda _{s_+}+\mu _{s_+}\alpha
_{s_+})(j-j_*^n)|},
\end{align}
\begin{align}
\label{kjhbcrjhb53628} 2^{\tilde \lambda _{s_--1} j+\tilde \mu
_{s_--1}\varphi _{s_-}^n(j) +\tilde\nu _{s_--1,n}}\le
2^{-\frac{1}{2}|(\lambda _{s_--1}+\mu _{s_--1}\alpha
_{s_-})(j-j_*^n)|}.
\end{align}

Let $j_*^n=0$. Then
$$
\sum \limits _{j=0}^{N_n} 2^{\tilde \lambda _{s_+} j+ \tilde \mu
_{s_+}\varphi _{s_+}^n(j)+\tilde\nu _{s_+,n}}\rho(2^{j+\tilde
N_n}) \stackrel{(\ref{rho2jsigma}),(\ref{kjhbcrjhb53627})}
{\le}$$$$\le M(\varepsilon_0)\rho(2^{\tilde N_n})\sum \limits
_{j=0}^{N_n} 2^{-\frac{1}{2}|(\lambda _{s_+}+\mu _{s_+}\alpha
_{s_+})j|}\cdot 2^{\varepsilon_0 j}
\stackrel{(\ref{est_eps})}{\underset{\Lambda
,c,\gamma}{\lesssim}}\rho(2^{\tilde N_n}),
$$
$$
\sum \limits _{j=0}^{N_n} 2^{\tilde \lambda _{s_--1} j+ \tilde \mu
_{s_--1}\varphi _{s_-}^n(j)+\tilde\nu _{s_--1,n}}\rho(2^{j+\tilde
N_n}) \stackrel{(\ref{rho2jsigma}),(\ref{kjhbcrjhb53628})}
{\le}$$$$\le M(\varepsilon_0)\rho(2^{\tilde N_n})\sum \limits
_{j=0}^{N_n} 2^{-\frac{1}{2}|(\lambda _{s_--1}+\mu _{s_--1}\alpha
_{s_-})j|}\cdot 2^{\varepsilon_0 j}
\stackrel{(\ref{est_eps})}{\underset{\Lambda
,c,\gamma}{\lesssim}}\rho(2^{\tilde N_n}).
$$
If $j_*^n=N_n$, then
$$
\sum \limits _{j=0}^{N_n} 2^{\tilde \lambda _{s_+} j+ \tilde \mu
_{s_+}\varphi _{s_+}^n(j)+\tilde\nu _{s_+,n}}\rho(2^{j+\tilde
N_n}) \stackrel{(\ref{rho2jsigma}),(\ref{kjhbcrjhb53627})} {\le}
$$
$$
\le M(\varepsilon_0)\rho(2^{N_n+\tilde N_n}) \sum \limits
_{j=0}^{N_n}2^{-\frac{1}{2}|(\lambda _{s_+}+\mu _{s_+} \alpha
_{s_+})(j-N_n)|}\cdot2^{\varepsilon _0(N_n-j)}
\stackrel{(\ref{est_eps})}{\underset{\Lambda
,c,\gamma}{\lesssim}}\rho(2^{N_n+\tilde N_n}),
$$
$$
\sum \limits _{j=0}^{N_n} 2^{\tilde \lambda _{s_--1} j+ \tilde \mu
_{s_--1}\varphi _{s_-}^n(j)+\tilde\nu _{s_--1,n}}\rho(2^{j+\tilde
N_n}) \stackrel{(\ref{rho2jsigma}),(\ref{kjhbcrjhb53628})}{\le}
$$
$$
\le M(\varepsilon_0)\rho(2^{N_n+\tilde N_n}) \sum \limits
_{j=0}^{N_n}2^{-\frac{1}{2}|(\lambda _{s_--1}+\mu _{s_--1} \alpha
_{s_-})(j-N_n)|}\cdot 2^{\varepsilon _0(N_n-j)}
\stackrel{(\ref{est_eps})}{\underset{\Lambda
,c,\gamma}{\lesssim}}\rho(2^{N_n+\tilde N_n}).
$$
Let $s_-\le s<s_+$. Since $\varphi
^n_{s+1}(j_*^n)\stackrel{(\ref{sphinsln})}{=} \varphi
^n_s(j_*^n)$, then
\begin{align}
\label{phins1jphins1jjs} \varphi ^n_{s+1}(j)-\varphi ^n_s(j)=
\alpha _{s+1}j+\beta _{s+1,n}-\alpha _s j -\beta _{s,n}=(\alpha
_{s+1}-\alpha _s)(j-j_*^n).
\end{align}
It follows from (\ref{sphinsln}), (\ref{psinj}) and
(\ref{psinjl0}) that $\tilde \lambda _s j_*^n+\tilde \mu _s\varphi
_{t(s)}^n(j_*^n) +\tilde\nu _{s,n}=\psi_n(j_*^n, \, l_*^n)=0$.
Therefore,
$$
\tilde \lambda _s j+\tilde \mu _s\varphi _{t(s)}^n(j) +\tilde\nu
_{s,n}=(\tilde\lambda_s+\tilde\mu_s\alpha_{t(s)})(j-j_*^n)\stackrel{(\ref{lstlse})}{\le}
$$
$$
\le (\lambda_s+\mu_s\alpha_{t(s)})(j-j_*^n)+
\varepsilon_0(1+A)|j-j_*^n|.
$$
By (\ref{est_eps}) and (\ref{lamsl0}),
\begin{align}
\label{akjhbcrjhb53627} 2^{\tilde \lambda _s j+\tilde \mu
_s\varphi _{t(s)}^n(j) +\tilde\nu _{s,n}}(|j-j_*^n|+1)
\underset{\Lambda}{\lesssim} 2^{-\frac{|(\lambda _s+\mu _s\alpha
_{t(s)})(j-j_*^n)|}{2}}.
\end{align}
This together with an inequality
$$
\sum \limits _{\varphi _s^n(j)\le l<\varphi ^n_{s+1}(j)}2^{\tilde
\mu _sl}\le 2^{\tilde \mu _s\varphi ^n_{t(s)}(j)}(\varphi
^n_{s+1}(j)-\varphi ^n_s(j)+1)
$$
imply that
$$
\sum \limits _{(j, \, l)\in \Z^2\cap E^n_s}2^{\tilde\psi _n(j, \,
l)}\rho(2^{j+\tilde N_n})\le \sum \limits _{j=0}^{N_n}2^{\tilde
\lambda _s j+\tilde \mu _s\varphi _{t(s)}^n(j) +\tilde\nu _{s,n}}
(\varphi ^n_{s+1}(j)-\varphi ^n_s(j)+1)\rho(2^{j+\tilde N_n})
\stackrel{(\ref{phins1jphins1jjs})}{\underset{\Lambda}{\lesssim}}
$$
$$
\lesssim \sum \limits _{j=0}^{N_n}2^{\tilde \lambda _s j+ \tilde
\mu _s\varphi _{t(s)}^n(j) +\tilde\nu
_{s,n}}(|j-j_*^n|+1)\rho(2^{j+\tilde N_n})
\stackrel{(\ref{akjhbcrjhb53627})}
{\underset{\Lambda}{\lesssim}}\sum \limits _{j=0}^{N_n}
2^{-\frac{|(\lambda _s+\mu _s\alpha _{t(s)})(j-j_*^n)|}{2}}
\rho(2^{j+\tilde N_n}) \stackrel{(\ref{rho2jsigma})}{\le}
$$
$$
\le M(\varepsilon_0)\sum \limits _{j=0}^{N_n}2^{-\frac{|(\lambda
_s+ \mu _s\alpha _{t(s)})(j-j_*^n)|}{2}}2^{\varepsilon_0
|j-j_*^n|}\rho(2^{j_*^n+\tilde N_n})
\stackrel{(\ref{est_eps})}{\underset{\Lambda,
\gamma,c}{\lesssim}}\rho(2^{j_*^n+\tilde N_n}).
$$

Let $s\notin S$, $s\ne s_--1$ and $s<m$. It follows from
(\ref{nnleclog2nphin}) that
\begin{align}
\label{estcardsum} {\rm card}\, \Z ^2\cap E^n_s
\underset{c}{\lesssim}(\log _2 n)^2.
\end{align}
Let $(\hat j, \, \hat l)$ be the maximum point of the function
$\tilde \psi _n$ on the set $\Z^2\cap E^n_s$. Then
$$
\sum \limits _{(j, \, l)\in \Z^2\cap E^n_s}2^{\tilde\psi _n(j, \,
l)}\rho(2^{j+\tilde N_n})
\stackrel{(\ref{rho2jsigma}),(\ref{estcardsum})}{\underset{c}{\lesssim}}
M(\varepsilon_0)2^{\varepsilon _0N_n}(\log _2 n)^2\cdot 2^{\tilde
\psi _n (\hat j, \, \hat l)}\rho(2^{\tilde N_n+j_*^n}).
$$
Since $\tilde \Psi\in {\cal O}_{\varepsilon _0,J}(\Psi)$, then
$$
|\tilde \psi _n(\hat j, \, \hat l) -\psi _n(\hat j, \,\hat
l)|\stackrel{(\ref{1622})}{\le} \varepsilon _0\log _2 n.
$$
Since $s\notin S$, $s\ne s_--1$ and $s<m$, then
$$
\{(0, \, \varphi^n_s(0)), \, (0, \, \varphi^n_{s+1}(0)), \, (N_n,
\, \varphi^n_s(N_n)), \, (N_n, \, \varphi^n_{s+1}(N_n))\} \subset
V_n\backslash (j_*^n, \, l_*^n).
$$
The function $\psi_n|_{E^n_s}$ is affine, therefore,
$$
\sup_{(\xi, \, \eta)\in E^n_s}\psi_n(\xi, \, \eta)\le \max \left\{
\psi_n(0, \, \varphi^n_s(0)), \, \psi_n(0, \,
\varphi^n_{s+1}(0)),\right.$$$$\left.\psi_n(N_n, \,
\varphi^n_s(N_n)), \, \psi_n(N_n, \,
\varphi^n_{s+1}(N_n))\right\}\stackrel{(\ref{psijllemglog2nvbjl})}{\le}
-\gamma \log_2n.
$$
Hence,
$$
\Sigma
^n_s\stackrel{(\ref{nnleclog2nphin})}{\underset{c}{\lesssim}}
M(\varepsilon_0)n^{c\varepsilon_0}(\log _2 n)^2 2^{-\gamma\log _2
n}2^{\varepsilon _0\log _2 n}\rho(2^{\tilde N_n+j_*^n})
\underset{c,\Lambda,\gamma}{\lesssim}
$$
$$
\lesssim n^{-\gamma+\varepsilon_0(c+1)}(\log _2 n)^2\rho(2^{\tilde
N_n+j_*^n}) \stackrel{(\ref{est_eps})} {\le} \rho(2^{\tilde
N_n+j_*^n})
$$
for sufficiently large $n$.

It remains to consider the case $s=m$, $s\notin S$. Since $\mu
_m<0$, then
$$
\Sigma ^n _m\underset{\Lambda}{\lesssim}\sum \limits _{j=0}^{N_n}
2^{\tilde \psi _n(j, \, \varphi _m^n(j))}\rho(2^{j+\tilde N_n})\le
\sum \limits _{j=0}^{N_n} 2^{\psi _n(j, \, \varphi
_m^n(j))+|\tilde\psi _n(j, \, \varphi _m^n(j))-\psi _n(j, \,
\varphi _m^n(j))|}\rho(2^{j+\tilde
N_n})\stackrel{(\ref{1622}),(\ref{nnleclog2nphin}),
(\ref{psijllemglog2nvbjl}), (\ref{rho2jsigma})}{\le}
$$
$$
\le M(\varepsilon_0)2^{\varepsilon _0N_n}(c\log _2n)2^{-\gamma
\log_2n+\varepsilon_0 \log_2n}\rho(2^{\tilde N_n+j_*^n})
\stackrel{(\ref{nnleclog2nphin})}
{\underset{\Lambda,c,\gamma}{\lesssim}}n^{-\gamma+\varepsilon_0(c+1)}
(\log_2 n)\rho(2^{\tilde
N_n+j_*^n})\stackrel{(\ref{est_eps})}{\le} \rho(2^{\tilde
N_n+j_*^n})
$$
for sufficiently large $n$.
\end{proof}
Define the spaces $X_1$, $\tilde X_1$, $X_2$ and projections
$P_{{\cal N}}$ by (\ref{x1}), (\ref{tilde_x1}), (\ref{x2}) and
(\ref{def_of1_pmn}), respectively.
\begin{Sta}
\label{add} Let $n\in \Z_+$, $n_k\in \Z_+$, ${\cal N}\subset
\Z_+\times \Z^d$, ${\cal N}_k\subset \Z_+\times \Z^d$, $1\le k\le
m$, $n=\sum \limits _{k=1}^m n_k$, ${\cal N}\subset \cup
_{k=1}^m{\cal N}_k$, $\vartheta _{j,0}=d_j$,
$\vartheta_{j,1}={\cal A}_j$, $j\in \Z_+$. Then for $s\in \{0, \,
1\}$ we have
$$
\vartheta_{n,s}(P_{{\cal N}}:X_1\rightarrow X_2)\le \sum \limits
_{k=1}^m\vartheta _{n_k,s}(P_{{\cal N}_k}:X_1\rightarrow X_2).
$$
\end{Sta}
\begin{proof}
Let ${\cal N}'_1={\cal N}_1\cap {\cal N}$, ${\cal
N}'_k=\left({\cal N}_k\backslash \cup _{i=1}^{k-1} {\cal
N}_i\right)\cap {\cal N}$. Then for any $k=1, \, \dots, \, m$
$$\vartheta _{n_k,s}(P_{{\cal N}'_k}:X_1\rightarrow
X_2)\stackrel{(\ref{proj}),(\ref{proj1})}{\le} \vartheta
_{n_k,s}(P_{{\cal N}_k}:X_1\rightarrow X_2).$$ Therefore, the
assertion follows from inequalities $d_{n'+n''}(M'+M'', \, X_2)\le
d_{n'}(M', \, X_2)+d_{n''}(M'', \, X_2)$, ${\cal
A}_{n'+n''}(T'+T'':X_1\rightarrow X_2)\le {\cal
A}_{n'}(T':X_1\rightarrow X_2)+{\cal A}_{n''}(T'':X_1\rightarrow
X_2)$, from the inclusion $P_{{\cal N}}(B_{X_1})\subset \sum
_{k=1}^m P_{{\cal N}'_k}(B_{X_1})$ and from the equality $P_{{\cal
N}}=\sum \limits _{k=1}^m P_{{\cal N}'_k}$.
\end{proof}

Let $E$ be a finite set, $F:E\rightarrow \R$. Denote
\begin{align}
\label{argmax}
{\rm Argmax}\,\{F(x):x\in E\}=\left\{y\in E:F(y)=\max _{x\in E}F(x)\right\}.
\end{align}

\renewcommand{\proofname}{\bf Proof of Theorem \ref{main}}
\begin{proof}
Without loss of generality we may assume that $|(x_1, \, \dots, \,
x_d)|=\max _{1\le i\le d}|x_i|$.

By (\ref{com_dia}),
$$
d_n({\rm Id}:B^{s_1}_{p_1,q_1}(\R^d, \, w_1)\rightarrow B^{s_2}_{p_2,q_2}(\R^d, \,
w_2))\overset{n}{\asymp} d_n({\rm id}:X_1\rightarrow X_2),
$$
$$
{\cal A}_n({\rm Id}:B^{s_1}_{p_1,q_1}(\R^d, \, w_1)\rightarrow B^{s_2}_{p_2,q_2}(\R^d, \,
w_2))\overset{n}{\asymp} {\cal A}_n({\rm id}:X_1\rightarrow X_2).
$$
If $d_n({\rm Id}:B^{s_1}_{p_1,q_1}(\R^d, \, w_1)\rightarrow
B^{s_2}_{p_2,q_2}(\R^d, \,
w_2))\underset{n\rightarrow\infty}{\rightarrow} 0$, then the
operator ${\rm Id}$ is compact and, therefore, by (\ref{aneqln})
we have
$$
{\cal A}_n({\rm Id}:B^{s_1}_{p_1,q_1}(\R^d, \, w_1)\rightarrow B^{s_2}_{p_2,q_2}(\R^d, \,
w_2))=\lambda_n({\rm Id}:B^{s_1}_{p_1,q_1}(\R^d, \, w_1)\rightarrow B^{s_2}_{p_2,q_2}(\R^d, \,
w_2)).
$$

Denote
$$
{\cal N}=\left\{(\nu , \, {\bf m}):2^{-\nu}{\bf m}\in \left[-\frac 12, \,
\frac 12\right]^d\right\}, \;\;{\cal N}^*=\left\{(\nu , \, {\bf m}):2^{-\nu}{\bf m}\notin \left[-\frac 12, \,
\frac 12\right]^d\right\}.
$$

The following estimate was proved in \cite{haroske} (see Lemma
4.2):
$$
{\cal A}_n(P_{{\cal N}^*}:X_1\rightarrow
X_2)\overset{n}{\asymp}n^{-\tilde\theta _1}.
$$
Similarly, one can prove that
$$
d_n(P_{{\cal N}^*}:X_1\rightarrow
X_2)\overset{n}{\asymp}n^{-\theta _1}
$$
(all arguments are the same, and instead of approximation numbers
one takes Kolmogorov widths; see also \cite{zhang1}). Therefore,
it is enough to obtain the estimate of $d_n(P_{{\cal
N}}:X_1\rightarrow X_2)$ and ${\cal A}_n(P_{{\cal N}}:X_1
\rightarrow X_2)$.

Let us obtain the order estimate of $2^{-\nu(s_1-s_2)}w_1(Q_{\nu ,{\bf
m}})^{-1/p_1}w_2(Q_{\nu ,{\bf m}})^{1/p_2}$ for $(\nu, \, {\bf
m})\in {\cal N}$. If ${\bf m}\ne 0$ and $x\in Q_{\nu , {\bf m}}$,
then
$$
w_1(x)\overset{x,\nu,{\bf m}}{\asymp}2^{-\nu \beta_g p_1}|{\bf m}|^{\beta_g
p_1}\left(\log _2\frac{2^\nu}{|{\bf m}|}\right)^{p_1\alpha _g}\rho
_g^{-p_1}\left(\log _2\frac{2^\nu}{|{\bf m}|}\right),
$$
$$
w_2(x)\overset{x,\nu,{\bf m}}{\asymp}2^{\nu \beta_v
p_2}|{\bf m}|^{-\beta_v p_2}\left(\log _2
\frac{2^\nu}{|{\bf m}|}\right)^{-p_2\alpha _v}\rho _v^{p_2}\left(\log_2
\frac{2^\nu}{|{\bf m}|}\right).
$$
Hence,
\begin{align}
\label{norm} 2^{-\nu (s_1-s_2)}w_1(Q_{\nu ,{\bf
m}})^{-1/p_1}w_2(Q_{\nu ,{\bf m}})^{1/p_2}\overset{\nu,{\bf
m}}{\asymp}|{\bf m}|^{-\delta}\left(\log _2\frac{2^\nu}{|{\bf
m}|}\right)^{-\alpha}\rho\left(\log _2\frac{2^\nu}{|\bf m|}\right).
\end{align}
If ${\bf m}=0$, then by (\ref{ilbw1ovrsrrbp1d}) and
(\ref{ilbw2ovrsrrbp2d}) the following  order equalities hold:
$$
w_1(Q_{\nu,0})=\int \limits _{Q_{\nu ,0}}x^{\beta_g p_1}|\log_2
|x||^{p_1\alpha _g}\rho _g^{-p_1}(|\log_2 |x||)\,
dx\overset{\nu}{\asymp} \left .r^{\beta_g p_1+d}|\log_2
r|^{p_1\alpha _g}\rho _g^{-p_1}(|\log_2 r|)\right |_{r=2^{-\nu}},
$$
$$
w_2(Q_{\nu,0})=\int \limits _{Q_{\nu ,0}}x^{-\beta_v p_2}|\log_2
|x||^{-p_2\alpha _v}\rho_v^{p_2}(|\log_2 |x||)\,
dx\overset{\nu}{\asymp} \left .r^{-\beta_v p_2+d}|\log_2
r|^{-p_2\alpha _v}\rho _v^{p_2}(|\log_2 r|)\right |_{r=2^{-\nu}}.
$$
Therefore,
\begin{align}
\label{norm1} 2^{-\nu (s_1-s_2)}w_1(Q_{\nu ,{\bf
m}})^{-1/p_1}w_2(Q_{\nu ,{\bf m}})^{1/p_2}\overset{\nu}{\lesssim} \nu
^{-\alpha}\rho(\nu).
\end{align}

{\it Proof of an upper estimate.} Let $\varepsilon >0$ (it will be chosen later).

Note that it is enough to consider $n=2^{Nd}$, where $N\in \Z_+$. Let us construct the
covering of the set ${\cal N}$ for any $n$.

{\bf Step 1.} Define the covering of the set $\{(\nu , \, {\bf m})\in
{\cal N}:{\bf m}\ne 0\}$.

Let $0\le j\le Nd$,
\begin{align}
\label{defcj} c_N(j)=\varepsilon j \text{ or }c_N(j)=\varepsilon
(Nd-j)
\end{align}
(the choice will be made later). For $2^j\le t<2^{j+1}$, $l\in \Z$
we put $$\nu _t(l)=N+t+l-\left[\frac jd\right]-[c_N(j)],$$
$$
{\cal M}^1_{\nu _t(l)}=\{ {\bf m}\in \Z^d:2^{-\nu _t(l)}{\bf m}\in
[-2^{-t}, \, 2^{-t}]^d\backslash [-2^{-t-1}, \, 2^{-t-1}]^d\}.
$$
Denote
$$
{\cal N}^1_{j,l}=\{(\nu _t(l), \, {\bf m}):2^j\le t<2^{j+1}, \;
{\bf m}\in {\cal M}^1_{\nu _t(l)}, \; \nu_t(l)\ge 0\}.
$$

Let us remark that if ${\bf m}\in {\cal M}^1_{\nu _t(l)}$, then
\begin{align}
\label{por_m_mnuk} |{\bf m}|\underset{d}{\asymp} 2^{-t}\cdot
2^{\nu _t(l)}=2^{N+l-\left[\frac jd\right] -[c_N(j)]},
\end{align}
and if $N+l-\left[\frac jd\right] -[c_N(j)]\ge 0$ (i.e.
${\cal M}_{\nu _t(l)}^1\ne \varnothing$), then
\begin{align}
\label{card_mnuk} {\rm card}\, {\cal M}^1_{\nu
_t(l)}=\left(2^{N+l-\left[\frac jd\right] -[c_N(j)]+1}+1\right)^d-
\left(2\cdot\left[2^{N+l-\left[\frac jd\right]
-[c_N(j)]-1}\right]+1\right)^d \underset{d}{\asymp}
2^{Nd+ld-j-[c_N(j)]d}.
\end{align}
In addition,
\begin{align}
\begin{array}{c}
\displaystyle \label{rk_p} \sum \limits _{j=0}^{Nd}\sum \limits
_{l<0}{\rm rk}\, P_{{\cal N}^1_{j,l}}=\sum \limits _{j=0}^{Nd}\sum
\limits _{l<0} \sum \limits _{2^j\le t<2^{j+1}}{\rm card}\, {\cal
M}^1_{\nu _t(l)}
\stackrel{(\ref{card_mnuk})}{\underset{d}{\lesssim}} \\
\lesssim \sum \limits _{j=0}^{Nd}\sum \limits _{l<0}2^j\cdot
2^{Nd+ld-j-[c_N(j)]d}\underset{d}{\asymp} n\sum \limits
_{j=0}^{Nd}2^{-[c_N(j)]d}\underset{\varepsilon, \,
d}{\stackrel{(\ref{defcj})}{\asymp}} n.
\end{array}
\end{align}

Let $j$, $l\in \Z_+$. For $2^{j+Nd}\le t<2^{j+1+Nd}$ we put
$\tilde\nu _t(l)=t+l$,
$$
{\cal M}^2_{\tilde\nu _t(l)}=\{{\bf m}\in \Z^d:2^{-\tilde\nu
_t(l)}{\bf m}\in [-2^{-t}, \, 2^{-t}]^d\backslash [-2^{-t-1}, \,
2^{-t-1}]^d\},
$$
$$
{\cal N}^2_{j,l}=\{(\tilde\nu _t(l), {\bf m}):2^{j+Nd}\le
t<2^{j+1+Nd}, \; {\bf m}\in {\cal M}^2_{\tilde\nu _t(l)}\}.
$$
If ${\bf m}\in {\cal M}^2_{\tilde\nu _t(l)}$, then
\begin{align}
\label{por_m_mnuk2} |{\bf m}|\underset{d}{\asymp} 2^{-t}\cdot
2^{\tilde\nu _t(l)}=2^l,
\end{align}
\begin{align}
\label{card_mnuk2} {\rm card}\, {\cal M}_{\tilde\nu
_t(l)}=\left(2\cdot 2^l+1\right)^d-
\left(2\cdot[2^{l-1}]+1\right)^d\underset{d}{\asymp} 2^{ld}.
\end{align}

Denote
$$
j_0(N)=Nd\cdot \max\left\{\frac{q_2}{2}-1, \, 0\right\}, \; \;
j_1(N)=Nd\cdot \max\left\{\frac{\min \{q_2, \, q_1'\}}{2}-1, \,
0\right\}.
$$

For $\nu\in \Z_+$, $s\in \{0, \, 1\}$ we put
$$
{\cal M}^{3,s}_\nu =\left\{{\bf m}\in \Z^d\backslash
\{0\}:2^{-\nu}{\bf m}\in \left[-2^{-2^{[j_s(N)]+Nd}}, \,
2^{-2^{[j_s(N)]+Nd}}\right]^d\right\},
$$
$$
{\cal N}^{3,s}=\{(\nu, \, {\bf m}): \; {\bf m}\in {\cal
M}^{3,s}_{\nu}\}.
$$

Let us show that for $s\in \{0, \, 1\}$ we have
$$
\left(\cup _{0\le j\le Nd, \; l\in \Z}{\cal N}_{j,l}^1\right)\cup
\left(\cup _{0\le j\le j_s(N), \; l\in \Z_+}{\cal
N}_{j,l}^2\right)\cup {\cal N}^{3,s}\supset \{(\nu, \, {\bf m})\in
{\cal N}:{\bf m}\ne 0\}.
$$
Indeed, if $2^{-\nu}{\bf m}\in \left[-\frac 12, \, \frac
12\right]^d\backslash \left[-2^{-2^{Nd}-1}, \,
2^{-2^{Nd}-1}\right]^d$, then there exist $0\le j\le Nd$, $2^j\le
t\le 2^{j+1}-1$ and $l\in \Z$ such that $2^{-\nu}{\bf m}\in
[-2^{-t}, \, 2^{-t}]^d\backslash [-2^{-t-1}, \, 2^{-t-1}]^d$ and
$\nu=N+t+l-\left[\frac jd\right]-[c_N(j)]$. Hence, $(\nu, \, {\bf
m})\in {\cal N}_{j,l}^1$.

If $2^{-\nu}{\bf m}\in \left[-2^{-2^{Nd}-1}, \,
2^{-2^{Nd}-1}\right]^d\backslash \left[-2^{-2^{Nd+[j_s(N)]}}, \,
2^{-2^{Nd+[j_s(N)]}}\right]^d$, then there exist $0\le j\le
j_s(N)$, $2^{Nd+j}\le t\le 2^{Nd+j+1}-1$ such that $2^{-\nu}{\bf
m}\in [-2^{-t}, \, 2^{-t}]^d\backslash [-2^{-t-1}, \,
2^{-t-1}]^d$. Here $\nu \ge t$, i.e. $\nu =t+l$ for some $l\in
\Z_+$. Therefore, $(\nu, \, {\bf m})\in {\cal N}^2_{j, \, l}$.

Finally, if $2^{-\nu}{\bf m}\in \left[-2^{-2^{[j_s(N)]+Nd}}, \,
2^{-2^{[j_s(N)]+Nd}}\right]^d$, then $(\nu, \, {\bf m})\in {\cal
N}^{3,s}$.

{\bf Step 2.} Let us construct the covering of $\{(\nu ,\, 0):\nu
\in \Z_+\}$. Let
$$
{\cal N}^4=\{(0, \, 0), \, \dots , \, (2^{Nd}-1, \, 0)\},
$$
$$
{\cal N}^5_j=\{(2^{j+Nd}, \, 0), \, (2^{j+Nd}+1, \, 0), \, \dots,
(2^{j+1+Nd}-1, \, 0)\}, \; j\in \Z_+,
$$
${\cal N}^{6,s}=\{(2^{Nd+[j_s(N)]}+t, \, 0), \; t\in \Z_+\}$,
$s\in \{0, \, 1\}$.

Then
$$
{\cal N}^4\cup\left(\cup_{j=0}^{[j_s(N)]-1}{\cal N}^5_j\right)\cup
{\cal N}^{6,s} \supset \{(\nu ,\, 0):\nu \in \Z_+\}.
$$

{\bf Step 3.} Let $\mu ^1_{j,l}$, $\mu^2_{j,l}$, $\mu ^5_j\in
\Z_+$, $s\in \{0, \, 1\}$, $L_j=\left\{l\in \Z_+:\,
N+l-\left[\frac jd\right]-[c_N(j)]\ge 0\right\}$,
$$
n_s:=\sum \limits _{j=0}^{Nd}\sum \limits _{l\in L_j}\mu ^1_{j,l}+
\sum \limits _{j=0}^{Nd}\sum \limits _{l\in \Z\backslash \Z_+, \,
{\cal N}^1_{j,l}\ne \varnothing}\dim \left(P_{{\cal
N}^1_{j,l}}X_1\right)+
$$
$$
+\sum \limits _{0\le j\le j_s(N)}\sum \limits _{l\in
\Z_+}\mu^2_{j,l}+ \dim (P_{{\cal N}^4}X_1)+\sum\limits _{0\le j\le
j_s(N)}\mu ^5_j<\infty,
$$
$\vartheta _{\mu,0}=d_\mu$, $\vartheta _{\mu,1}={\cal A}_\mu$
($\vartheta _{\mu,1}=\lambda _\mu$ if the embedding is compact),
$\mu\in \N$. Then by Proposition \ref{add},
$$
\vartheta _{n_s,s}(P_{{\cal N}}:X_1\rightarrow X_2)\le \sum
\limits _{j=0}^{Nd} \sum \limits _{l\in L_j} \vartheta_{\mu
_{j,l}^1,s}(P_{{\cal N}^1_{j,l}}:X_1\rightarrow X_2)+
$$
$$
+\sum \limits _{0\le j\le j_s(N)}\sum \limits _{l\in
\Z_+}\vartheta_{\mu _{j,l}^2,s} (P_{{\cal
N}^2_{j,l}}:X_1\rightarrow X_2)+ \vartheta_{0,s}(P_{{\cal
N}^{3,s}}:X_1\rightarrow X_2)+$$$$+\sum \limits _{0\le j\le
j_s(N)}\vartheta_{\mu ^5_j,s}(P_{{\cal N}^5_j}:X_1\rightarrow
X_2)+ \vartheta_{0,s}(P_{{\cal N}^{6,s}}:X_1\rightarrow X_2).
$$
Combining (\ref{x1}), (\ref{tilde_x1}), (\ref{x2}),
(\ref{norm_bspq_proj}), (\ref{norm}), (\ref{por_m_mnuk}),
(\ref{card_mnuk}) and Theorems \ref{fin_dim_mix_norm} and
\ref{lin_fin}, we get
$$
\vartheta_{\mu _{j,l}^1,s}(P_{{\cal N}^1_{j,l}}:X_1\rightarrow
X_2) \overset{n,j,l}{\lesssim}\left(2^{N+l-\left[\frac jd\right]
-[c_N(j)]}\right)^{-\delta}2^{-\alpha
j}\rho\left(2^j\right)\vartheta_{\mu
_{j,l}^1,s}(B_{p_1,q_1}^{m^1_{j,l}, k^1_{j,l}}, \,
l_{p_2,q_2}^{m^1_{j,l}, k ^1_{j,l}}),
$$
where
\begin{align}
\label{mjl1_kjl1} m^1_{j,l} =2^{Nd+ld-j-[c_N(j)]d}, \;\;
k^1_{j,l}=2^j.
\end{align}
Similarly, combining (\ref{x1}), (\ref{tilde_x1}), (\ref{x2}),
(\ref{norm_bspq_proj}), (\ref{norm}), (\ref{por_m_mnuk2}),
(\ref{card_mnuk2}) and Theorems \ref{fin_dim_mix_norm} and
\ref{lin_fin}, we obtain
$$
\vartheta_{\mu _{j,l}^2,s}(P_{{\cal N}^2_{j,l}}:X_1\rightarrow
X_2)\overset{n,j,l}{\lesssim} 2^{-\delta
l}\left(2^{j+Nd}\right)^{-\alpha} \rho\left(2^{j+Nd}\right)
\vartheta_{\mu ^2_{j,l},s} (B_{p_1,q_1}^{m^2_{j,l}, k^2_{j,l}}, \,
l_{p_2,q_2}^{m_{j,l}^2, k ^2_{j,l}}),
$$
\begin{align}
\label{mjl1_kjl2} m^2_{j,l}=
2^{ld}, \;\; k^2_{j,l}=2^{j+Nd}.
\end{align}
Denote $r_0=\max\left(\frac{q_2}{2}, \, 1\right)$,
$r_1=\max\left(\frac{\min(q_2, \, q_1')}{2}, \, 1\right)$. By
(\ref{x1}), (\ref{tilde_x1}), (\ref{x2}), (\ref{norm_bspq_proj}),
(\ref{norm}) and inequalities $|{\bf m}|\ge 1$, $\log_2
\frac{2^\nu}{|\bf m|}\ge 2^{[j_s(N)]+Nd}$, $(\nu, \, {\bf m})\in
{\cal N}^{3,s}$, we have
$$
\vartheta_{0,s}(P_{{\cal N}^{3,s}}:X_1\rightarrow
X_2)\overset{n}{\lesssim}(2^{j_s(N)+Nd})^{-\alpha}\rho(2^{j_s(N)+Nd})
\overset{n}{\asymp}n^{-\alpha r_s}\rho\left(n^{r_s}\right).
$$

Further, from (\ref{x1}), (\ref{tilde_x1}), (\ref{x2}),
(\ref{norm_bspq_proj}), (\ref{norm1}) we get
$$
\vartheta_{0,s}(P_{{\cal N}^{6,s}}:X_1\rightarrow
X_2)\overset{n}{\lesssim} \sup _{t\in
\Z_+}(2^{j_s(N)+Nd}+t)^{-\alpha}\rho(2^{j_s(N)+Nd}+t)
\overset{n}{\asymp}
$$
$$
\asymp(2^{j_s(N)+Nd})^{-\alpha}\rho(2^{j_s(N)+Nd})\overset{n}{\asymp}
n^{-\alpha r_s}\rho\left(n^{r_s}\right),
$$
$$
\vartheta_{\mu ^5_j,s}(P_{{\cal N}^5_j}:X_1\rightarrow
X_2)\overset{j,n}{\lesssim}
(2^{j+Nd})^{-\alpha}\rho(2^{j+Nd})\vartheta_{\mu
^5_j,s}(B_{q_1}^{k_j^5}, \, l_{q_2}^{k_j^5}),
$$
where
\begin{align}
\label{mj6kj6} k_j^5=2^{j+Nd}.
\end{align}
Thus,
$$
\vartheta_{n_s,s}(P_{{\cal N}}:X_1\rightarrow
X_2)\overset{n}{\lesssim}\sum \limits _{j=0}^{Nd} \sum\limits
_{l\in L_j} n^{-\frac{\delta}{d}}2^{-l\delta}2^{j(-\alpha+
\frac{\delta}{d})} \rho\left(2^j\right) 2^{c_N(j)\delta}
\vartheta_{\mu _{j,l}^1,s}(B_{p_1,q_1}^{m^1_{j,l}, k^1_{j,l}}, \,
l_{p_2,q_2}^{m^1_{j,l}, k ^1_{j,l}})+
$$
$$
+\sum \limits _{0\le j\le j_s(N)}\sum \limits _{l\in \Z_+}
n^{-\alpha}2^{-l\delta} 2^{-j\alpha}\rho\left(2^jn\right)
\vartheta_{\mu ^2_{j,l},s}(B_{p_1,q_1}^{m^2_{j,l}, k^2_{j,l}}, \,
l_{p_2,q_2}^{m_{j,l}^2, k ^2_{j,l}})+
$$
$$
+n^{-\alpha r_s}\rho\left(n^{r_s}\right)+\sum \limits _{0\le j\le
j_s(N)} n^{-\alpha}2^{-j\alpha}\rho(2^jn)\vartheta_{\mu
^5_j,s}(B_{q_1}^{k_j^5}, \, l_{q_2}^{k_j^5})=:\Sigma _s.
$$

{\bf Step 4.} Let $p_2\le 2$ and $q_2\le 2$. We take
$\mu^1_{j,l}=0$, $\mu^2_{j,l}=0$, $\mu^5_j=0$,
$c_N(j)=\varepsilon|j-j_*^n|$, where $j_*^n=0$, if
$\alpha>\frac{\delta}{d}$, and $j_*^n=Nd$, if
$\alpha<\frac{\delta}{d}$. Then
$$
\max(\Sigma _0, \, \Sigma _1)\overset{n}{\lesssim}\sum\limits
_{j=0}^{Nd} \sum\limits _{l\ge
0}n^{-\frac{\delta}{d}}2^{-l\delta}2^{j\left(-\alpha+\frac{\delta}{d}\right)}
\rho\left(2^j\right)2^{\varepsilon\delta|j-j_*^n|}+\sum \limits
_{l\in \Z_+}n^{-\alpha}2^{-l\delta}
\rho\left(n\right)+n^{-\alpha}\rho(n).
$$
Since $\alpha\ne \frac{\delta}{d}$, then $\max(\Sigma _0, \,
\Sigma _1)\overset{n}{\lesssim}\max \left\{n^{-\frac{\delta}{d}},
\, n^{-\alpha}\rho(n)\right\}$ for sufficiently small
$\varepsilon$.

Similarly one can prove that if $p_1\ge 2$ and $q_1\ge 2$,
$\varepsilon$ is sufficiently small, $\mu^1_{j,l}$, $\mu^2_{j,l}$,
$\mu^5_j$ and $c_N(j)$ are defined as above, then $\Sigma
_1\overset{n}{\lesssim}\max \left\{n^{-\frac{\delta}{d}}, \,
n^{-\alpha}\rho(n)\right\}$.

{\bf Step 5.} Let us obtain an estimate of $\Sigma _0$ for $p_2\ge
2$, $q_2\ge 2$ and under conditions of part 1 of the theorem. For
$i=1, \, 2$ we put
$$
\mu ^i_{j,l}=\left\{ \begin{array}{l} n\cdot 2^{-[\sigma_n
^i(j)]-[\tau_n^i(l, \, j)]}, \text{ if }l\le \overline{l}_i(j),
\\ 0, \text{ if }l>\overline{l}_i(j),\end{array}\right .
$$
where
$$
\overline{l}_1(j)=N\left(\frac{p_2}{2}-1\right)+
\frac{j}{d}\left(1-\frac{p_2}{q_2}\right), \;\;\;
\overline{l}_2(j)=Np_2\left(\frac{1}{2}-\frac{1}{q_2}\right)-
\frac{p_2j}{q_2d},
$$
$$
\sigma ^i_n(j)=\varepsilon |j-j_*^{n,i}|, \;\; \tau ^i_n(l, \,
j)=\varepsilon |l-\hat l^{n,i}(j)|,
$$
$$
\hat l^{n,i}(j)=\left\{ \begin{array}{l} 0, \text{ if
}l_*^{n,i}=0, \\ \tilde l_i(j), \text{ if }l_*^{n,i}=\tilde
l_i(j_*^{n,i}), \\ \overline{l}_i(j), \text{ if
}l_*^{n,i}=\overline{l}_i(j_*^{n,i});\end{array}\right.
$$
functions $\tilde l_i$ and numbers $j_*^{n,1}\in \{0, \, Nd\}$,
$j_*^{n,2}\in \{0, \, j_0(N)\}$, $l_*^{n,i}\in \{0, \, \tilde
l_i(j_*^{n,i}), \, \overline{l}_i(j_*^{n,i})\}$ ($i=1, \, 2$) will
be defined later. Then $\sum \limits _{j,l}\mu
^i_{j,l}\overset{n}{\lesssim}n$. In addition, put
$c_N(j)=\varepsilon|j-j_*^{n,1}|$.

{\bf Substep 5.1.} Estimate the first term of the sum $\Sigma _0$.
Denote for $l\in L_j$
$$d_{n,j,l}=n^{-\frac{\delta}{d}}2^{-l\delta}2^{j(-\alpha+\frac{\delta}{d})}
\cdot2^{c_N(j)\delta}d_{\mu _{j,l}^1}(B_{p_1,q_1}^{m^1_{j,l},
k^1_{j,l}}, \, l_{p_2,q_2}^{m^1_{j,l}, k ^1_{j,l}}).$$

We apply Theorem \ref{fin_dim_mix_norm}.
\begin{itemize}
\item $p_1\le 2$, $q_1\le 2$. Then by (\ref{mjl1_kjl1}) we get
$$
d_{n,j,l}\overset{n,j,l}{\lesssim} \left\{\begin{array}{l}
n^{-\frac{\delta}{d}+\frac{1}{p_2}-\frac 12}\cdot
2^{l\left(-\delta+\frac{d}{p_2}\right)+\frac{\tau ^1_n(l, \,
j)}{2}+j\left(-\alpha+\frac{\delta}{d}-\frac{1}{p_2}
+\frac{1}{q_2}\right)+\frac{\sigma ^1_n(j)}{2}+c_N(j)\left(\delta
-\frac{d}{p_2}\right)},\text{ if }l\le \overline{l}_1(j),
\\ n^{-\frac{\delta}{d}}2^{-l\delta}2^{j(-\alpha+\frac{\delta}{d})+
c_N(j)\delta}, \text{ if }l>\overline{l}_1(j).
\end{array}\right.
$$
Put
$$
F_n^1(j, \, l)=\left\{\begin{array}{l}
n^{-\frac{\delta}{d}+\frac{1}{p_2}-\frac 12}\cdot
2^{l\left(-\delta
+\frac{d}{p_2}\right)+j\left(-\alpha+\frac{\delta}{d}-\frac{1}{p_2}
+\frac{1}{q_2}\right)}, \text{ if }l\le \overline{l}_1(j),
\\ n^{-\frac{\delta}{d}} 2^{-l\delta} 2^{j\left(-\alpha+\frac{\delta}{d}\right)}, \text{ if }
l>\overline{l}_1(j)\end{array}\right.
$$
and take some
$$
(j_*^{n,1}, \, l_*^{n,1})\in {\rm Argmax}\, \left\{F_n^1(0, \, 0),
\, F_n^1(Nd, \, 0), \, F_n^1(0, \,\overline{l}_1(0)), \,
F_n^1(Nd,\, \overline{l}_1(Nd))\right\}=
$$
$$
={\rm Argmax}\, \left\{n^{-\frac{\delta}{d}+\frac{1}{p_2}-\frac
12}, \, n^{-\alpha+\frac{1}{q_2}-\frac 12}, \,
n^{-\frac{p_2\delta}{2d}}, \, n^{-\alpha
-\frac{p_2\delta}{2d}+\frac{p_2\delta}{q_2\cdot d}}\right\}
$$
(see (\ref{argmax})). Here we may assume that either $j_*^{n,1}=0$
for any $n\in \N$ or $j_*^{n,1}=Nd$ for any $n\in \N$.

Apply Lemma \ref{osn_lemma} for $\varphi_1^n(\xi)=0$ and
$\varphi^n_2(\xi)=\overline l_1(\xi)$. If $\varepsilon$ is
sufficiently small, then
\begin{align}
\label{summuj1} \displaystyle
\begin{array}{c}\Sigma_0':=\sum \limits _{j=0}^{Nd}
\sum\limits _{l\in L_j}n^{-\delta
/d}2^{-l\delta}2^{j(-\alpha+\delta/d)}
\rho\left(2^j\right)2^{c_N(j)\delta}d_{\mu
_{j,l}^1}(B_{p_1,q_1}^{m^1_{j,l}, k^1_{j,l}}, \,
l_{p_2,q_2}^{m^1_{j,l}, k ^1_{j,l}})\overset{n}{\lesssim}
\\
\lesssim\max \left\{n^{-\frac{\delta}{d}+\frac{1}{p_2}-\frac 12},
\, n^{-\frac{p_2\delta}{2d}}, \, n^{-\alpha-\frac
12+\frac{1}{q_2}}\rho(n), \, n^{-\frac{p_2\delta}{2d}-\alpha
+\frac{p_2\delta}{q_2\cdot d}}\rho(n), \, n^{-\frac{\alpha
q_2}{2}}\rho(n^{\frac{q_2}{2}})\right\}\overset{n}{\lesssim}\\
\overset{n}{\lesssim}\max
\left\{n^{-\frac{\delta}{d}+\frac{1}{p_2}-\frac 12}, \,
n^{-\frac{p_2\delta}{2d}}, \, n^{-\alpha-\frac
12+\frac{1}{q_2}}\rho(n), \, n^{-\frac{\alpha
q_2}{2}}\rho(n^{\frac{q_2}{2}})\right\}.
\end{array}
\end{align}
Indeed, let $q_2>2$, $p_2>2$ (cases $p_2=2$ and $q_2=2$ can be
treated in a similar way). Notice that
\begin{align}
\label{p2d} \frac{p_2\delta}{2d}+\alpha-
\frac{p_2\delta}{q_2d}=\frac{\alpha q_2}{2}\cdot
\frac{2}{q_2}+\frac{p_2\delta}{2d} \left(1-\frac{2}{q_2}\right)\ge
\min \left\{\frac{\alpha q_2}{2}, \, \frac{p_2\delta}{2d}\right\}
\end{align}
(and the equality holds only if $\frac{\alpha
q_2}{2}=\frac{p_2\delta}{2d}$). By conditions of the theorem,
$\frac{\alpha q_2}{2}\ne \min \left\{\frac{\delta}{d}+\frac 12
-\frac{1}{p_2}, \, \frac{p_2\delta}{2d}, \, \alpha +\frac
12-\frac{1}{q_2}\right\}$.

If $\frac{\alpha q_2}{2}<\min \left\{\frac{\delta}{d}+\frac 12
-\frac{1}{p_2}, \, \frac{p_2\delta}{2d}, \, \alpha +\frac
12-\frac{1}{q_2}\right\}$, then there exists $\theta>0$ such that
$\frac{\alpha q_2}{2}+\theta<\min \left\{\frac{\delta}{d}+\frac 12
-\frac{1}{p_2}, \, \frac{p_2\delta}{2d}, \, \alpha +\frac
12-\frac{1}{q_2}, \, \alpha+\frac{p_2\delta}{2d}
-\frac{p_2\delta}{q_2\cdot d}\right\}$. Therefore, $\Sigma
'_0\overset{n}{\lesssim} n^{-\frac{\alpha
q_2}{2}-\theta}\overset{n}{\lesssim}n^{-\frac{\alpha
q_2}{2}}\rho(n^{\frac{q_2}{2}})$ (here the set $S$ from Lemma
\ref{osn_lemma} is empty).

Let $\frac{\alpha q_2}{2}>\min \left\{\frac{\delta}{d}+\frac 12
-\frac{1}{p_2}, \, \frac{p_2\delta}{2d}, \, \alpha +\frac
12-\frac{1}{q_2}\right\}$. By (\ref{p2d}),
$$\frac{p_2\delta}{2d}+\alpha- \frac{p_2\delta}{q_2d}>\min
\left\{\frac{\delta}{d}+\frac 12 -\frac{1}{p_2}, \,
\frac{p_2\delta}{2d}, \, \alpha +\frac 12-\frac{1}{q_2}\right\}$$
and minimum of the right-hand side is attained at the unique
point. Hence, the set $S$ from Lemma \ref{osn_lemma} is nonempty
and the sum can be estimated from above by $\max
\left\{n^{-\frac{\delta}{d}+ \frac{1}{p_2} -\frac 12}, \,
n^{-\frac{p_2\delta}{2d}}, \, n^{-\alpha-\frac
12+\frac{1}{q_2}}\rho(n)\right\}$.

\item $\lambda({\bf p})\le \lambda({\bf q})$, $\lambda({\bf
p})<1$. Put $\tilde
l_1(j):=\frac{j}{d}\left(1-\frac{2}{q_2}\right)$,
$\overline{m}^1_{j,l}=2^{Nd+ld-j}$. Then the condition
$n>\overline{m}^1_{j,l} (k^1_{j,l})^{2/q_2}$ is equivalent to the
inequality $l< \tilde l_1(j)$. Since $j\le Nd$, $p_2\ge 2$ and
$q_2\ge 2$, we have $\tilde l_1(j)\le \overline{l}_1(j)$.
Therefore,
$$
d_{n,j,l}\overset{n,j,l}{\lesssim} n^{-\frac{\delta}{d}-\frac{1}{p_1}+\frac{1}{p_2}}\cdot
2^{l\left(-\delta+\frac{d}{p_2}-\frac{d}{p_1}+ \frac{\lambda({\bf
q})d}{2}\right)+j\left(-\alpha+\frac{\delta}{d}-
\frac{1}{p_2}+\frac{1}{p_1}+\frac{\lambda({\bf q})}{q_2}-
\frac{\lambda({\bf q})}{2}\right)}\times
$$
$$
\times 2^{\frac{\sigma^1_n(j)\lambda({\bf q})}{2}+\frac{\tau
^1_n(l,\, j)\lambda({\bf q})}{2}+c_N(j)d
\left(\frac{\delta}{d}-\frac{1}{p_2} +\frac{1}{p_1}
-\frac{\lambda({\bf q})}{2}\right)},\text{ if }0\le l< \tilde
l_1(j),
$$
$$
d_{n,j,l}\overset{n,j,l}{\lesssim}n^{-\frac{\delta}{d}-\frac{\lambda({\bf
p})}{2}+\frac{\lambda({\bf p})}{p_2}}\cdot 2^{l\left(-\delta
+\frac{\lambda({\bf p})d}{p_2}\right)+j\left(-\alpha+\frac{\delta}{d}
-\frac{\lambda({\bf p})}{p_2}+\frac{\lambda({\bf
p})}{q_2}\right)}\times
$$
$$
\times 2^{ \frac{\sigma^1_n(j)\lambda({\bf p})}{2}+ \frac{\tau
^1_n(l,\, j)\lambda({\bf p})}{2}+
c_N(j)d\left(\frac{\delta}{d}-\frac{\lambda({\bf p})}{p_2}\right)
}, \text{ if }\tilde l_1(j)\le l<\overline{l}_1(j),
$$
$$
d_{n,j,l}\overset{n,j,l}{\lesssim}n^{-\frac{\delta}{d}}2^{-l\delta}
2^{j(-\alpha+\frac{\delta}{d})}\cdot 2^{c_N(j)\delta},\text{ if }
l\ge\overline{l}_1(j).
$$
Let
$$
F_n^1(j, \, l)=n^{-\frac{\delta}{d}-\frac{1}{p_1}+\frac{1}{p_2}}\cdot
2^{l\left(-\delta+\frac{d}{p_2}-\frac{d}{p_1}+ \frac{\lambda({\bf
q})d}{2}\right)+j\left(-\alpha+\frac{\delta}{d}-
\frac{1}{p_2}+\frac{1}{p_1}+\frac{\lambda({\bf q})}{q_2}-
\frac{\lambda({\bf q})}{2}\right)},
$$
if $0\le l< \tilde l_1(j)$,
$$
F_n^1(j, \, l)=n^{-\frac{\delta}{d}-\frac{\lambda({\bf
p})}{2}+\frac{\lambda({\bf p})}{p_2}}\cdot 2^{l\left(-\delta
+\frac{\lambda({\bf p})d}{p_2}\right)+j\left(-\alpha+\frac{\delta}{d}
-\frac{\lambda({\bf p})}{p_2}+\frac{\lambda({\bf
p})}{q_2}\right)},
$$
if $\tilde l_1(j)\le l<\overline{l}_1(j)$,
$$
F_n^1(j, \, l)=n^{-\frac{\delta}{d}}2^{-l\delta}
2^{j(-\alpha+\frac{\delta}{d})}, \text{ if }l\ge\overline{l}_1(j).
$$
Note that $\tilde l_1(0)=0.$ Choose
$$
(j_*^{n,1}, \, l_*^{n,1})\in {\rm Argmax}\, \left\{F_n^1(0, \, 0),
\, F_n^1(Nd, \, 0), \, F_n^1(0, \,\overline{l}_1(0)),
\right.$$$$\left. F_n^1(Nd,\, \overline{l}_1(Nd)), \, F_n^1(Nd,\,
\tilde{l}_1(Nd))\right\}.
$$

By Lemma \ref{osn_lemma}, for enough small $\varepsilon$ we obtain
\begin{align}
\label{summuj2} \displaystyle
\begin{array}{c}\sum \limits _{j=0}^{Nd}
\sum\limits _{l\in L_j}n^{-\delta
/d}2^{-l\delta}2^{j(-\alpha+\delta/d)}
\rho\left(2^j\right)2^{c_N(j)\delta}d_{\mu
_{j,l}^1}(B_{p_1,q_1}^{m^1_{j,l}, k^1_{j,l}}, \,
l_{p_2,q_2}^{m^1_{j,l}, k ^1_{j,l}})\overset{n}{\lesssim}
\\
\lesssim\max\left\{n^{-\frac{\delta}{d}-
\frac{1}{p_1}+\frac{1}{p_2}}, \, n^{-\alpha+\frac{\lambda({\bf
q})}{q_2}- \frac{\lambda({\bf
q})}{2}}\rho(n), \, n^{-\frac{p_2\delta}{2d}}, \right .\\
\left . n^{-\frac{\delta p_2}{2d}-\alpha+\frac{\delta
p_2}{q_2\cdot d}}\rho(n), \, n^{-\alpha
+\left(-\frac{\delta}{d}-\frac{1}{p_1}+\frac{1}{p_2}\right)
\left(1-\frac{2}{q_2}\right)}\rho(n), \, n^{-\frac{\alpha
q_2}{2}}\rho(n^{\frac{q_2}{2}})\right\}\overset{n}{\lesssim}\\
\lesssim\max\left\{n^{-\frac{\delta}{d}-
\frac{1}{p_1}+\frac{1}{p_2}}, \, n^{-\alpha+\frac{\lambda({\bf
q})}{q_2}- \frac{\lambda({\bf q})}{2}}\rho(n), \,
n^{-\frac{p_2\delta}{2d}}, \, n^{-\frac{\alpha
q_2}{2}}\rho(n^{\frac{q_2}{2}})\right\}.
\end{array}
\end{align}
It can be proved in a similar way as in the previous case. Here
$\varphi_1^n(\xi)=0$, $\varphi_2^n(\xi)=\tilde l_1(\xi)$,
$\varphi_3^n(\xi)=\overline{l}_1(\xi)$.

\item $\lambda({\bf p})\ge \lambda({\bf q})$, $\lambda({\bf
q})<1$. Let $\tilde l_1(j)=\left(N-\frac{j}{d}\right)
\left(\frac{p_2}{2}-1\right)$, $\overline{m}^1_{j,l}=2^{Nd+ld-j}$.
Then the condition $n>(\overline{m}^1_{j,l})^{\frac{2}{p_2}}
k^1_{j,l}$ is equivalent to the inequality $l< \tilde l_1(j)$.
Hence, by Theorem \ref{fin_dim_mix_norm},
$$
d_{n,j,l}\lesssim n^{-\frac{\delta}{d}-\frac{\lambda({\bf
p})}{2}+\frac{\lambda({\bf p})}{p_2}}\cdot
2^{j\left(-\alpha+\frac{\delta}{d}+\frac{\lambda({\bf
p})}{2}-\frac{\lambda({\bf p})}{p_2}+
\frac{1}{q_2}-\frac{1}{q_1}\right)+l\left(-\delta
+\frac{\lambda({\bf p})d}{p_2}\right)}\times$$$$\times
2^{\frac{\sigma^1_n(j)\lambda({\bf p})}{2}+\frac{\tau^1_n(l, \,
j)\lambda({\bf p})}{2}+c_N(j)\left(\delta-\frac{\lambda({\bf
p})d}{p_2}\right)},
$$
if $0\le l<\tilde{l}_1(j)$,
$$
d_{n,j,l}\overset{n,j,l}{\lesssim}n^{-\frac{\delta}{d}-
\frac{\lambda({\bf q})}{2}+\frac{\lambda({\bf q})}{p_2}}\cdot
2^{j\left(-\alpha+\frac{\delta}{d}-\frac{\lambda({\bf q})}{p_2}+
\frac{\lambda({\bf q})}{q_2}\right) +l\left(-\delta
+\frac{\lambda({\bf q})d}{p_2}\right)}\times$$$$\times
2^{\frac{\sigma^1_n(j)\lambda({\bf q})}{2}+\frac{\tau ^1_n(l, \,
j)\lambda({\bf q})}{2}+c_N(j)\left(\delta-\frac{\lambda({\bf
q})d}{p_2}\right)},
$$
if $\tilde{l}_1(j)\le l<\overline{l}_1(j)$,
$$
d_{n,j,l}\overset{n,j,l}{\lesssim}n^{-\frac{\delta}{d}}2^{-l\delta}
2^{j(-\alpha+\frac{\delta}{d})+c_N(j)\delta},
$$
if $l\ge \overline{l}_1(j)$.

Put
$$
F^1_n(j, \, l)= n^{-\frac{\delta}{d}- \frac{\lambda({\bf p})}{2}
+\frac{\lambda({\bf p})}{p_2}}\cdot
2^{j\left(-\alpha+\frac{\delta}{d}+\frac{\lambda({\bf
p})}{2}-\frac{\lambda({\bf p})}{p_2}+
\frac{1}{q_2}-\frac{1}{q_1}\right)+l\left(-\delta
+\frac{\lambda({\bf p})d}{p_2}\right)},
$$
if $0\le l<\tilde l_1(j)$,
$$
F^1_n(j, \, l) =n^{-\frac{\delta}{d}-\frac{\lambda({\bf
q})}{2}+\frac{\lambda({\bf q})}{p_2}}\cdot
2^{j\left(-\alpha+\frac{\delta}{d}-\frac{\lambda({\bf
q})}{p_2}+\frac{\lambda({\bf q})}{q_2}\right) +l\left(-\delta
+\frac{\lambda({\bf q})d}{p_2}\right)}, \text{ if }\tilde
l_1(j)\le l<\overline{l}_1(j),
$$
$$
F^1_n(j, \, l)=n^{-\frac{\delta}{d}}2^{-l\delta}
2^{j(-\alpha+\frac{\delta}{d})}, \text{ if }
l\ge\overline{l}_1(j).
$$
Note that $\tilde{l}_1(0)=\overline{l}_1(0)$ and
$\tilde{l}_1(Nd)=0$. Choose
$$
(j_*^{n,1}, \, l_*^{n,1})\in {\rm Argmax}\, \left\{F_n^1(0, \, 0),
\, F_n^1(Nd, \, 0), \, F_n^1(0, \,\overline{l}_1(0)), \,
F_n^1(Nd,\, \overline{l}_1(Nd))\right\}.
$$

By Lemma \ref{osn_lemma}, for enough small $\varepsilon$ we have
\begin{align}
\label{summuj3} \displaystyle
\begin{array}{c}\sum \limits _{j=0}^{Nd}
\sum\limits _{l\in L_j}n^{-\delta
/d}2^{-l\delta}2^{j(-\alpha+\delta/d)}
\rho\left(2^j\right)2^{c_N(j)\delta}d_{\mu
_{j,l}^1}(B_{p_1,q_1}^{m^1_{j,l}, k^1_{j,l}}, \,
l_{p_2,q_2}^{m^1_{j,l}, k ^1_{j,l}})\overset{n}{\lesssim}
\\
\lesssim\max\left\{ n^{-\frac{\delta}{d}-\frac{\lambda({\bf
p})}{2}+ \frac{\lambda({\bf p})}{p_2}}, \,
n^{-\alpha+\frac{1}{q_2}-\frac{1}{q_1}}\rho(n), \,
n^{-\frac{p_2\delta}{2d}}, \, n^{-\alpha+\frac{\delta
p_2}{q_2\cdot d}-\frac{\delta p_2}{2d}}\rho(n), \,
n^{-\frac{\alpha
q_2}{2}}\rho(n^{\frac{q_2}{2}})\right\}\overset{n}{\lesssim}\\
\lesssim\max\left\{ n^{-\frac{\delta}{d}-\frac{\lambda({\bf
p})}{2}+ \frac{\lambda({\bf p})}{p_2}}, \,
n^{-\alpha+\frac{1}{q_2}-\frac{1}{q_1}}\rho(n), \,
n^{-\frac{p_2\delta}{2d}}, \, n^{-\frac{\alpha
q_2}{2}}\rho(n^{\frac{q_2}{2}})\right\}.
\end{array}
\end{align}
\end{itemize}
Here $\varphi_1^n(\xi)=0$, $\varphi_2^n(\xi)=\tilde l_1(\xi)$,
$\varphi_3^n(\xi)=\overline{l}_1(\xi)$.

{\bf Substep 5.2.} To estimate the second term of the sum $\Sigma
_0$, we apply Theorem \ref{fin_dim_mix_norm} again. Note that
\begin{align}
\label{l2j00} \overline{l}_2(j_0(N))=0.
\end{align}
\begin{itemize}
\item $p_1\le 2$, $q_1\le 2$. Then, by (\ref{mjl1_kjl2}),
$$
n^{-\alpha}2^{-l\delta} 2^{-j\alpha} d_{\mu
^2_{j,l}}(B_{p_1,q_1}^{m^2_{j,l}, k^2_{j,l}}, \,
l_{p_2,q_2}^{m_{j,l}^2, k^2_{j,l}})\overset{n,j,l}{\lesssim}
$$
$$
\lesssim \left\{ \begin{array}{l} n^{-\alpha+ \frac{1}{q_2}-\frac
12}\cdot 2^{l\left(-\delta+\frac{d}{p_2}\right)}
2^{j\left(-\alpha+\frac{1}{q_2}\right)}\cdot
2^{\frac{\sigma^2_n(j)}{2}+\frac{\tau ^2_n(l, \, j)}{2}},
\text{ if }l\le \overline{l}_2(j), \\
n^{-\alpha}2^{-l\delta} 2^{-j\alpha}, \text{ if
}l>\overline{l}_2(j).\end{array}\right .
$$
Put
$$
F^2_n(j, \, l)=\left\{ \begin{array}{l} n^{-\alpha+
\frac{1}{q_2}-\frac 12}2^{l\left(-\delta+\frac{d}{p_2}\right)}
2^{j\left(-\alpha+\frac{1}{q_2}\right)}, \text{ if }l\le
\overline{l}_2(j), \\ n^{-\alpha}2^{-l\delta} 2^{-j\alpha}, \text{
if }l>\overline{l}_2(j).\end{array}\right.
$$

Choose
$$
(j_*^{n,2}, \, l_*^{n,2})\in {\rm Argmax}\, \left\{F_n^2(0, \, 0),
\, F_n^2(0, \,\overline{l}_2(0)), \, F_n^2(j_0(N),\, 0)\right\}.
$$

Apply Lemma \ref{osn_lemma} for $\varphi^n_1(\xi)=0$,
$\varphi^n_2(\xi)=\overline{l}_2(\xi)$, taking into account
(\ref{l2j00}). If $\varepsilon$ is sufficiently small, then
\begin{align}
\label{2in1} \displaystyle
\begin{array}{c}
\Sigma_0'':=\sum \limits _{0\le j\le j_0(N)}\sum \limits _{l\in
\Z_+} n^{-\alpha}2^{-l\delta} 2^{-j\alpha}\rho\left(2^jn\right)
d_{\mu ^2_{j,l}}(B_{p_1,q_1}^{m^2_{j,l}, k^2_{j,l}}, \,
l_{p_2,q_2}^{m_{j,l}^2, k ^2_{j,l}})\overset{n}{\lesssim}
\\
\lesssim\max \left\{n^{-\frac{\delta}{d}+\frac{1}{p_2}-\frac 12},
\, n^{-\frac{p_2\delta}{2d}}, \, n^{-\alpha-\frac
12+\frac{1}{q_2}}\rho(n), \, n^{-\frac{p_2\delta}{2d}-\alpha
+\frac{\delta p_2}{q_2d}}\rho(n), \, n^{-\frac{\alpha
q_2}{2}}\rho(n^{\frac{q_2}{2}})\right\}\overset{n}{\lesssim}\\
\overset{n}{\lesssim}\max
\left\{n^{-\frac{\delta}{d}+\frac{1}{p_2}-\frac 12}, \,
n^{-\frac{p_2\delta}{2d}}, \, n^{-\alpha-\frac
12+\frac{1}{q_2}}\rho(n), \, n^{-\frac{\alpha
q_2}{2}}\rho(n^{\frac{q_2}{2}})\right\}.
\end{array}
\end{align}
Indeed, consider the case $p_2>2$, $q_2>2$ (in other cases of the
assertion can be proved in a similar way). By conditions of the
theorem, we have $$\min \left\{\frac{\delta}{d}+\frac
12-\frac{1}{p_2}, \, \frac{p_2\delta}{2d}\right\}\ne\min
\left\{\alpha+\frac 12-\frac{1}{q_2}, \, \frac{\alpha
q_2}{2}\right\}.$$ If $\min \left\{\frac{\delta}{d}+\frac
12-\frac{1}{p_2}, \, \frac{p_2\delta}{2d}\right\}<\min
\left\{\alpha+\frac 12-\frac{1}{q_2}, \, \frac{\alpha
q_2}{2}\right\}$, then by (\ref{p2d})
$$
\frac{p_2\delta}{2d}+\alpha -\frac{\delta p_2}{q_2d}>\min
\left\{\frac{\delta}{d}+\frac 12-\frac{1}{p_2}, \,
\frac{p_2\delta}{2d}\right\}
$$
and $$\Sigma_0''\overset{n}{\lesssim}n^{-\min
\left\{\frac{\delta}{d}+\frac 12-\frac{1}{p_2}, \,
\frac{p_2\delta}{2d}\right\}}.$$ Here the set $S$ from Lemma
\ref{osn_lemma} is empty.

If $\min \left\{\frac{\delta}{d}+\frac 12-\frac{1}{p_2}, \,
\frac{p_2\delta}{2d}\right\}>\min \left\{\alpha+\frac
12-\frac{1}{q_2}, \, \frac{\alpha q_2}{2}\right\}$, then by
(\ref{p2d})
$$
\frac{p_2\delta}{2d}+\alpha -\frac{\delta p_2}{q_2d}> \min
\left\{\alpha+\frac 12-\frac{1}{q_2}, \, \frac{\alpha
q_2}{2}\right\},
$$
the right-hand side attains its minimum at the unique point and
$$
\Sigma _0'' \overset{n}{\lesssim}\max \left\{ n^{-\alpha-\frac
12+\frac{1}{q_2}}\rho(n), \, n^{-\frac{\alpha
q_2}{2}}\rho(n^{\frac{q_2}{2}})\right\}.
$$
Here the set $S$ from Lemma \ref{osn_lemma} is nonempty.

\item $\lambda({\bf p})\le \lambda({\bf q})$, $\lambda({\bf
p})<1$. Put $\tilde
l_2(j)=N\left(1-\frac{2}{q_2}\right)-\frac{2j}{q_2d}$. Then the
condition $n>m^2_{j,l}(k^2_{j,l})^{2/q_2}$ is equivalent to the
inequality $l<\tilde l_2(j)$ and
$$
n^{-\alpha}2^{-l\delta} 2^{-j\alpha} d_{\mu ^2_{j,l}}
(B_{p_1,q_1}^{m^2_{j,l}, k^2_{j,l}}, \, l_{p_2,q_2}^{m_{j,l}^2,
k^2_{j,l}})\overset{n,j,l}{\lesssim}
$$
$$
\lesssim \left\{ \begin{array}{l} n^{-\alpha-\frac{\lambda({\bf
q})}{2} + \frac{\lambda({\bf q})}{q_2}}\cdot
2^{j\left(-\alpha+\frac{\lambda({\bf
q})}{q_2}\right)+l\left(-\delta+\frac{d}{p_2}-\frac{d}{p_1}+\frac{\lambda({\bf
q})d}{2}\right)}2^{\frac{\lambda({\bf q})\sigma
^2_n(j)}{2}+\frac{\lambda({\bf q})\tau ^2_n(l,\, j)}{2}}, \text{
if
}l< \tilde{l}_2(j), \\
n^{-\alpha -\frac{\lambda({\bf p})}{2}+\frac{\lambda({\bf
p})}{q_2}}\cdot 2^{j\left(-\alpha+\frac{\lambda({\bf
p})}{q_2}\right)+l\left(-\delta+\frac{\lambda({\bf
p})d}{p_2}\right)}2^{\frac{\lambda({\bf p})\sigma
^2_n(j)}{2}+\frac{\lambda({\bf p})\tau ^2_n(l,\, j)}{2}}, \text{
if
}\tilde l_2(j)\le l<\overline{l}_2(j), \\
n^{-\alpha}2^{-l\delta} 2^{-j\alpha}, \text{ if }l\ge
\overline{l}_2(j).
\end{array}\right .
$$
Let
$$
F^2_n(j, \, l)=\left\{
\begin{array}{l} n^{-\alpha-\frac{\lambda({\bf q})}{2}
+\frac{\lambda({\bf q})}{q_2}}\cdot
2^{j\left(-\alpha+\frac{\lambda({\bf q})}{q_2}\right)
+l\left(-\delta+\frac{d}{p_2}-\frac{d}{p_1}+ \frac{\lambda({\bf
q})d}{2}\right)}, \text{ if }0\le l< \tilde l_2(j), \\
n^{-\alpha -\frac{\lambda({\bf p})}{2}+\frac{\lambda({\bf
p})}{q_2}}\cdot 2^{j\left(-\alpha+\frac{\lambda({\bf
p})}{q_2}\right)+l\left(-\delta+\frac{\lambda({\bf
p})d}{p_2}\right)}, \text{ if }\tilde l_2(j)< l\le
\overline{l}_2(j), \\
n^{-\alpha}2^{-l\delta} 2^{-j\alpha},\text{ if } l\ge
\overline{l}_2(j).
\end{array}\right .
$$

Choose
$$
(j_*^{n,2}, \, l_*^{n,2})\in {\rm Argmax}\, \left\{F_n^2(0, \, 0),
\, F_n^2(0, \,\overline{l}_2(0)), \, F_n^2(0, \,\tilde{l}_2(0)),
\, F_n^2(j_0(N),\, 0)\right\}.
$$

By Lemma \ref{osn_lemma} (taking into account (\ref{l2j00}) and
the equality $\tilde l_2(j_0(N))=0$), for enough small
$\varepsilon$ we obtain
\begin{align}
\label{2in2} \displaystyle
\begin{array}{c}
\sum \limits _{0\le j\le j_0(N)}\sum \limits _{l\in
\Z_+}n^{-\alpha}2^{-l\delta} 2^{-j\alpha}\rho\left(2^jn\right)
d_{\mu ^2_{j,l}}(B_{p_1,q_1}^{m^2_{j,l}, k^2_{j,l}}, \,
l_{p_2,q_2}^{m_{j,l}^2, k ^2_{j,l}})\overset{n}{\lesssim}
\\
\lesssim\max\left\{n^{-\frac{\delta}{d}-
\frac{1}{p_1}+\frac{1}{p_2}}, \, n^{-\alpha+\frac{\lambda({\bf
q})}{q_2}- \frac{\lambda({\bf
q})}{2}}\rho(n), \, n^{-\frac{p_2\delta}{2d}}, \right .\\
\left . n^{-\frac{\delta p_2}{2d}-\alpha+\frac{\delta
p_2}{q_2\cdot d}}\rho(n), \, n^{-\alpha
+\left(-\frac{\delta}{d}-\frac{1}{p_1}+\frac{1}{p_2}\right)
\left(1-\frac{2}{q_2}\right)}\rho(n), \, n^{-\frac{\alpha
q_2}{2}}\rho(n^{\frac{q_2}{2}})\right\}\overset{n}{\lesssim}\\
\lesssim\max\left\{n^{-\frac{\delta}{d}-
\frac{1}{p_1}+\frac{1}{p_2}}, \, n^{-\alpha+\frac{\lambda({\bf
q})}{q_2}- \frac{\lambda({\bf q})}{2}}\rho(n), \,
n^{-\frac{p_2\delta}{2d}}, \, n^{-\frac{\alpha
q_2}{2}}\rho(n^{\frac{q_2}{2}})\right\}.
\end{array}
\end{align}

\item $\lambda({\bf p})\ge \lambda({\bf q})$, $\lambda({\bf
q})<1$. Then $n\le (m^2_{j,l})^{\frac{2}{p_2}}k^2_{j,l}$. Hence,
$$
n^{-\alpha}2^{-l\delta} 2^{-j\alpha} d_{\mu
^2_{j,l}}(B_{p_1,q_1}^{m^2_{j,l}, k^2_{j,l}}, \,
l_{p_2,q_2}^{m_{j,l}^2, k^2_{j,l}})\overset{n,j,l}{\lesssim}
$$
$$
\lesssim \left\{
\begin{array}{l} n^{-\alpha-\frac{1}{q_1}
+\frac{1}{q_2}}2^{l\left(-\delta+ \frac{\lambda({\bf
q})d}{p_2}\right)} 2^{j\left(-\alpha+\frac{\lambda({\bf
q})}{q_2}\right)}\cdot 2^{\frac{\lambda({\bf
q})\sigma^2_n(j)}{2}+\frac{\lambda({\bf q})\tau^2_n(l, \, j)}{2}},
\text{ if }l<\overline{l}_2(j), \\
n^{-\alpha}2^{-l\delta} 2^{-j\alpha},\text{ if }l\ge
\overline{l}_2(j).\end{array}\right .
$$
Put
$$
F^2_n(j, \, l)=\left\{
\begin{array}{l} n^{-\alpha-\frac{1}{q_1}
+\frac{1}{q_2}}2^{l\left(-\delta+ \frac{\lambda({\bf
q})d}{p_2}\right)} 2^{j\left(-\alpha+\frac{\lambda({\bf
q})}{q_2}\right)},
\text{ if }l<\overline{l}_2(j), \\
n^{-\alpha}2^{-l\delta} 2^{-j\alpha},\text{ if }l\ge
\overline{l}_2(j)\end{array}\right .
$$
and choose
$$
(j_*^{n,2}, \, l_*^{n,2})\in {\rm Argmax}\, \left\{F_n^2(0, \, 0),
\, F_n^2(0, \,\overline{l}_2(0)), \, F_n^2(j_0(N),\, 0)\right\}.
$$

By Lemma \ref{osn_lemma}, for enough small $\varepsilon$ we obtain
\begin{align}
\label{2in3} \displaystyle
\begin{array}{c}
\sum \limits _{0\le j\le j_0(N)}\sum \limits _{l\in
\Z_+}n^{-\alpha}2^{-l\delta} 2^{-j\alpha}\rho\left(2^jn\right)
d_{\mu ^2_{j,l}}(B_{p_1,q_1}^{m^2_{j,l}, k^2_{j,l}}, \,
l_{p_2,q_2}^{m_{j,l}^2, k ^2_{j,l}})\overset{n}{\lesssim}
\\
\lesssim\max\left\{ n^{-\frac{\delta}{d}-\frac{\lambda({\bf
p})}{2}+ \frac{\lambda({\bf p})}{p_2}}, \,
n^{-\alpha+\frac{1}{q_2}-\frac{1}{q_1}}\rho(n), \,
n^{-\frac{p_2\delta}{2d}}, \, n^{-\alpha+\frac{\delta p_2}{q_2
d}-\frac{\delta p_2}{2d}}\rho(n), \, n^{-\frac{\alpha
q_2}{2}}\rho(n^{\frac{q_2}{2}})\right\}\overset{n}{\lesssim}\\
\lesssim\max\left\{ n^{-\frac{\delta}{d}-\frac{\lambda({\bf
p})}{2}+ \frac{\lambda({\bf p})}{p_2}}, \,
n^{-\alpha+\frac{1}{q_2}-\frac{1}{q_1}}\rho(n), \,
n^{-\frac{p_2\delta}{2d}}, \, n^{-\frac{\alpha
q_2}{2}}\rho(n^{\frac{q_2}{2}})\right\}.
\end{array}
\end{align}
\end{itemize}

{\bf Substep 5.3.} Now we estimate the last term of the sum
$\Sigma _0$.

Put $j_*^n=0$, if $\alpha \ge \frac{\lambda({\bf q})}{q_2}$,
$j_*^n=j_0(N)$, if $\alpha <\frac{\lambda({\bf q})}{q_2}$,
$\sigma^5_n(j)=\varepsilon|j-j_*^n|$, $\mu ^5_j=n2^{-[\sigma
^5_n(j)]}$. Then $\sum \limits _{j=0}^{j_0}\mu
^5_j\overset{n}{\lesssim}n$. Theorem \ref{teor_glus} together with
(\ref{mj6kj6}) yield
$$
n^{-\alpha}\sum \limits _{0\le j\le
j_0(N)}2^{-j\alpha}\rho(2^jn)d_{\mu ^5_j}(B_{q_1}^{k_j^5}, \,
l_{q_2}^{k_j^5})\overset{n}{\lesssim} n^{-\alpha}\sum \limits
_{0\le j\le j_0(N)}2^{-j\alpha}\rho(2^jn)n^{-\frac{\lambda({\bf
q})}{2}}\cdot 2^{\frac{\lambda({\bf q})\sigma
^5_n(j)}{2}}2^{(j+Nd)\frac{\lambda({\bf q})}{q_2}}=
$$
$$
=n^{-\alpha +\frac{\lambda({\bf q})}{q_2}-\frac{\lambda({\bf
q})}{2}}\sum \limits _{0\le j\le
j_0(N)}2^{j\left(-\alpha+\frac{\lambda({\bf
q})}{q_2}\right)}\rho(2^jn)\cdot 2^{\frac{\lambda({\bf q})\sigma
^5_n(j)}{2}}=:\Sigma _*.
$$
For sufficiently small $\varepsilon >0$ we get
\begin{align}
\label{sz} \Sigma_*\overset{n}{\lesssim} \left \{
\begin{array}{l}
n^{-\alpha +\frac{\lambda({\bf q})}{q_2}-\frac{\lambda({\bf
q})}{2}}, \text{ if } \alpha>\frac{\lambda({\bf q})}{q_2},
\\ n^{-\alpha q_2/2}\rho(n^{q_2/2}), \text{ if }\alpha <
\frac{\lambda({\bf q})}{q_2}.
\end{array}
\right.
\end{align}
If $\alpha =\frac{\lambda({\bf q})}{q_2}$, then by conditions of
the theorem  we have $\frac{\alpha
q_2}{2}=\alpha+\frac{\lambda({\bf q})}{2}-\frac{\lambda({\bf
q})}{q_2}>\min\left\{\frac{\delta}{d} +\frac{\lambda({\bf
p})}{2}-\frac{\lambda({\bf p})}{p_2}, \,
\frac{p_2\delta}{2d}\right\}$. Therefore,
\begin{align}
\label{sz1} \Sigma_*\overset{n}{\lesssim}
n^{-\min\left\{\frac{\delta}{d} +\frac{\lambda({\bf
p})}{2}-\frac{\lambda({\bf p})}{p_2}, \,
\frac{p_2\delta}{2d}\right\}}.
\end{align}

Combining (\ref{summuj1}), (\ref{summuj2}), (\ref{summuj3}),
(\ref{2in1}), (\ref{2in2}), (\ref{2in3}), (\ref{sz}) and
(\ref{sz1}), we obtain the desired upper estimate.

{\bf Step 6.} Let us estimate $\Sigma _1$ for $1<p_1\le 2\le
p_2<\infty$ and $1<q_1\le 2\le q_2<\infty$. By Theorem
\ref{lin_fin},
$$
\lambda _n(B^{m,k}_{p_1,q_1}, \,
l^{m,k}_{p_2,q_2})\overset{m,k,n}{\lesssim}\left\{
\begin{array}{l}
\Phi _0(m,\, k, \, n, \, p_1, \, p_2, \, q_1, \, q_2), \;\;
\frac{1}{p_1}+\frac{1}{p_2}\ge 1, \;
\frac{1}{q_1}+\frac{1}{q_2}\ge 1,\\
\Phi _0(m,\, k, \, n, \, p_2', \, p_1', \, q_1, \, q_2), \;\;
\frac{1}{p_1}+\frac{1}{p_2}\le 1, \;
\frac{1}{q_1}+\frac{1}{q_2}\ge 1,\\
\Phi _0(m,\, k, \, n, \, p_1, \, p_2, \, q_2', \, q_1'),\;\;
\frac{1}{p_1}+\frac{1}{p_2}\ge 1, \;
\frac{1}{q_1}+\frac{1}{q_2}\le 1,\\
\Phi _0(m,\, k, \, n, \, p_2', \, p_1', \, q_2', \, q_1'),\;\;
\frac{1}{p_1}+\frac{1}{p_2}\le 1, \;
\frac{1}{q_1}+\frac{1}{q_2}\le 1.
\end{array}\right .
$$
It remains to use (\ref{gluskin_lin}) and the upper estimate for
$\Sigma _0$ which is already proved.

{\it Proof of the lower estimate.} Let $p_2\ge 2$, $q_2\ge 2$.
Take $\varepsilon=0$. By (\ref{norm_bspq_proj_m1}), (\ref{proj}),
(\ref{norm}), (\ref{por_m_mnuk}), (\ref{card_mnuk}),
(\ref{por_m_mnuk2}) and (\ref{card_mnuk2}), there exist
$\nu_n\overset{n}{\asymp} n$ and $\mu_n\overset{n}{\asymp} n$ such
that $\mu_n\le \frac{\nu_n}{2}$ and
$$
d_{\mu_n}(P_{{\cal N}}:X_1\rightarrow
X_2)\overset{n}{\gtrsim}\max\left\{d_{\mu_n}(P_{{\cal
N}^1_{0,0}}:X_1 \rightarrow X_2),\right.
$$
$$
\left. d_{\mu_n}(P_{{\cal
N}^1_{0,\overline{l}_1(0)}}:X_1\rightarrow X_2), \,
d_{\mu_n}(P_{{\cal N}^1_{Nd,0}}:X_1\rightarrow X_2), \,
d_{\mu_n}(P_{{\cal N}^2_{j_0(N),0}}:X_1\rightarrow
X_2)\right\}\overset{n}{\gtrsim}$$
$$
\gtrsim\max\left\{n^{-\frac{\delta}{d}}d_{\mu_n}(B^{\nu_n,1}_{p_1,q_1},
\, l^{\nu_n,1}_{p_2,q_2}), \, n^{-\frac{\delta}{d}}\cdot
2^{N\left(\frac{p_2}{2}-1\right)\delta}, \right .
$$
$$
\left. n^{-\alpha}\rho(n)d_{\mu_n}(B^{1,\nu_n}_{p_1,q_1}, \,
l^{1,\nu_n}_{p_2,q_2}), \, n^{-\alpha}\cdot
2^{-j_0(N)\alpha}\rho(2^{j_0(N)}n)\right\}=
$$
$$
=\max\left\{n^{-\frac{\delta}{d}}d_{\mu_n}(B^{\nu_n}_{p_1}, \,
l^{\nu_n}_{p_2}), \, n^{-\frac{p_2\delta}{2d}}, \,
n^{-\alpha}\rho(n)d_{\mu_n}(B^{\nu_n}_{q_1}, \, l^{\nu_n}_{q_2}),
\, n^{-\alpha q_2/2}\rho(n^{q_2/2})\right\}.
$$
Similarly, if $1<p_1\le 2\le p_2<\infty$, $1<q_1\le 2\le
q_2<\infty$, then
$$
{\cal A}_{\mu_n}(P_{{\cal N}}:X_1\rightarrow
X_2)\overset{n}{\gtrsim}
\max\left\{n^{-\frac{\delta}{d}}\lambda_{\mu_n}(B^{\nu_n}_{p_1},\,
l^{\nu_n}_{p_2}), \, n^{-\frac{\min(p_2, \, p_1')\delta}{2d}},
\right.$$$$\left.
n^{-\alpha}\rho(n)\lambda_{\mu_n}(B^{\nu_n}_{q_1}, \,
l^{\nu_n}_{q_2}), \, n^{-\frac{\alpha \min (q_2, \,
q_1')}{2}}\rho(n^{\min(q_2, \, q_1')/2})\right\}
$$
(we use (\ref{proj1}) instead of (\ref{proj})). It remains to
apply Theorem \ref{teor_glus}.

If $p_2\le 2$ and $q_2\le 2$, then we similarly get
$$
{\cal A}_{\mu_n}(P_{{\cal N}}:X_1\rightarrow X_2)\ge
d_{\mu_n}(P_{{\cal N}}:X_1\rightarrow
X_2)\overset{n}{\gtrsim}$$$$\gtrsim\max\left\{n^{-\frac{\delta}{d}}d_{\mu_n}(B^{\nu_n}_{p_1},
\, l^{\nu_n}_{p_2}), \,
n^{-\alpha}\rho(n)d_{\mu_n}(B^{\nu_n}_{q_1}, \,
l^{\nu_n}_{q_2})\right\}\overset{n}{\asymp}\max\left\{n^{-\frac{\delta}{d}},
\, n^{-\alpha}\rho(n)\right\}.
$$
The same lower estimate holds for ${\cal A}_{\mu_n}(P_{{\cal
N}}:X_1\rightarrow X_2)$, if $p_1\ge 2$, $q_1\ge 2$.
\end{proof}

{\bf Acknowledgements}

In conclusion, the author expresses her sincere gratitude to I.G.
Tsar’kov, A.S. Kochurov, V.D. Stepanov and E.M. Galeev for the
discussion of results and reading the manuscript. The research was
supported by RFBR under grant no. 10-01-00442.

\end{document}